\documentclass[letterpaper]{amsart}

\usepackage[alphabetic,initials,nobysame]{amsrefs}
\usepackage{amssymb, mathtools, mathrsfs, enumitem, color, comment}
\usepackage{pgfplots}\pgfplotsset{compat=1.18}
\usetikzlibrary{decorations.markings}
\usepackage[bookmarksdepth=2]{hyperref} 

\BibSpec{collection.article}{%
	+{}  {\PrintAuthors}				{author}
	+{,} { \textit}                  	{title}
	+{.} { }                         	{part}
	+{:} { \textit}                  	{subtitle}
	+{,} { \PrintContributions}      	{contribution}
	+{,} { \PrintConference}         	{conference}
	+{}  {\PrintBook}                	{book}
	+{,} { }                         	{booktitle}
	+{,} { }						 	{series}
	+{,} { }						 	{publisher}
	+{,} { }						 	{address}
	+{,} { \PrintDateB}              	{date}
	+{,} { pp.~}                     	{pages}
	+{,} { }                         	{status}
	+{,} { \PrintDOI}                	{doi}
	+{,} { available at \eprint}     	{eprint}
	+{}  { \parenthesize}            	{language}
	+{}  { \PrintTranslation}        	{translation}
	+{;} { \PrintReprint}            	{reprint}
	+{.} { }                         	{note}
	+{.} {}                          	{transition}
	+{}  {\SentenceSpace \PrintReviews} {review}
}

\theoremstyle{plain}
\newtheorem{theorem}{Theorem}
\newtheorem{lemma}[theorem]{Lemma}
\newtheorem{proposition}[theorem]{Proposition}
\newtheorem{corollary}[theorem]{Corollary}

\theoremstyle{definition}

\theoremstyle{remark}
\newtheorem{remark}[theorem]{Remark}
\newtheorem{example}[theorem]{Example}

\numberwithin{equation}{section}
\numberwithin{theorem}{section}

\renewcommand{\bar}{\overline}
\newcommand{\R}{\mathbb R}

\newcommand{\C}{\mathbb C}
\newcommand{\D}{\mathbb D}
\newcommand{\Z}{\mathbb Z}
\newcommand{\N}{\mathbb N}

\newcommand{\UHP}{\mathbb H}

\DeclareMathOperator{\dist}{dist}
\DeclareMathOperator{\diam}{diam}
\DeclareMathOperator{\id}{id}

\DeclareMathOperator{\re}{Re}
\DeclareMathOperator{\im}{Im}

\DeclareMathOperator{\Mod}{Mod}

\DeclareMathOperator{\graph}{Graph}
\DeclareMathOperator{\loc}{loc}

\DeclareMathOperator{\arcsinh}{arcsinh}

\title{Characterization of tangent quasicircles and quasiannuli}
\author{Dimitrios Ntalampekos}
\address{Department of Mathematics, Aristotle University of Thessaloniki, Thessaloniki, 54152, Greece.}
\thanks{The author is partially supported by the ERC Starting Grant, Grant Agreement no.\ 101214615, GRComPaS}
\email{dntalam@math.auth.gr}
\date{\today}

\begin{document}

	\begin{abstract}
	We give a necessary and sufficient condition so that a pair of disjoint Jordan regions in the sphere can be quasiconformally mapped to a pair of disks. As a consequence, we obtain a simple characterization that involves Lipschitz functions for the case that one of the Jordan regions is a half-plane. We apply these results to prove that all polynomial cusps are quasiconformally equivalent and that a quasisymmetric embedding of the union of two disjoint disks  extends to a quasiconformal map of the sphere, quantitatively. Also, in combination with previous work of the author, we obtain a new characterization of compact sets that are quasiconformally equivalent to Schottky sets.
	\end{abstract}

\maketitle

\setcounter{tocdepth}{1}
\tableofcontents

\section{Introduction}

\subsection{Motivation}

In this paper we study the question whether two disjoint Jordan regions in the sphere $\widehat \C=\C\cup\{\infty\}$ can be mapped to a pair of disjoint disks with a quasiconformal homeomorphism of $\widehat \C$. The closures of the regions are allowed to intersect each other. Our goal is to provide a characterization of such regions that is quantitative in nature. Namely the quasiconformal dilatation should depend only on the geometric features of the Jordan regions and not on their relative position. 

Ahlfors \cite{Ahlfors:reflections} provided an excellent and deep characterization of curves that can be mapped to a circle with a quasiconformal homeomorphism of the plane. Specifically, a Jordan curve $J$ in the plane can be quasiconformally mapped to a circle if and only if there exists $L\geq 1$ such that for every pair of points $z,w\in J$, if $E$ is the arc between $z$ and $w$ with the smallest diameter, then
$$\diam (E) \leq L|z-w|.$$
Such a Jordan curve is called an $L$-quasicircle. The region bounded by a quasicircle is called a quasidisk. Geometrically, quasicircles do not have cusps. Moreover, in the theorem of Ahlfors the quasiconformal dilatation and the constant $L$ are quantitatively related. See Section \ref{section:quasidisks} for further details and definitions. Therefore, if we wish to map quasiconformally a pair of disjoint Jordan regions $U,V$ to a pair of disjoint disks, then the boundaries $\partial U,\partial V$ must be quasicircles. 

It is easy to see that this condition is not sufficient. For example, note that the boundary of each square is a $2$-quasicircle. However, we cannot map a pair of disjoint open squares $U,V$ with disjoint closures to a pair of disjoint disks with a $K$-quasiconformal map, where $K$ is bounded and does not depend on the relative position of the squares. Specifically, it is observed by the author in \cite{Ntalampekos:CarpetsThesis}*{pp.~174--176} that as the relative distance of the squares, defined by 
$$\Delta(U,V)=\frac{\dist(U,V)}{\min\{\diam (U),\diam (V)\}},$$
shrinks to $0$, the value of the quasiconformal dilatation $K$ must blow up to $\infty$. Therefore, it is not sufficient to assume that $\partial U,\partial V$ are uniform quasicircles, but in this case the relative distance also affects the quasiconformal dilatation.

Herron \cite{Herron:uniform} provided a sufficiency criterion for the quasiconformal transformation of two disjoint Jordan regions $U,V$ into two disks with quantitative control. Namely, if $\partial U,\partial V$ are $L$-quasicircles and $\Delta(\partial U,\partial V)\geq \delta>0$, then there exists a $K(L,\delta)$-quasiconformal map with the claimed properties. Note that the converse is not true. Indeed, the condition $\Delta(\partial U,\partial V)\geq \delta$ is quasiconformally quasi-invariant (see Lemma \ref{lemma:quasimobius:cross}) and two disks of radius $1$ that are very close to each other, or even tangent to each other, do not satisfy it.

In a subsequent work Herron and Koskela \cite{HerronKoskela:QEDcircledomains} proved that if $\Omega$ is a uniform domain (see Section \ref{section:uniform}), then it can be mapped with a quasiconformal map of the sphere to a \textit{circle domain}, i.e., a domain whose boundary components are points or circles. The boundary components of a uniform domain $\Omega$ can be infinite in number, but they are all points or uniform quasicircles that are \textit{uniformly relatively separated}; that is, their mutual relative distances are uniformly bounded from below. On the other hand, the definition of a uniform domain involves a technical requirement that is even stronger than these conditions. 

Two decades later, Bonk \cite{Bonk:uniformization} generalized the aforementioned results, relying on a powerful tool devised by Schramm \cite{Schramm:transboundary}, known as transboundary modulus. Bonk proved that if $\{U_i\}_{i\in I}$ is a collection of disjoint Jordan regions in the sphere such that $\partial U_i$ is an $L$-quasicircle for each $i\in I$ and $\Delta(\partial U_i,\partial U_j)\geq \delta>0$ for every $i\neq j$, then there exists a $K(L,\delta)$-quasiconformal map of the sphere that maps each $U_i$ to a disk.  This result implies the theorem of Herron and Koskela and shows that the existence of such a quasiconformal map should rely on the geometry of the regions $U_i$, $i\in I$, rather than on the intrinsic geometry of the set $\widehat \C\setminus \bigcup_{i\in I}\bar{U_i}$. 

Bonk's result was applied in the problem of quasiconformal uniformization of Sierpi\'nski carpets, in connection with the Kapovich--Kleiner conjecture from geometric group theory, in rigidity results related to Sierpi\'nski carpets \cites{BonkMerenkov:rigidity, BonkMerenkov:rigiditySpCarpets}, and in rigidity problems in complex dynamics \cite{BonkLyubichMerenkov:carpetJulia}. Moreover, it has been extended to non-planar carpets \cites{MerenkovWildrick:uniformization,Rehmert:thesis}.

Recent progress in uniformization problems in complex dynamics opened the way for uniformization results without the uniform relative separation assumption. Specifically, Luo and the author \cite{LuoNtalampekos:gasket} characterized, under some mild conditions, Julia sets of rational maps that can be quasiconformally mapped to round gaskets; these are sets whose complementary components are Jordan regions that can potentially touch each other at the boundary. In particular, there is no separation between the Jordan regions. The methods of \cite{LuoNtalampekos:gasket} were extended in \cite{LuoMjMukherjee:basilica}, where it is shown that certain Basilica Julia sets and limit sets of Kleinian groups can be quasiconformally mapped to round Basilicas. 

While these results are purely dynamical and do not extend to arbitrary gaskets, the author \cite{Ntalampekos:schottky} was able to obtain a general characterization of sets in the sphere that can be quasiconformally mapped to \textit{Schottky sets}; that is sets, whose complementary components are disks. The main result in \cite{Ntalampekos:schottky} (restated below in Theorem \ref{theorem:intro:schottky}) asserts that if $\{U_i\}_{i\in I}$ is a collection of disjoint Jordan regions in the sphere, then there exists a quasiconformal map of the sphere that maps each $U_i$ to a disk if and only if the collection $\{U_i\}_{i\in I}$ is \textit{uniformly quasiconformally pairwise circularizable}; that is, there exists $K\geq 1$ such that for every pair $U_i,U_j$, $i\neq j$, there exists a $K$-quasiconformal map of the sphere that maps $U_i$ and $U_j$ to disks. For instance, this condition is satisfied under the assumption $\Delta(\partial U_i,\partial U_j)\geq \delta$ thanks to the result of Herron \cite{Herron:uniform}. Therefore, Bonk's theorem follows as a consequence. 

Summarizing, the main result in \cite{Ntalampekos:schottky} allows the quasiconformal uniformization of a collection of disjoint Jordan regions by disks if every pair can be quasiconformally uniformized with controlled dilatation. Therefore, the last missing piece in this puzzle is to characterize pairs of disjoint Jordan regions that can be mapped quasiconformally to disks. This is precisely the main objective of the present work. 

\subsection{Main result}

Our main result provides a quantitative characterization of pairs of disjoint quasidisks that can be quasiconformally mapped to pairs of disks. There are two fundamental cases that we treat separately: either the boundaries of the two quasidisks intersect at a point or the quasidisks have disjoint closures. 

The main results are expressed in terms of the notion of the relative hyperbolic metric corresponding to a pair of quasidisks that we introduce in this paper.  If $U\subset \widehat\C$ is an open set we denote by $U^*$ the complement of $\bar U$. If $\partial U$ contains at least two points, we denote by $h_U$ the hyperbolic metric on $U$. 

Let $U,V\subset \widehat \C$ be disjoint Jordan regions such that $\bar U\cap \bar V$ is either empty or contains only one point; in particular, $\widehat \C\setminus (\bar U\cup \bar V)$ is connected. We define the \textit{relative hyperbolic metric} corresponding to the pair $(V,U)$ as
$$d_{V,U}(z,w)= \inf_\gamma\{\ell_{h_{U^*}}(\gamma)\},\quad z,w\in \partial V\setminus \bar U,$$
where the infimum is taken over rectifiable curves $\gamma\colon [a,b]\to U^*\setminus V$ such that $\gamma(a)=z$, $\gamma(b)=w$. It turns out that if $V$ is a quasidisk, then $d_{V,U}$ is a metric on $\partial V\setminus \bar U$. See Section \ref{section:relative:definition} for details. This metric can be explicitly computed when $U,V$ are disjoint half-planes or their boundaries are concentric circles; see Lemma \ref{lemma:parallel}. In Section \ref{section:relative:examples} we estimate the metric in several other useful cases.

Our first main result provides a characterization in the case that the quasidisks $U,V$ have a common point at the boundary. We may assume that this common point of ``tangency" is at $\infty$. Thus, in the next statement we use the topology of $\C$ and $|\cdot|$ denotes the Euclidean metric. 

\begin{theorem}[Tangent quasidisks]\label{theorem:intro:characterization_tangent}
Let $U,V\subset \C$ be unbounded quasidisks such that $\bar U\cap \bar V=\emptyset$. There exists a quasiconformal map $f\colon \C\to \C$ that maps $U$ and $V$ to half-planes if and only if the identity map
$$\id \colon (\partial V,d_{V,U})\to (\partial V,|\cdot|)$$
is quasisymmetric, quantitatively.  
\end{theorem}

The proof is given in Section \ref{section:relative:tangent}. Next, we give the corresponding result for quasidisks with disjoint closures. Here we use the topology of the sphere and $\chi$ denotes the chordal metric.

\begin{theorem}[Quasiannulus]\label{theorem:intro:characterization_quasiannuli}
Let $U,V\subset \widehat \C$ be quasidisks such that $\bar U\cap \bar V=\emptyset$. There exists a quasiconformal map $f\colon \widehat\C\to \widehat \C$ that maps $U$ and $V$ to disks if and only if the identity map
$$\id \colon (\partial V,d_{V,U})\to (\partial V,\chi)$$
is quasi-M\"obius, quantitatively.  
\end{theorem}

This result is proved in Section \ref{section:relative:quasiannuli}. Note that we need to switch here to quasi-M\"obius maps, which generalize quasisymmetric maps and were introduced by V\"ais\"al\"a \cite{Vaisala:quasimobius}. The reason is that the existence of a quasiconformal map $f\colon \widehat\C\to \widehat \C$ that maps $U$ and $V$ to disks is a quasiconformal invariant. Yet, quasiconformal maps of the sphere are not quasisymmetric in a quantitative way; in general, they are only quasi-M\"obius. This partially justifies the use of quasi-M\"obius maps in Theorem \ref{theorem:intro:characterization_quasiannuli}.

We may combine the two statements into a single one as follows. Again, we use the topology of the sphere and the proof is included in Section \ref{section:relative:quasiannuli}.

\begin{theorem}\label{theorem:intro:characterization_combined}
Let $U,V\subset \widehat \C$ be quasidisks such that $\bar U\cap \bar V$ contains at most one point. There exists a quasiconformal map $f\colon \widehat\C\to \widehat \C$ that maps $U$ and $V$ to disks if and only if the identity map
$$\id \colon (\partial V\setminus \bar U,d_{V,U})\to (\partial V\setminus \bar U,\chi)$$
is quasi-M\"obius, quantitatively.  
\end{theorem}

We would like to emphasize the \textit{quantitative} nature of the results. For example, in Theorem \ref{theorem:intro:characterization_combined}, if $U,V$ are $L$-quasidisks and the displayed identity map is $\eta$-quasi-M\"obius, then the quasiconformal map $f$ is $K(L,\eta)$-quasiconformal. Conversely, if $f$ is $U,V$ are $L$-quasidisks and $f$ is $K$-quasiconformal, then the displayed identity map is $\eta$-quasi-M\"obius, where $\eta$ depends only on $K,L$. Note that in the above results we have given preference to the metric $d_{V,U}$ defined on $\partial V$. Of course the roles of $U,V$ are symmetric so one could instead use the metric $d_{U,V}$ on $\partial U$. 

In Section \ref{section:relative:examples} we establish the conditions of the above theorems in some special and useful cases. We include illustrations here and defer the precise statements to Section \ref{section:relative:examples}. We prove in Lemma \ref{lemma:quasicircle_parallel} that the condition of Theorem \ref{theorem:intro:characterization_tangent} holds quantitatively in the case that $U,V$ are unbounded quasidisks whose boundaries are close to parallel lines; see Figure \ref{figure:parallel}. More generally, in Lemma \ref{lemma:wormhole} we see that the condition holds if $U,V$ are separated by a chord-arc curve $J$ whose distance and Hausdorff distance to $\partial U,\partial V$ are comparable to $1$; see Figure \ref{figure:wormhole}. Moreover,  the condition of Theorem \ref{theorem:intro:characterization_quasiannuli} holds quantitatively if $\Delta(\partial U,\partial V)$ is bounded from below away from $0$ (see Lemma \ref{lemma:relatively_separated}) or if $U,V$ are disjoint quasidisks whose boundaries are close to concentric circles (see Lemma \ref{lemma:quasicircle_concentric}), as in Figure \ref{figure:concentric}. 

\begin{figure}
	\begin{tikzpicture}[scale=0.7]
		\draw[dashed] (-4,0)--(4,0);
		\draw[dashed] (-4,-1)--(4,-1);
		\draw[dashed] (-4,1)--(4,1);
		\draw[dashed] (-4,2)--(4,2);
		\draw plot[domain=-4:4, samples=50] (\x, 
		{-0.5
	   	+ 0.25*sin((3*\x + 0.7))
   		+ 0.14*sin((17*sqrt(2)*\x + 1.3))
   		+ 0.09*cos((41*pi*\x + 0.2))
   		+ 0.05*sin((97*e*\x + 2.1))
   		+ 0.03*cos((211*sqrt(5)*\x + 0.9))});
   		\draw plot[domain=-4:4, samples=50] (\x,
		{1.3
   		+ 0.22*cos((2.5*\x + 1.1))
   		+ 0.16*sin((13*sqrt(3)*\x + 0.4))
   		+ 0.10*cos((29*pi*\x + 1.9))
   		+ 0.06*sin((73*e*\x + 0.6))
   		+ 0.035*cos((181*sqrt(7)*\x + 2.4))
   		+ 0.02*sin((419*pi*\x + 0.3))
  		});
  		
  		\draw[<->, xshift=.2cm] (4,-0.1)--(4,-.9) node[pos=0.5,right] {$LK$};
  		\draw[<->, xshift=.2cm] (4,1.1)--(4,1.9) node[pos=0.5,right] {$LK$};
  		\draw[<->, xshift=.2cm] (4,0.1)--(4,0.9) node[pos=0.5,right] {$L$};
  		\node[left] at (-4,-0.7) {$\partial V$};
  		\node[left] at (-4,1.5) {$\partial U$};
	\end{tikzpicture}
	\caption{Two unbounded quasidisks $ U, V$ whose boundaries are close to parallel lines, as in Lemma \ref{lemma:quasicircle_parallel}. In this case, $d_{V,U}(z,w)\simeq_K \dist_e(\partial U,\partial V)^{-1}|z-w|$.}\label{figure:parallel}
\end{figure}
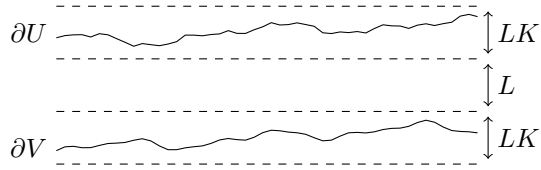

\begin{figure}
	\begin{tikzpicture}[scale=0.7]
		\draw plot[domain=-2:6.2, samples=100] 
		(\x,{sin(deg(\x))+ 0.20*sin(deg(3*\x + 0.7))
      	+ 0.12*sin(deg(7*\x + 1.4))
     	+ 0.08*cos(deg(13*\x + 0.3))
     	+ 0.05*sin(deg(23*\x + 2.1))
      	+ 0.03*cos(deg(41*\x + 0.9))}) node[right] {$\partial U$};
      
      \draw plot[domain=-2:6.2, samples=100]
		(\x, {-2+ sin(deg(\x))
    	+ 0.18*sin(deg(4*\x + 1.1))
    	+ 0.14*cos(deg(9*\x + 0.2))
    	+ 0.09*sin(deg(15*\x + 1.8))
    	+ 0.06*cos(deg(27*\x + 0.6))
    	+ 0.035*sin(deg(45*\x + 2.4))
  		}) node[right] {$\partial V$};
  		
  		\draw[dashed] plot[domain=-2:6.28, samples=100](\x, {-1+ sin(deg(\x))}) node[right] {$J$};
	\end{tikzpicture}
	\caption{Two unbounded quasidisks $U,V$ and a chord-arc curve $J$ whose distance and Hausdorff distance to each of $\partial U,\partial V$ are comparable to $1$, as in Lemma \ref{lemma:wormhole}. In this case, $d_{V,U}(z,w)\simeq_K |z-w|$.}\label{figure:wormhole}
\end{figure}
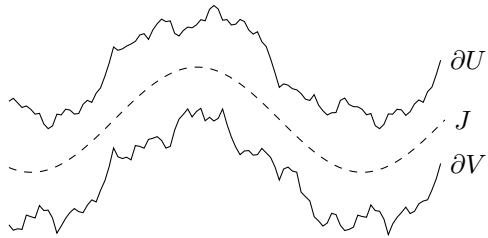

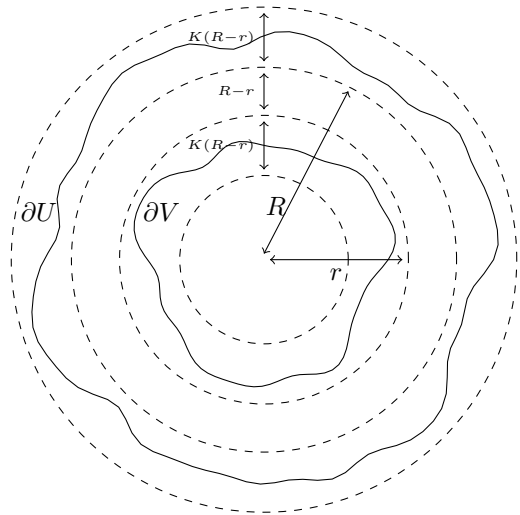
\begin{figure}
	\begin{tikzpicture}[scale=1.6]

		\draw[dashed](0,0) circle (1.2cm);
		\draw[dashed](0,0) circle (0.7cm);
		\draw[dashed](0,0) circle (1.6cm);
		\draw[dashed](0,0) circle (2.1cm);
		\draw plot[domain=0:360, samples=100] (
	  	{(1
    	+ 0.08*sin(3*\x + 0.5)
    	+ 0.05*sin(5*\x + 1.0)
    	+ 0.03*cos(7*\x + 0.3)
    	+ 0.02*sin(11*\x + 1.5)
   		)*cos(\x)},
  		{(1
    	+ 0.08*sin(3*\x + 0.5)
    	+ 0.05*sin(5*\x + 1.0)
    	+ 0.03*cos(7*\x + 0.3)
    	+ 0.02*sin(11*\x + 1.5)
  		)*sin(\x)}
		);
		\draw plot[domain=0:360, samples=100] (
  		{(1.85
    	+ 0.07*sin(4*\x + 1.2)
    	+ 0.06*sin(6*\x + 0.4)
    	+ 0.04*cos(9*\x + 0.9)
    	+ 0.025*sin(14*\x + 1.7)
    	+ 0.015*cos(20*\x + 0.5)
   		)*cos(\x)},
  		{(1.85
    	+ 0.07*sin(4*\x + 1.2)
    	+ 0.06*sin(6*\x + 0.4)
    	+ 0.04*cos(9*\x + 0.9)
    	+ 0.025*sin(14*\x + 1.7)
    	+ 0.015*cos(20*\x + 0.5)
   		)*sin(\x)}
		);
	
		\draw[<->] (0.05,0)--(1.15,0) node[pos=.5,below]{$r$};
		\draw[<->] (0,0.05)--(0.7,1.4);
		\node at (0.1,0.45) {$R$};
	
		\draw[<->] (0,0.75)--(0,1.15)node[pos=.5,left]{$\scriptscriptstyle K(R-r)$};
		\draw[<->] (0,1.25)--(0,1.55)node[pos=.5,left]{$\scriptscriptstyle R-r$};
		\draw[<->] (0,1.65)--(0,2.05)node[pos=.5,left]{$\scriptscriptstyle K(R-r)$};
		\node at (-0.85,0.4) {$\partial V$};
		\node at (-1.87,0.4) {$\partial U$};
	\end{tikzpicture}
	\caption{Two quasidisks $U,V$ whose boundaries are close to concentric circles, as in Lemma \ref{lemma:quasicircle_concentric}. In this case, $d_{V,U}(z,w)\simeq_K {\dist_e(\partial U,\partial V)}^{-1}{|z-w|}$.}\label{figure:concentric}
\end{figure}

One can ask whether there exists a simpler condition, than the one in Theorem \ref{theorem:intro:characterization_combined}, to characterize pairs of quasicircles that can be quasiconformally mapped to circles. For example, quasicircles can be characterized by means of cross ratios (see \eqref{quasicircle:crossratio}), so it would be ideal to have a similar characterization for pairs. However, it is highly doubtful that such a simple-minded criterion exists. Indeed, in Theorem \ref{theorem:intro:graph} below we give a neat characterization in the case that one of the quasicircles is a line and the other is the graph of a Lipschitz function; yet this characterization is still involved and is quite far from a geometric condition involving cross ratios.

The proofs of the main results rely on several tools from the theory of quasiconformal and quasisymmetric maps, but also from classical conformal and hyperbolic geometry. In addition, we use the notion of uniform domains, and distortion properties of the hyperbolic and quasihyperbolic metrics under quasiconformal maps. All the required preliminaries are included in Seciton \ref{section:preliminaries}. 

In order to establish the theorems, we need to show that the condition that 
$$\text{the map}\,\,\,\id\colon (\partial V,d_{V,U})\to (\partial V,\chi)\,\,\, \text{is quasi-M\"obius}$$
is quasiconformally invariant. The necessity part of the theorems then easily follows from the fact the condition holds in the case that $U,V$ are disks. The mentioned quasiconformal invariance follows from some special properties of the relative hyperbolic metric $d_{V,U}$ that we establish. One important ingredient is that if $f$ is a quasiconformal homeomorphism of the sphere mapping the quasidisks $U,V$ to $U',V'$, then the map 
$$f\colon (\partial V,d_{V,U})\to (\partial V',d_{V',U'})$$
is quasisymmetric. This is established in Theorem \ref{theorem:relative_hyperbolic_qc}. Another important tool that allows us to estimate $d_{V,U}$ is that it is quasisymmetrically equivalent to the quasihyperbolic distance in the domain $Z$ that is bounded by $\partial U$ and some quasiconformal reflection of $\partial U$ along $\partial V$; see Corollary \ref{corollary:quasisymmetric_dk} and Figure \ref{figure:z}. 

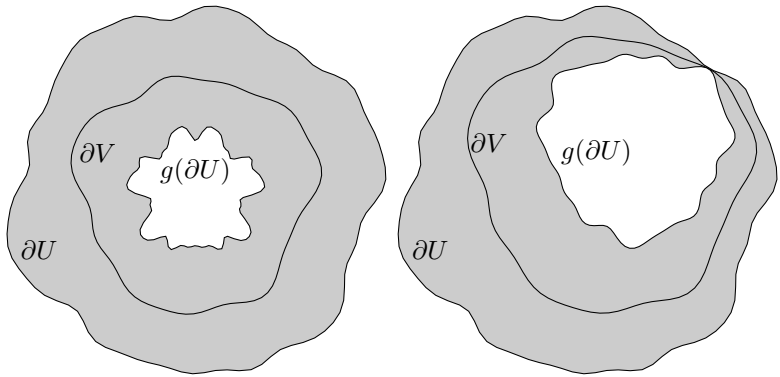
\begin{figure}
	\begin{tikzpicture}[scale=1.3]
	\begin{scope}
		\filldraw[fill=black!20] plot[domain=0:360, samples=100] (
  		{(1.85
    	+ 0.07*sin(4*\x + 1.2)
    	+ 0.06*sin(6*\x + 0.4)
    	+ 0.04*cos(9*\x + 0.9)
    	+ 0.025*sin(14*\x + 1.7)
    	+ 0.015*cos(20*\x + 0.5)
   		)*cos(\x)},
  		{(1.85
    	+ 0.07*sin(4*\x + 1.2)
    	+ 0.06*sin(6*\x + 0.4)
    	+ 0.04*cos(9*\x + 0.9)
    	+ 0.025*sin(14*\x + 1.7)
    	+ 0.015*cos(20*\x + 0.5)
   		)*sin(\x)}
		);
		\draw plot[domain=0:360, samples=100] (
	  	{(1.2
    	+ 0.08*sin(3*\x + 0.5)
    	+ 0.05*sin(5*\x + 1.0)
    	+ 0.03*cos(7*\x + 0.3)
    	+ 0.02*sin(11*\x + 1.5)
   		)*cos(\x)},
  		{(1.2
    	+ 0.08*sin(3*\x + 0.5)
    	+ 0.05*sin(5*\x + 1.0)
    	+ 0.03*cos(7*\x + 0.3)
    	+ 0.02*sin(11*\x + 1.5)
  		)*sin(\x)}
		);
		
		\filldraw[fill=white] plot[domain=0:360, samples=100] (
  		{(0.6
    	+ 0.065*sin(5*\x + 0.9)
    	+ 0.055*sin(7*\x + 0.6)
    	+ 0.045*cos(10*\x + 1.1)
    	+ 0.028*sin(15*\x + 1.3)
    	+ 0.018*cos(22*\x + 0.2)
   		)*cos(\x)},
  		{(0.6
    	+ 0.065*sin(5*\x + 0.9)
    	+ 0.055*sin(7*\x + 0.6)
   	 	+ 0.045*cos(10*\x + 1.1)
   	 	+ 0.028*sin(15*\x + 1.3)
    	+ 0.018*cos(22*\x + 0.2)
  		)*sin(\x)}
		);
	
		\node at (-1,0.4) {$\partial V$};
		\node at (-1.6,-0.6) {$\partial U$};
		\node at (0,0.2) {$g(\partial U)$};
	\end{scope}
	\begin{scope}[shift={(4,0)}]
	\filldraw[fill=black!20] plot[domain=0:360, samples=100] (
  		{(1.85
    	+ 0.07*sin(4*\x + 1.2)
    	+ 0.06*sin(6*\x + 0.4)
    	+ 0.04*cos(9*\x + 0.9)
    	+ 0.025*sin(14*\x + 1.7)
    	+ 0.015*cos(20*\x + 0.5)
   		)*cos(\x)},
  		{(1.85
    	+ 0.07*sin(4*\x + 1.2)
    	+ 0.06*sin(6*\x + 0.4)
    	+ 0.04*cos(9*\x + 0.9)
    	+ 0.025*sin(14*\x + 1.7)
    	+ 0.015*cos(20*\x + 0.5)
   		)*sin(\x)}
		);
		\draw plot[domain=0:360, samples=100] (
	  	{(1.4
    	+ 0.08*sin(3*\x + 0.5)
    	+ 0.05*sin(5*\x + 1.0)
    	+ 0.03*cos(7*\x + 0.3)
    	+ 0.02*sin(11*\x + 1.5)
   		)*cos(\x)+0.25},
  		{(1.4
    	+ 0.08*sin(3*\x + 0.5)
    	+ 0.05*sin(5*\x + 1.0)
    	+ 0.03*cos(7*\x + 0.3)
    	+ 0.02*sin(11*\x + 1.5)
  		)*sin(\x)+0.21}
		);
		\filldraw[fill=white]
		plot[domain=0:360, samples=140] (
  		{(1
    	+ 0.07*(sin(3*\x+1.1)-sin(1.1))
    	+ 0.035*(cos(8*\x+0.1)-cos(0.1))
    	+ 0.03*(sin(11*\x+0.4)-sin(0.4))
    	+ 0.02*(cos(15*\x+1.3)-cos(1.3))
   		)*cos(\x)+0.5},
  		{(1
    	+ 0.07*(sin(3*\x+1.1)-sin(1.1))
    	+ 0.035*(cos(8*\x+0.1)-cos(0.1))
    	+ 0.03*(sin(11*\x+0.4)-sin(0.4))
    	+ 0.02*(cos(15*\x+1.3)-cos(1.3))
   		)*sin(\x)+0.5}
		);
		\node at (-1,0.5) {$\partial V$};
		\node at (-1.6,-0.6) {$\partial U$};
		\node at (0.1,0.4) {$g(\partial U)$};
	\end{scope}
	\end{tikzpicture}
	\caption{Shown in gray is the region $Z$ bounded by $\partial U$ and a quasiconformal reflection $g(\partial U)$ of $\partial U$ along $\partial V$ when $\bar U\cap \bar V=\emptyset$ (left) and when $\bar U\cap \bar V$ is a singleton (right).}\label{figure:z}
\end{figure}

One of the difficulties in the proof of the main results is that several normalizations are required for the estimation of $d_{V,U}$. Ideally, we would like to have that $\partial U$ and $\partial V$ are close to concentric circles or parallel lines, but we can achieve this only after normalizations and certain transformations. Another issue is that in the setting of Theorem \ref{theorem:intro:characterization_quasiannuli} it is not true that we can map $\partial U,\partial V$ to \textit{concentric}  circles with a quasisymmetric map, but only with a quasiconformal or quasi-M\"obius map.  This makes estimates more delicate throughout the proofs in an effort to ensure that quantitative control is not lost. 

We now discuss some applications of the main theorems in various problems in geometric function theory. 

\subsection{Schottky sets}

We can incorporate Theorem \ref{theorem:intro:characterization_combined} into the main theorem of \cite{Ntalampekos:schottky} and obtain the following quasiconformal characterization of Schottky sets.

\begin{theorem}[Schottky sets]\label{theorem:intro:schottky}
Let $\{U_i\}_{i\in I}$ be a collection of at least two disjoint Jordan regions in $\widehat \C$ such that $\bar{U_i}\cap \bar{U_j}$ contains at most one point for $i\neq j$. The following are quantitatively equivalent.
\begin{enumerate}[label=\normalfont(\roman*)]
\item For each pair of distinct $i,j\in I$ the identity map 
$$\id\colon (\partial U_i\setminus \bar {U_j}, d_{U_i,U_j})\to (\partial U_i\setminus \bar {U_j},\chi)$$
is uniformly quasi-M\"obius.
\item For each $i,j\in I$ there exists a uniformly quasiconformal homeomorphism of $\widehat \C$ that maps $U_i$ and $U_j$ to disks.
\item There exists a quasiconformal homeomorphism $f$ of $\widehat \C$ that maps $U_i$ to a disk for each $i\in I$ and is $1$-quasiconformal on $S=\widehat \C\setminus \bigcup_{i\in I} U_i$. 
\end{enumerate}
In this case, the map $f|_S$ is unique up to postcomposition with a conformal automorphism of $\widehat \C$. 
\end{theorem}

\subsection{Graphs and cusps}

We apply Theorem \ref{theorem:intro:characterization_tangent} to characterize graphs of Lipschitz functions and quasicircles in the upper half-plane $\mathbb H$ that can be mapped to a line with a quasiconformal homeomorphism of $\mathbb H$. The proofs of the results in this section are given in Section \ref{section:graph}. Remarkably, the next two results convey an interesting interplay between Lipschitz and quasiconformal geometry. 

\begin{theorem}[Graphs]\label{theorem:intro:graph}
Let $f\colon \R\to (0,\infty)$ be a Lipschitz function. There exists a quasiconformal homeomorphism of $\C$ that preserves the real line and maps the graph of $f$ onto a line and only if an antiderivative of $1/f$ is quasisymmetric. The statement is quantitative.
\end{theorem}

For example, the graph of the function $f(x)=(|x|+1)^{-p}$, $p>-1$, and the line $y=0$ can be quasiconformally mapped to a pair of parallel lines. On the other hand, this is impossible for $f(x)=e^{-|x|}$. See Example \ref{example:function} for details. Next, we provide a practical characterization of unbounded Jordan curves in the upper half-plane that can be mapped to a line with a quasiconformal map of that preserves the real axis. Here an \textit{unbounded Jordan curve} is a set $J\subset \C$ such that $J\cup \{\infty\}$ is a Jordan curve in the topology of the sphere. Each complementary component of an unbounded Jordan curve is an \textit{unbounded Jordan region}.

\begin{theorem}\label{theorem:intro:graph:characterization}
Let $U\subset \UHP$ be an unbounded Jordan region. There exists a quasiconformal homeomorphism of $\C$ that preserves the real line and maps $\partial U$ onto a line if and only if
\begin{enumerate}[label=\normalfont(\arabic*)]
\item\label{theorem:graph:characterization:1} $U$ is a quasidisk,
\item\label{theorem:graph:characterization:3} there exists a $1$-Lipschitz function $f\colon \R\to (0,\infty)$ and $a\geq 1$ such that $\partial U\subset \{ (x,y)\in \R^2: a^{-1}f(x)\leq y\leq af(x)\}$, and
\item\label{theorem:graph:characterization:4} an antiderivative of $1/f$ is quasisymmetric.
\end{enumerate}
The statement is quantitative.
\end{theorem}

The proof relies on Theorem \ref{theorem:intro:characterization_tangent} is one of the most technically involved arguments in the paper.

As an application of Theorem \ref{theorem:intro:graph}, we prove the quasiconformal equivalence of \textit{polynomial cusps}  $C_\alpha=\{(x,y)\in \R^2: |y|=x^\alpha, \, 0\leq x\leq 1\}$, where $\alpha>1$.

\begin{corollary}[Cusps]\label{corollary:intro:cusps}
All polynomial cusps are quasiconformally equivalent.
\end{corollary}

It is plausible that an explicit quasiconformal map can be found. A recent related result of Chrontsios--Tyson \cite{ChrontsiosTyson:assouad} provides a quasiconformal classification of \textit{polynomial spirals} $S_\alpha=\{x^{-\alpha} e^{ix}: x>0\}$, $\alpha>0$. Specifically, it is shown that for $\alpha>\beta>0$, $S_\alpha$ can be mapped to $S_\beta$ with a $K$-quasiconformal map if and only if $K\geq \alpha/\beta$. It would be interesting to prove an analogous result for cusps. 

Also, it is worth mentioning that a cusp cannot be straightened to a line segment with a quasiconformal map of $\C$, since such maps quasi-preserve angles. Yet, it can be straightened with \textit{mappings with finite distortion}; this class generalizes quasiconformal maps, which have bounded distortion instead. This problem is studied in the series of papers \cites{KoskelaTakkinen:cusp1, KoskelaTakkinen:cusp2, KoskelaTakkinen:cusp3, KoskelaTakkinen:cusp2note, GuoKoskelaTakkinen:quasidisks, IwaniecOnninenZhu:quasidisks}.

\subsection{The quasiconformal extension problem}

We present another application of Theorem \ref{theorem:intro:characterization_combined} in the problem of quasiconformally extending quasisymmetric or quasi-M\"obius maps, quantitatively. It is a classical fact, due to Ahlfors and Beurling \cites{BeurlingAhlfors:extension, Ahlfors:reflections}, that every quasisymmetric embedding of the unit circle into the plane extends to a quasiconformal homeomorphism of the plane, quantitatively (see Corollary \ref{corollary:beurling_ahlfors_embedding}). More generally, the same is true for every quasisymmetric embedding of a quasicircle. For which other classes of sets is such a result true? 

The question has been studied by Vellis \cite{Vellis:qs_extension_long} in the case of relatively connected compact sets whose complementary domains are uniform; we direct the reader to the referenced paper for definitions. Thus, a satisfactory resolution to the extension problem in the case of sets with controlled geometry has been provided. For example, every quasisymmetric embedding of the middle-thirds Cantor set extends to a quasiconformal homeomorphism of the plane, quantitatively \cite{MacManus:cantor}.

As an application of the main theorems, we resolve the extension problem for embeddings of pairs of circles. Note that the complementary domain of two tangent disks is not uniform. Also, if two disks have disjoint closures, the complementary domain is uniform, but the uniformity constant degenerates as the disks tend to be tangent to each other. Hence, the previous results do not apply in these cases. 

\begin{theorem}[Tangent disks]\label{theorem:intro:extension}
Let $U,V\subset \widehat \C$ be disjoint tangent disks and $f\colon \partial U\cup \partial  V\to \widehat \C$ be a quasi-M\"obius embedding. Then there exists an extension of $f$ to a (possibly orientation-reversing) quasiconformal homeomorphism of $\widehat \C$, quantitatively.
\end{theorem}
Note that we do not impose the assumption that $f$ admits an extension to a homeomorphism of the sphere; this follows from the assumptions that $f$ is quasi-M\"obius and $\partial U\cap \partial V\neq \emptyset$. However, in the next result it is obviously necessary to impose the existence of a homeomorphic extension.

\begin{theorem}[Annulus]\label{theorem:extension:disjoint}
Let $U,V\subset \widehat \C$ be disks with disjoint closures and $f\colon \partial U\cup \partial  V\to \widehat \C$ be a quasi-M\"obius embedding that extends to an orientation-preserving homeomorphism of $\widehat \C$. Denote by $\Gamma$ (resp.\ $\Gamma'$) the family of curves separating $\partial U$ and $\partial V$ (resp.\ $f(\partial U)$ and $f(\partial V)$). Suppose that one of the following conditions holds.
\begin{enumerate}[label=\normalfont(\arabic*)] 
\item $\displaystyle K_0^{-1}\leq  \frac{\Mod\Gamma'}{\Mod\Gamma}\leq K_0$ for some $K_0\geq 1$.\label{theorem:extension:disjoint:1}
\medskip
\item $\displaystyle \Mod\Gamma\leq M$ for some $M>0$. \label{theorem:extension:disjoint:2}
\end{enumerate}
Then there exists an extension of $f$ to a quasiconformal homeomorphism of $\widehat \C$, quantitatively.
\end{theorem}

Since quasiconformal maps quasi-preserve modulus, it is immediate that the modulus bound in \ref{theorem:extension:disjoint:1} is a necessary condition for an extension with quantitative control. When $\Mod\Gamma$ is small, as in \ref{theorem:extension:disjoint:2}, condition \ref{theorem:extension:disjoint:1} follows from the assumption that the embedding $f$ is quasi-M\"obius. In general, \ref{theorem:extension:disjoint:1} does not follow from the assumption that $f$ is quasi-M\"obius and cannot be dropped. Indeed, consider the map $f$ that is the identity map on the unit circle $\mathbb S(0,1)$ and scales the circle $\mathbb S(0,e^{-n})$ to $\mathbb S(0,\frac{1}{n})$, where $n\geq 2$. This map is $\eta$-quasisymmetric and $\eta$-quasi-M\"obius (see property \ref{q:qs_qm} below) for some distortion function $\eta$ that does not depend on $n$. However, \ref{theorem:extension:disjoint:1} is not true with a constant $K_0$ independent of $n$.

Theorem \ref{theorem:intro:extension} and Theorem \ref{theorem:extension:disjoint} are proved in Section \ref{section:extension_disks} and are based on several quasiconformal extension results from Section \ref{section:extension_general}. For example we prove that a quasisymmetric map between two pairs of parallel lines or two pairs of concentric circles extends to a quasiconformal map of the plane. While results in this spirit seem very natural and have been extensively used in complex dynamics \cite{BrannerFagella:surgery}*{Section 2.3}, we have not been able to locate in the literature \textit{quantitative proofs}. Namely, the quasiconformal dilatation of the extension in Theorem \ref{theorem:intro:extension} depends only on the distortion function of the quasi-M\"obius embedding $f$. Even more delicate is the proof of Theorem \ref{theorem:extension:disjoint}, where the dilatation depends also on the constant appearing in condition \ref{theorem:extension:disjoint:1} or \ref{theorem:extension:disjoint:2}.

\section{Preliminaries}\label{section:preliminaries}

\subsection{Notation}

If $a$ is a parameter we use the notation $C(a),L(a)$, etc.\ for positive constants that depend only on the parameter $a$. For quantities $A$ and $B$ we write $A\lesssim B$ if there exists a constant $c>0$ such that $A\leq cB$. If the constant $c$ depends on another quantity $H$ that we wish to emphasize, then we write instead $A\leq c(H)B$ or $A\lesssim_H B$. Moreover, we use the notation $A\simeq B$ if $A\lesssim B$ and $B\lesssim A$. As previously, we write $A\simeq_H B$ to emphasize the dependence of the implicit constants on the quantity $H$. All constants in the statements are assumed to be positive even if this is not stated explicitly and the same letter may be used in different statements to denote a different constant. We say that a statement is \textit{quantitative} if the constants or parameters appearing in the conclusions depend only on the constants or parameters in the assumptions.

The chordal metric on $\widehat \C$ is defined for $z\in \C$ by 
$$\chi(z,w)= \frac{2|z-w|}{\sqrt{1+|z|^2}\sqrt{1+|w|^2}}  \,\,\, \text{when $w\in \C$ and}\,\,\, \chi(z,\infty)= \frac{2}{\sqrt{1+|z|^2}}.$$
Let $(X,d)$ be a metric space. The open ball of radius $r>0$ centered at a point $x\in X$ is denoted by $B_d(x,r)$ and the closed ball is denoted by $\bar B_d(x,r)$. The diameter of a set $E\subset X$ is denoted by $\diam (E)$ or $\diam_d(E)$ if the metric needs to be emphasized. The length of a curve $\gamma\colon [a,b]\to X$ is denoted by $\ell_d(\gamma)$. The trace $\gamma$ is the set $|\gamma|=\gamma([a,b])$. For the Euclidean metric in $\R^n$ we use the subscript $e$ when necessary. For example, $B_e(z,r)=\{w\in \R^n: |z-w|<r\}$. Moreover, especially for the Euclidean metric in the complex plane we also use the notation $\D(z,r)$ for the ball $B_e(z,r)$, $\mathbb S(z,r)$ for the circle $\partial B_e(z,r)$, and $\mathbb A(z;r,R)$ for the annulus $\{w\in \C: r<|z-w|<R\}$, where $R>r>0$. 

\subsection{Quasiconformal maps}

Let $K\geq 1$. An orientation-preserving homeomorphism $f\colon U\to V$ between open subsets of $\C$ is \textit{quasiconformal} if $f$ lies in the Sobolev space $W^{1,2}_{\loc}(U)$ and
$$\|Df(z)\|^2\leq KJ_f(z)$$ 
for a.e.\ $z\in \C$; here $\|\cdot \|$ denotes the operator norm of the matrix $Df$ and $J_f$ is its Jacobian determinant. In this case we say that $f$ is $K$-quasiconformal. A homeomorphism $f\colon U\to V$ between open subsets of $\widehat \C$ is quasiconformal if $f|_{U\setminus \{\infty, f^{-1}(\infty)\}}$ is quasiconformal in the above sense. We will freely use the fundamental facts that the inverse of a quasiconformal map is quasiconformal, that the composition of two quasiconformal maps is quasiconformal, and that $1$-quasiconformal maps are conformal. We direct the reader to \cite{LehtoVirtanen:quasiconformal} for further background and to \cite{Vaisala:quasiconformal} for the theory of quasiconformal maps in higher dimensions. 

Unless otherwise specified, quasiconformal are considered to be orienta\-tion-pre\-serving, as in the above definition. We also define \textit{orientation-reversing quasiconformal maps} in the obvious manner. 

We will use the notion of modulus of a curve family $\Gamma$, which we denote by $\Mod\Gamma$; see \cite{LehtoVirtanen:quasiconformal}*{Chapter I} for the definition and details. We record here some fundamental facts. Let $\Gamma$ be the family of curves separating the boundary components of an annulus $\mathbb A(z;r,R)$, where $0<r<R<\infty$. We have
$$\Mod \Gamma =\frac{1}{2\pi} \log \left(\frac{R}{r}\right).$$ 
We note that modulus is conformally invariant and quasiconformally quasi-inva\-riant. That is, if $f\colon U\to V$ is a $K$-quasiconformal map between open subsets of $\widehat \C$ and $\Gamma$ is a family of curves in $U$, then
$$K^{-1}\Mod \Gamma\leq \Mod f(\Gamma) \leq K\Mod \Gamma.$$

\subsection{Quasi-M\"obius and quasisymmetric maps}

Let $(X,d_X)$ be a metric space. We define the cross ratio of a quadruple of distinct points $a,b,c,d\in X$ to be
$$[a,b,c,d]= \frac{d_X(a,c)d_X(b,d)}{d_X(a,d)d_X(b,c)}.$$
Observe that if $X=\widehat \C$, equipped with the chordal metric, and $a,b,c,d\neq \infty$, then 
$$[a,b,c,d]= \frac{|a-c||b-d|}{|a-d||b-c|}.$$
The factors containing the point $\infty$ are omitted. For example, 
$$[a,b,c,\infty]= \frac{|a-c|}{|b-c|}.$$

A homeomorphism $\eta\colon [0,\infty)\to [0,\infty)$ is called a \textit{distortion function}. A homeomorphism $f\colon (X,d_X)\to (Y,d_Y)$ between metric spaces is \textit{quasi-M\"obius} if there exists a distortion function $\eta$ such that 
$$[f(a),f(b),f(c),f(d)] \leq \eta ([a,b,c,d])$$
for every quadruple of distinct points $a,b,c,d\in X$. In this case, we say that $f$ is $\eta$-quasi-M\"obius. Suppose that $X,Y\subset \widehat \C$ and $f\colon X\to Y$ is $\eta$-quasi-M\"obius. Then the cross ratios in $X$ and $Y$ may be computed with the chordal or the Euclidean distance interchangeably (in case the points involved are different from $\infty$). Thus, we do not need to specify which of the two metrics is used in $X$ and $Y$. Also, since cross ratios are invariant under M\"obius transformations, the composition of $f$ with a M\"obius transformation is $\eta$-quasi-M\"obius as well.  

A homeomorphism $f\colon (X,d_X)\to (Y,d_Y)$ between metric spaces is \textit{quasisymmetric} if there exists a distortion function $\eta$ such that
$$ \frac{d_Y(f(a),f(b))}{d_Y(f(a),f(c))} \leq \eta\left(\frac{d_X(a,b)}{d_X(a,c)}\right)$$
for every triple of distinct points $a,b,c\in X$. In this case we say that $f$ is $\eta$-quasisymmetric. An elementary consequence of the definition (see \cite{Heinonen:metric}*{Proposition 10.8}) is that if $A\subset B\subset X$ and $0<\diam (A)\leq \diam (B)<\infty$, then $\diam (f(B))<\infty$ and 
\begin{align}\label{definition:qs}
\frac{1}{2\eta\left( \frac{\diam (B)}{\diam (A)}\right)} \leq \frac{\diam (f(A))}{\diam (f(B))}\leq \eta\left( \frac{2\diam (A)}{\diam (B)}\right).
\end{align}

A homeomorphism $f\colon (X,d_X)\to (Y,d_Y)$ between metric spaces is \textit{bi-Lipschitz} if there exists $L\geq 1$ such that for all $x,y\in X$ we have
$$L^{-1}d_X(x,y)\leq d_Y(f(x),f(y))\leq Ld_X(x,y).$$
In this case we say that $f$ is $L$-bi-Lipschitz. 

We list some fundamental properties of the above types of maps. 

\begin{enumerate}[label=\normalfont (Q-\arabic*)]
	\item\label{q:inverse} If $f$ is $\eta$-quasisymmetric (resp.\ quasi-M\"obius), then $f^{-1}$ is $\eta'$-quasi\-sym\-metric (resp.\ quasi-M\"obius) for $\eta'(t)=1/\eta^{-1}(t^{-1})$, $t>0$. Also, if $f$ is $\eta_1$-quasisymmetric (resp.\ quasi-M\"obius) and $g$ is $\eta_2$-quasisymmetric (resp.\ quasi-M\"obius), then $g\circ f$ is $(\eta_2\circ \eta_1)$-quasisymmetric (resp.\ quasi-M\"obius).
	
	\item\label{q:completion}Suppose that $X,Y\subset \C$ (resp.\ are bounded) and $f\colon X\to Y$ is $\eta$-quasi\-symmetric (resp.\ $\eta$-quasi-M\"obius). Then $f$ extends to an $\eta$-quasisymmetric (resp.\ $\eta$-quasi-M\"obius) homeomorphism from $\bar X$ onto $\bar Y$.  
	
	\item\label{q:qs_qm} Quasisymmetric maps are quasi-M\"obius, quantitatively \cite{Vaisala:quasimobius}*{Theorem 3.2}. Conversely, quasi-M\"obius maps between bounded metric spaces are quasisymmetric \cite{Vaisala:quasimobius}*{Theorem 3.12}. Specifically, if $f\colon X\to Y$ is quasi-M\"obius and there exists $\lambda\geq 1$ and points $x_1,x_2,x_3\in X$ with 
	$$d_X(x_i,x_j) \geq \lambda^{-1}\diam (X) \,\,\, \text{and}\,\,\, d_Y(f(x_i),f(x_j))\geq \lambda^{-1}\diam (Y)\,\,\, \text{for $i\neq j$},$$
	then $f$ is quasisymmetric, quantitatively.	
	
	\item\label{q:qm_qc} Quasi-M\"obius maps between open subsets of $\widehat \C$ are quasiconformal, quantitatively \cite{Vaisala:quasimobius}*{Theorem 5.2}. Conversely, quasiconformal self-homeo\-morphisms of the disk $\D$ or of the sphere $\widehat \C$ are quasi-M\"obius, quantitatively \cite{Vaisala:quasimobius}*{Theorem 5.4}.

	\item\label{q:qm_qs} Suppose that $X$ is an unbounded metric space, $Y$ is a metric space and $f\colon X\to Y$ is $\eta$-quasi-M\"obius. If $f(x)\to \infty$ as $x\to\infty$, then $f$ is $\eta$-quasisymmetric \cite{Vaisala:quasimobius}*{Theorem 3.10}.
	
	\item\label{q:qc_qs}Let $f\colon U\to V$ be a $K$-quasiconformal map between open subsets of $\R^n$. For each $a>1$ there exists a distortion function $\eta$, depending only on $n,K$, and $a$, such that if $B_e(z,r)\subset B_e(z,ar)\subset U$, then $f$ is $\eta$-quasisymmetric in $B_e(z,r)$ \cite{Heinonen:metric}*{Theorem 11.4}. 
	
	\item\label{q:bilip}If $f\colon X\to Y$ is an $L$-bi-Lipschitz map between metric spaces, then $f$ is $\eta$-quasi\-sym\-metric for $\eta(t)=L^2t$. If $X,Y$ are domains in $\C$, then $f$ is $L^2$-quasi\-con\-formal \cite{Vaisala:quasiconformal}*{Theorem 34.1}.
	\end{enumerate}

A \textit{continuum} $E$ is a compact and connected metric space. If $E$ contains more than one points, we say that $E$ is a \textit{non-degenerate continuum}. We define the \textit{relative distance} of two disjoint compact subsets $E,F$ of a metric space that contain more than one points as 
$$\Delta(E,F)=\frac{\dist(E,F)}{\min\{\diam(E),\diam(F)\}}.$$
As before we use the notation $\Delta_d(E,F)$ to emphasize the metric that we use. 

\begin{lemma}\label{lemma:quasisymmetric_relative_distance}
Let $f\colon (X,d_X)\to (Y,d_Y)$ be a quasisymmetric homeomorphism between metric spaces. Let $E,F\subset X$ be sets with positive diameters. Then 
$$ \frac{\dist(f(E),f(F))}{\diam (f(E)) }\leq \eta\left(2\frac{\dist(E,F)}{\diam (E)}\right).$$
\end{lemma}

\begin{proof}
Let $x\in E$, $z\in F$ be arbitrary points and let $\varepsilon>0$. Then there exists a point $y\in E$ such that $d(x,y)\geq \frac{1}{2+\varepsilon}\diam (E)$. We have
\begin{align*}
 \frac{\dist(f(E),f(F))}{\diam (f(E)) }&\leq \frac{d(f(x), f(z))}{d(f(x), f(y))}\leq \eta\left(\frac{d(x,z)}{d(x,y)}\right)\leq \eta \left((2+\varepsilon) \frac{d(x,z)}{\diam (E)}\right).
\end{align*}
Infimizing over $x,z$ and letting $\varepsilon\to 0$ gives the conclusion.
\end{proof}

\begin{lemma}[\cite{Ntalampekos:schottky}*{Lemma 2.4}]\label{lemma:quasimobius:cross}
Let $f\colon (X,d_X)\to (Y,d_Y)$ be an $\eta$-quasi-M\"obius homeomorphism between metric spaces. Then there exists a distortion function $\widetilde \eta$ that depends only on $\eta$ such that for every pair of disjoint non-degenerate continua $E,F\subset X$ we have
$$\Delta(f(E),f(F)) \leq \widetilde \eta ( \Delta(E,F)).$$
\end{lemma}

\begin{remark}\label{remark:quasimobius:cross}
We record a simple consequence of Lemma \ref{lemma:quasimobius:cross}. The identity map $\id\colon (\C,\chi)\to (\C,|\cdot|)$ preserves cross ratios so it is $\eta$-quasi-M\"obius for $\eta(t)=t$. If $E,F\subset \C$, $a>0$, and $\Delta_\chi(E,F)\leq a$ (resp.\ $\geq a$), then Lemma \ref{lemma:quasimobius:cross} implies that $\Delta_e(E,F)\leq \widetilde \eta(a)$ (resp.\ $\geq (\widetilde \eta)^{-1}(a)$). Conversely, if $\Delta_e(E,F)$ is bounded above or below, one obtains bounds for $\Delta_\chi(E,F)$.
\end{remark}

\subsection{Quasidisks and quasicircles}\label{section:quasidisks}
A set $U\subset \widehat \C$ is a \textit{quasidisk} if it is the image of the unit disk $\D$ (or equivalently of any other disk) under a quasiconformal homeomorphism $f\colon \widehat \C\to \widehat \C$. If $f$ is $K$-quasiconformal for some $K\geq1$, we say that $U$ is a $K$-quasidisk. Note that the complement of a quasidisk is also a quasidisk. We say that a set $U\subset \C$ is an \textit{unbounded quasidisk} if $U$ is unbounded as a subset of $\C$ and it is a quasidisk as a subset of $\widehat \C$.

Let $J\subset \widehat \C$ be a Jordan curve. We say that $J$ is a \textit{quasicircle} if there exists a constant $L\geq 1$ such that for every pair of points $z,w\in J$ there exists an arc $E\subset J$ with endpoints $z,w$ such that 
$$\diam_\chi (E)\leq L\chi(z,w).$$
In this case we say that $J$ is an $L$-quasicircle. We note that $J$ is a quasicircle if and only if there exists a constant $L'\geq 1$ such that for every pair of points $z,w\in J\setminus \{\infty\}$ there exists an arc $E\subset J\setminus \{\infty\}$ with endpoints $z,w$ such that
$$\diam_e(E)\leq L'|z-w|.$$
This equivalence is quantitative and follows from \cite{Bonk:uniformization}*{Proposition 4.4}; in fact it is shown that both of the above conditions are quantitatively equivalent to the condition that there exists $\delta>0 $ such that for each quadruple of distinct points $a,b,c,d\in J$ in cyclic order we have
\begin{align}\label{quasicircle:crossratio}
[a,b,c,d]\geq \delta>0.
\end{align}
If $J$ is a quasicircle and $\infty\in J$, we call the set $J\setminus \{\infty\}\subset \C$ an \textit{unbounded quasicircle}. The next result is elementary and follows from the argument in the proof of \cite{Ntalampekos:schottky}*{Lemma 3.2}.
\begin{lemma}\label{lemma:quasicircle_strip}
Let $J$ be an unbounded quasicircle that is contained in a strip $\{z\in \C: a<\im(z)<b\}$. Then $J$ separates the boundary components of the strip.
\end{lemma}

The following deep result of Ahlfors provides the link between quasidisks and quasicircles. 

\begin{theorem}[Ahlfors \cite{Ahlfors:reflections}]\label{prop:quasidisk_quasicircle}
Let $U\subset \widehat \C$ be a Jordan region and $J=\partial U$. Then $U$ is a quasidisk if and only if $J$ is a quasicircle, quantitatively.
\end{theorem}

A proof can be found in \cite{Ahlfors:qc}*{Section IV.D} or \cite{LehtoVirtanen:quasiconformal}*{Theorem II.8.6} or \cite{Pommerenke:conformal}*{Section 5}. Quasicircles are quasiconformally removable in the sense that if a homeomorphism of $\widehat \C$ is quasiconformal in the complement of a quasicircle $J$ then it is quasiconformal on all of $\widehat \C$; see \cite{LehtoVirtanen:quasiconformal}*{Theorem I.8.3}. 

A \textit{quasiconformal reflection} along a Jordan curve $J\subset \widehat \C$ is an orientation-reversing quasiconformal map $f\colon \widehat \C\to \widehat \C$ such that $f(z)=z$ for each $z\in J$ and $f$ interchanges the complementary components of $J$. Each quasidisk admits a quasiconformal reflection; see the discussion in \cite{Ahlfors:qc}*{Section IV.D}.

\begin{proposition}\label{proposition:extension_by_reflection}
Let $U,V\subset \widehat \C$ be quasidisks and $f\colon U\to V$ be a quasiconformal homeomorphism. Then there exists an extension of $f$ to a quasiconformal homeomorphism of $\widehat \C$, quantitatively.
\end{proposition}
\begin{proof}
The map $f$ extends to a homeomorphism of the closures $\bar U,\bar V$; see \cite{LehtoVirtanen:quasiconformal}*{Theorem I.8.2}. The quasidisks $U,V$ admit quasiconformal reflections $g,h$, respectively. The map $F$ defined by $f$ on $\bar U$ and by $h\circ f\circ g$ on $\widehat \C\setminus \bar U$ is a homeomorphism of $\widehat \C$ that is quasiconformal in $U$ and in $\widehat \C\setminus \bar U$. The quasiconformal removability of quasicircles implies that $F$ is quasiconformal on $\widehat \C$.  
\end{proof}

We recall the celebrated Beurling--Ahlfors extension theorem.

\begin{theorem}[Beurling--Ahlfors \cite{BeurlingAhlfors:extension}]\label{theorem:beurling_ahlfors}
Let $f\colon \R\to \R$ be an increasing quasisymmetric homeomorphism. Then there exists an extension of $f$ to a quasiconformal homeomorphism of the upper half-plane $\UHP$, quantitatively.
\end{theorem}

This result readily yields an extension result for quasisymmetric maps of the unit circle $\partial \D$. 

\begin{corollary}\label{corollary:beurling_ahlfors}
Let $f\colon \partial \D\to \partial \D$ be an orientation-preserving homeomorphism.
\begin{enumerate}[label=\normalfont(\arabic*)]
\item\label{corollary:beurling_ahlfors:1} If $f$ is quasi-M\"obius, then there exists an extension of $f$ to a quasiconformal homeomorphism of $\D$, quantitatively.
\item\label{corollary:beurling_ahlfors:2} If $f$ quasisymmetric, then there exists an extension of $f$ to a quasisymmetric homeomorphism of $\D$ that fixes the point $0$, quantitatively.
\end{enumerate}
\end{corollary}

\begin{proof}
For \ref{corollary:beurling_ahlfors:1} suppose that $f$ is $\eta$-quasi-M\"obius. By composing $f$ with a rotation, we may assume that it fixes the point $1$. Consider a M\"obius transformation $\phi\colon \bar\D\to \bar \UHP\cup \{\infty\}$ that maps $1$ to $\infty$ and $\partial \D\setminus \{1\}$ onto $\R$. The map $\phi\circ f\circ \phi^{-1}\colon \R\to \R$ is an increasing $\eta$-quasi-M\"obius homeomorphism. In fact, it is $\eta$-quasisymmetric by property \ref{q:qm_qs}. Therefore, by Theorem \ref{theorem:beurling_ahlfors} this map extends to a quasiconformal homeomorphism $F$ of $\UHP$, quantitatively. The map $G=\phi^{-1}\circ F\circ \phi$ is the quasiconformal extension requested in part \ref{corollary:beurling_ahlfors:1}. 

If $f\colon \partial \D\to \partial \D$ is $\eta$-quasisymmetric as in \ref{corollary:beurling_ahlfors:2}, then for the points $z_1=1$, $z_2=i$, $z_3=-1$ we have $|f(z_i)-f(z_j)|\simeq_\eta 1$ for $i\neq j$ by \eqref{definition:qs}. Consider the extension $G$ given by part \ref{corollary:beurling_ahlfors:1}. By \ref{q:qm_qc} and \ref{q:completion} the map $G$ is quasi-M\"obius in $\bar \D$. Therefore, by property \ref{q:qs_qm} the extension $G$ is quasisymmetric, quantitatively.  Let $A\subset \D$ be the radial segment from $G(0)$ to $\partial \D$. Since 
$$\frac{1}{2}\leq \frac{\diam(G^{-1}(A))}{\diam (\D)}\leq 1,$$
by \eqref{definition:qs} we conclude that $\diam(A)\simeq_\eta 1$. Thus,  $|G(0)|\leq r(\eta)<1$. We may now postcompose $G$ with an $L(\eta)$-bi-Lipschitz map of $\C$ that maps $G(0)$ to $0$ and is the identity map for $|z|\geq  \frac{1+r(\eta)}{2}$. By property \ref{q:bilip}, the composition remains quasisymmetric, quantitatively.
\end{proof}

\begin{corollary}\label{corollary:beurling_ahlfors_embedding}
Let $f\colon A\to \C$ be a quasisymmetric embedding, where either $A=\R$ or $A=\partial \D$ and $f$ is orientation-preserving. Then there exists an extension of $f$ to a quasiconformal homeomorphism of $\C$, quantitatively. 
\end{corollary}
\begin{proof}
We prove the result in the case that $A=\partial \D$. The other case is similar. The assumption that $f$ is quasisymmetric implies that $f(\partial \D)$ is a quasicircle (see \cite{Pommerenke:conformal}*{Proposition 5.10}). This quasicircle bounds a quasidisk $U\subset \C$ by Theorem \ref{prop:quasidisk_quasicircle}. Let $\phi$ be a quasiconformal map of $\widehat \C$ that maps $U$ onto $\D$. By postcomposing with a M\"obius transformation of $\D$ we assume that $\phi$ fixes $\infty$. Thus, $\phi\colon \C\to \C$ is quasiconformal and quasisymmetric by \ref{q:qc_qs}. The map $\phi\circ f\colon \partial \D\to \partial \D$ is quasisymmetric. By Corollary \ref{corollary:beurling_ahlfors}, the map $\phi\circ f$ extends to a quasiconformal homeomorphism $F$ of $\D$ that fixes $0$. By reflection (see \cite{LehtoVirtanen:quasiconformal}*{Section I.8.4}), we may obtain a quasiconformal homeomorphism $F$ of $\C$. Thus, $\phi^{-1}\circ F$ is the desired quasiconformal extension of $f$. 
\end{proof}

\subsection{Uniform domains}\label{section:uniform}

A set $A\subset \R^n$ is \textit{quasiconvex} if there exists $L\geq 1$ such that for every $z,w\in A$ there exists a curve $\gamma\colon [a,b]\to A$ with $\gamma(a)=z$, $\gamma(b)=w$, and $\ell_e(\gamma)\leq L|z-w|$. In this case we say that $A$ is $L$-quasiconvex. By the lower semicontinuity of length, observe that $\bar A$ is also $L$-quasiconvex.

A domain $U\subset \R^n$ is called \textit{uniform} if there exists a constant $L\geq 1 $ such that for every pair of points $z,w\in U$ there exists a curve $\gamma\colon [a,b]\to U$ with $\gamma(a)=z$, $\gamma(b)=w$, $\ell_e(\gamma)\leq L|z-w|$, and 
$$\min\{\ell_e(\gamma|_{[a,t]}),\ell_e(\gamma|_{[t,b]})\}\leq L \dist_e(\gamma(t), \partial U)\,\,\, \text{for each $t\in [a,b]$}.$$
In this case we say that $U$ is $L$-uniform. By definition, uniform domains are quasiconvex. This is the only property that we will need in this work.

The space $\R^n$, each ball in $\R^n$, and the exterior of each ball in $\R^n$ are uniform domains. It is a consequence of \cite{GehringOsgood:uniform}*{Corollary 3} uniformity is preserved under quasiconformal maps of $\R^n$. In particular, each quasidisk in $\C$ and its exterior are uniform domains. More generally, the following result is true. 

\begin{theorem}[\cite{Vaisala:quasimobius}*{Theorem 5.6}]\label{theorem:qm_uniform}
Let $U\subset \R^n$ be a uniform domain and $f\colon U\to \R^n$ be a quasi-M\"obius embedding. Then $f(U)$ is a uniform domain, quantitatively. 
\end{theorem}

A consequence of the quasiconvexity of quasidisks is that any point $z$ in the boundary of a quasidisk $U$ is \textit{rectifiably accessible}; that is, there exists a rectifiable curve $\gamma\colon [a,b]\to \bar U$ such that $\gamma([a,b))\subset U$ and $\gamma(b)=z$.

\begin{corollary}\label{cor:quasidisk:rectifiably_accessible}
Each point in the boundary of a quasidisk is rectifiably accessible. 
\end{corollary}

\subsection{Quasihyperbolic metric}

Let $U\subsetneq \R^n$ be a domain. We define the \textit{quasihyperbolic metric} on $U$ as
$$k_U(z,w) =\inf_\gamma \int_\gamma \frac{1}{\dist_e(\cdot,\partial U)}\,ds, $$
where the infimum is over all rectifiable curves $\gamma$ in $U$ with endpoints $z,w$. Each pair $z,w\in U$ can be joined with a \textit{quasihyperbolic geodesic} $\gamma$ in $U$, that is, 
$$k_U(z,w)=\int_\gamma \frac{1}{\dist_e(, \partial U)}\, ds= \ell_{k_U}(\gamma).$$
This implies that $(U,k_U)$ is $1$-quasiconvex; see \cite{GehringOsgood:uniform}*{Lemma 1}. The next lemma gives a convenient estimate of the quasihyperbolic metric. 

\begin{lemma}\label{lemma:qh_estimate}
Let $K\geq 1$. Let $U\subsetneq \R^n$ be a domain and $B\subset U$ be a $K$-quasiconvex set such that $\diam_e (B)\leq K{\dist_e (B,\partial U)}$. Then for every $z,w\in B$ we have
$$ C(K)^{-1} \frac{|z-w|}{\dist_e(B,\partial U)}\leq k_U(z,w) \leq C(K) \frac{|z-w|}{\dist_e(B,\partial U)}.$$
\end{lemma}
\begin{proof}
Let $z,w\in B$. By the quasiconvexity assumption, there exists a curve $\gamma$ in $B$ with endpoints $z,w$ such that $\ell_e(\gamma)\leq K |z-w|$. Thus, 
$$k_U(z,w) \leq \int_\gamma \frac{1}{\dist_e(\cdot,\partial U)}\, ds \leq \frac{\ell_e(\gamma)}{\dist_e(B,\partial U)} \leq K\frac{|z-w|}{\dist_e(B,\partial U)}.$$
For the reverse inequality, consider the ball $B_e(z,R)$, where $R=2^{-1}\dist_e(B,\partial U)$. Let $\gamma$ be a curve in $U$ that connects $z$ and $w$. If $\gamma$ is not contained in $B(z,R)$, then
$$\int_\gamma \frac{1}{\dist_e(\cdot,\partial U)} \, ds \geq \frac{R}{\dist_e(z,\partial U)+R} \geq \frac{R}{\diam_e (B)+\dist_e(B,\partial U)+R}\geq \frac{1}{2K+3}.$$ 
On the other hand, 
$$\frac{|z-w|}{\dist_e(B,\partial U)} \leq \frac{\diam_e(B)}{\dist_e(B,\partial U)}\leq K$$
so we have
\begin{align}\label{lemma:qh_estimate:1}
\int_\gamma \frac{1}{\dist(\cdot,\partial U)} \, ds \geq C(K)^{-1} \frac{|z-w|}{\dist_e(B,\partial U)}.
\end{align}
Now, suppose that $\gamma$ is contained in $B_e(z,R)$. Then
\begin{align*}
\int_\gamma \frac{1}{\dist_e(\cdot,\partial U)} \, ds &\geq  \frac{|z-w|}{\dist_e(z,\partial U)+R} \geq \frac{|z-w|}{\diam_e(B)+\dist_e (B,\partial U)+R}\\
&\geq \frac{2|z-w|}{(2K+3)\dist_e(B,\partial U)}.
\end{align*}
Therefore, \eqref{lemma:qh_estimate:1} holds in this case too with a different constant. Infimizing over all curves $\gamma$ connecting $z$ and $w$ gives the desired inequality.
\end{proof}

Gehring and Osgood \cite{GehringOsgood:uniform}*{Theorem 3} studied the distortion of the quasihyperbolic metric under quasiconformal maps.

\begin{theorem}[Gehring--Osgood]\label{theorem:gehring_osgood}
Let $K\geq 1$. Let $f\colon U\to V$ be a $K$-quasi\-con\-for\-mal homeomorphism between domains $U,V\subsetneq \R^n$. If $\alpha=K^{1/(1-n)}$, we have
$$k_{V}(f(z),f(w))\leq C(n,K) \max \{ k_U(z,w), k_U(z,w)^\alpha\}\,\,\, \text{for all $z,w\in U$}.$$
\end{theorem}

Combining the above theorem with the fact that the inverse of a quasiconformal map is quasiconformal, we conclude that in large scale the map $f$ is bi-Lipschitz in the respective quasihyperbolic metrics. We prove that if one takes into account small scales as well, then $f$ is quasisymmetric in the quasihyperbolic metrics. 

\begin{theorem}\label{theorem:qc_qs_qh}
Let $K\geq 1$. Let $f\colon U\to V$ be a $K$-quasi\-con\-for\-mal homeomorphism between domains $U,V\subsetneq \R^n$. Then the map
$$f\colon (U,k_U)\to (V,k_V)$$
is quasisymmetric, quantitatively, depending only on $n,K$.
\end{theorem}

This result generalizes \cite{AckermanFletcher:qc_qs}*{Theorem 1.3} and \cite{FletcherHahn:uqc_homogeneous}*{Theorem 1.2}, where additional assumptions are imposed on the domain $U$.

\begin{proof}
Suppose that $f$ is $K$-quasiconformal for some $K\geq 1$. A result of V\"ais\"al\"a \cite{Vaisala:free_quasiworld}*{Theorem 6.6} implies that every weakly quasisymmetric map between quasiconvex spaces is quasisymmetric. Therefore, it suffices to show that the map $f\colon (U,k_U)\to (V,k_V)$ is \textit{weakly quasisymmetric}; that is, there exists a constant $H=H(n,K)>0$ such that for every triple of points $z_1,z_2,z_3\in U$ with $k_U(z_1,z_2)\leq k_U(z_1,z_3)$ we have
\begin{align*}
k_V(z_1',z_2')\leq H k_V(z_1',z_3'),
\end{align*}
where we use the notation $z'=f(z)$ for $z\in U$.

Suppose that $k_U(z_1,z_2)\leq k_U(z_1,z_3)$. If $k_U(z_1,z_2)\geq 1/3$, then $k_U(z_1,z_3)\geq 1/3$ and by Theorem \ref{theorem:gehring_osgood} (applied to $f^{-1}$) we have 
$$k_{V}(z_1',z_2')\gtrsim_{n,K} 1\,\,\, \text{and}\,\,\,k_{V}(z_1',z_3')\gtrsim_{n,K} 1.$$
Thus, by Theorem \ref{theorem:gehring_osgood} we conclude that 
$$k_{V}(z_1',z_2')\simeq_{n,K} k_U(z_1,z_2)\,\,\, \text{and}\,\,\,k_{V}(z_1',z_3')\simeq_{n,K} k_U(z_1,z_3).$$
Therefore, $k_V(z_1',z_2')\lesssim_{n,K} k_V(z_1',z_3')$. Next, suppose that $k_U(z_1,z_2)< 1/3\leq k_U(z_1,z_3)$. By Theorem \ref{theorem:gehring_osgood} we obtain $k_V(z_1',z_2')\lesssim_{n,K} 1$ and as above we have $k_V(z_1',z_3')\gtrsim_{n,K}1$. Therefore, $k_V(z_1',z_2')\lesssim_{n,K} k_V(z_1',z_3')$.

Finally, we suppose that $k_U(z_1,z_2)\leq k_U(z_1,z_3)< 1/3$ and let $R=2^{-1}\cdot \dist_e(z_1,\partial U)$. If $z\in U$ and $|z-z_1|\geq R$, then for each curve $\gamma$ in $U$ that connects $z_1,z$ we have
$$ \int_{\gamma} \frac{1}{\dist_e(\cdot,\partial U)}\, ds \geq  \frac{R}{\dist_e(z_1,\partial U)+R}=\frac{1}{3}$$
so $k_U(z_1,z) \geq 1/3$. We conclude that $z_2,z_3\in B\coloneqq B_e(z_1,R)$. By Lemma \ref{lemma:qh_estimate},
$$k_U(z_1,z_i) \simeq \frac{|z_1-z_i|}{\dist_e(B,\partial U)}$$
for $i=2,3$. We conclude that $|z_1-z_2|\lesssim |z_1-z_3|$.

By property \ref{q:qc_qs}, the map $f|_{B}$ is quasisymmetric in the Euclidean metric, quantitatively, depending on $n,K$. In particular, $|z_1'-z_2'|\lesssim_{n,K} |z_1'-z_3'|$. Moreover, the image $B'=f(B)$ is a uniform domain by Theorem \ref{theorem:qm_uniform}, so it is quasiconvex, quantitatively. Finally, we have $\diam_e( B')\lesssim_{n,K} \dist_e(B',\partial V)$; see \cite{Heinonen:metric}*{Exercise 11.18} or \cite{Gehring:Rings}*{Theorem 11}. By Lemma \ref{lemma:qh_estimate}, we have
$$k_V(z_1',z_i') \simeq_{n,K} \frac{|z_1'-z_i'|}{\dist_e(B',\partial V)}$$
for $i=2,3$. Therefore, $k_{V}(z_1',z_2')\lesssim_{n,K} k_{V}(z_1',z_3')$. This completes the proof.
\end{proof}

\subsection{Hyperbolic metric}

Let $\Omega\subset \widehat{\C}$ be a domain whose boundary contains at least two points. Let $\rho_\Omega$ be the hyperbolic density in $\Omega$; see \cite{BeardonMinda:hyperbolic} for the definition. The \textit{hyperbolic distance} of two points $z,w\in \Omega$ is defined as 
$$h_\Omega(z,w)=\inf_\gamma \left\{ \int_\gamma \rho_\Omega(\zeta)\,|d\zeta| \right\},$$
where the infimum is taken over all rectifiable curves $\gamma$ in $\Omega$ with endpoints $z,w$.   If $\Omega$ is a half-plane then  
$$\rho_\Omega(z)= \frac{1}{\dist_e(z,\partial \Omega)},\,\,\,z\in \Omega$$
and if $\Omega=\D(0,r)$ for some $r>0$, then
$$\rho_\Omega(z)= \frac{2r}{r^2-|z|^2}\,\,\, \text{for}\,\,\,z\in \Omega\quad \text{and}\quad \rho_{\Omega^*}(z)= \frac{2r}{|z|^2-|r|^2}\,\,\, \text{for}\,\,\,z\in \Omega^*.$$

\begin{theorem}\label{theorem:hyperbolic_distance}
Let $\Omega\subset \C$ be a domain.
\begin{enumerate}[label=\normalfont(\arabic*)]
\item For every $z\in \Omega$ we have $\rho_\Omega(z) \leq 2^{-1}{\dist_e(z,\partial \Omega)}$.
\item If $\Omega$ is simply connected, then for every $z\in \Omega$ we have 
$$\frac{1}{2\dist_e(z,\partial \Omega)} \leq \rho_\Omega(z)\leq  \frac{2}{\dist_e(z,\partial \Omega)}.$$
\item If $\C\setminus \Omega$ has exactly two components $E,F$ one of which is unbounded and $\Delta_e(E,F)\leq K$ for some $K\geq 1$, then for every $z\in \Omega$ we have
$$\frac{1}{C(K)\dist_e(z,\partial \Omega)} \leq \rho_\Omega(z) \leq  \frac{2}{\dist_e(z,\partial \Omega)}.$$
\end{enumerate}
\end{theorem}

\begin{proof}
The first two parts are discussed in \cite{BeardonMinda:hyperbolic}*{Section 8} and in \cite{BeardonPommerenke:hyperbolic}*{(2.1), (2.2)}. The left inequality in the third part follows from \cite{BeardonPommerenke:hyperbolic}*{Theorem 1}, as soon as we prove that the quantity
$$\beta(z)=\inf \left\{\left|\log\frac{|z-a|}{|a-b|}\right|: a,b\in \partial \Omega, \, a\neq b,\, |z-a|=\dist_e(a,\partial \Omega)\right\} $$
is uniformly bounded in $\Omega$, depending only on $K$.  Suppose that $E$ is an unbounded component of $\C\setminus \Omega$. Let $z\in \Omega$ and suppose that $\dist_e(z,\partial \Omega)=\dist_e(z,E)$. Let $a\in E$ such that $|z-a|=\dist_e(z,E)$. Since $E$ is connected and it is not contained in $\D(a,|z-a|)$, we can find a point $b\in  E$ such that $|a-b|=|z-a|$. Thus, $\beta(z)=0$. 

Now, suppose that $\dist_e(z,\partial \Omega)=\dist_e(z,F)$ and let $a\in F$ such that $|z-a|=\dist_e(z,F)\eqqcolon r$. Consider the annulus $A=\mathbb A(a;r, (2K+2)r)$. Suppose that $\partial \Omega$ does not intersect the annulus $A$, so $F\subset \bar \D(a,r)$. Since $E$ is unbounded, the annulus $A$ separates $E$ and $F$. Thus, 
$$\Delta_e(E,F) \geq \frac{(2K+2)r-r}{2r}= \frac{2K+1}{2}>K,$$
a contradiction. Thus, there exists a point $b\in \partial \Omega\cap A$, so
$$\frac{1}{2K+2}\leq \frac{|z-a|}{|a-b|}\leq 1.$$
This implies that $\beta(z)\leq \log(2K+2)$.
\end{proof}

\section{Relative hyperbolic metric}

\subsection{Definition}\label{section:relative:definition}
Recall that if $U\subset \widehat\C$ is an open set, we set $U^*=\widehat{\C}\setminus \bar  U$. Let $U,V\subset \widehat \C$ be disjoint Jordan regions such that $\bar U\cap \bar V$ is either empty or contains only one point; in particular, $\widehat \C\setminus (\bar U\cup \bar V)$ is connected. The {relative hyperbolic metric} corresponding to the pair $(V,U)$ is defined as
$$d_{V,U}(z,w)= \inf_\gamma\{\ell_{h_{U^*}}(\gamma)\},\quad z,w\in \partial V\setminus \bar U, $$
where the infimum is over all rectifiable curves $\gamma\colon [a,b]\to U^*\setminus V$ such that $\gamma(a)=z$, $\gamma(b)=w$. We remark that in this definition we adhere to the hyperbolic metric rather than the quasihyperbolic one, since we are working with subsets of the sphere.

If $V$ is a quasidisk, then $\widehat \C\setminus \bar V$ is also a quasidisk. By Corollary \ref{cor:quasidisk:rectifiably_accessible}, every pair of points $z,w\in \partial V\setminus \bar U$ can be connected by a rectifiable curve $\gamma$ as in the definition of $d_{V,U}(z,w)$. Therefore, $d_{V,U}(z,w)<\infty$ for every pair $z,w\in \partial V\setminus \bar U$, which implies that $d_{V,U}$ is a metric on $\partial V\setminus \bar U$. 

Observe that if $\phi\colon U^*\to \widehat \C$ is a conformal embedding that extends to a homeomorphism of the sphere,  then 
$$d_{\phi(V),\phi(U)}(\phi(z),\phi(w))= d_{V,U}(z,w) \,\,\, \text{for every $z,w\in \partial V\setminus \bar U$.}$$

\subsection{Basic examples}\label{section:relative:examples}
We estimate the metric $d_{V,U}$ in some important cases.

\begin{lemma}[Parallel lines and concentric circles]\label{lemma:parallel}
Let $U,V\subset \widehat \C$ be disjoint disks. 
\begin{enumerate}[label=\normalfont(\arabic*)]
\item\label{lemma:parallel:1}  If $U,V\subset \C$ are half-planes with $\bar U\cap \bar V\cap \C=\emptyset$, then, for every $z,w\in \partial V\cap \C$,
$$d_{V,U}(z,w)=\frac{|z-w|}{\dist_e(\partial V,\partial U)}.$$
\item\label{lemma:parallel:2} If $U^*,V\subset \C$ are concentric Euclidean disks, then, for every $z,w\in \partial V$,
$$\frac{|z-w|}{\dist_e(\partial V,\partial U)}\leq d_{V,U}(z,w) \leq \pi \frac{|z-w|}{\dist_e(\partial V,\partial U)}$$
and for every $z,w\in \partial U$,
$$\frac{\diam_e(\partial V)|z-w|}{\diam_e(\partial U)\dist_e(\partial V,\partial U)}\leq d_{U,V}(z,w) \leq  \frac{\pi \diam_e(\partial V)|z-w|}{\diam_e(\partial U)\dist_e(\partial V,\partial U)}.$$
\end{enumerate}
\end{lemma}
\begin{proof}
In \ref{lemma:parallel:1}, for any curve $\gamma$ in $U^*\setminus V$ connecting $z,w\in \partial V\cap \C$ we have
$$\ell_{h_{U^*}}(\gamma)= \int_{\gamma} \frac{|d\zeta| }{\dist_e(\zeta,\partial U)} \geq \frac{\ell_e(\gamma)}{\dist_e(\partial U,\partial V)} \geq \frac{|z-w|}{\dist_e(\partial U,\partial V)}= \ell_{h_{U^*}}([z,w]),$$
so $d_{V,U}(z,w)=\ell_{h_{U^*}}([z,w])$, as desired. 

For \ref{lemma:parallel:2}, let $U=\widehat \C\setminus \D(0,R)$ and $V=\D(0,r)$, $r<R$, so then $U^*=\D(0,R)$. For any curve $\gamma$ in $U^*\setminus V$ connecting $z,w\in \partial V$ we have
$$\ell_{h_{U^*}}(\gamma) = \int_\gamma \frac{2R|d\zeta|}{R^2-|\zeta|^2}\geq \frac{2R\ell_e(\gamma)}{R^2-r^2}\geq \frac{2R\ell_e([z,w])}{R^2-r^2}=\ell_{h_{U^*}}([z,w]),$$
where $[z,w]$ is the shortest arc of $\partial V$ with endpoints $z,w$. Thus,
$$d_{V,U}(z,w)=\ell_{h_{h^*}}([z,w])=\frac{2R}{R+r} \frac{\ell_e([z,w])}{\dist_e(\partial U,\partial V)}.$$
Finally, observe that $|z-w|\leq \ell_e([z,w])\leq \frac{\pi}{2} |z-w|$. This proves the inequalities for $d_{V,U}$. 

We now discuss the inequality for $d_{U,V}$. Consider the conformal map $\phi(z)=1/z$ in $\widehat \C$. For $z,w\in \partial U$, by the above we have
\begin{align*}
d_{U,V}(z,w)=d_{\phi(U),\phi(V)}(z^{-1},w^{-1})=\frac{2r^{-1}\ell_e([z^{-1},w^{-1}])}{r^{-2}-R^{-2}},
\end{align*}
where $[z^{-1},w^{-1}]$ is the shortest arc of $\phi(\partial U)$ that connects $z^{-1},w^{-1}$. Note that 
$$\frac{|z-w|}{R^2}= \frac{|z-w|}{|zw|}=\left|\frac{1}{z}-\frac{1}{w}\right|\leq \ell_e([z^{-1},w^{-1}]) \leq \frac{\pi}{2} \frac{|z-w|}{R^2}.$$
Thus,
$$\frac{2r|z-w|}{R^2-r^2}\leq d_{U,V}(z,w) \leq \pi \frac{r|z-w|}{R^2-r^2}.$$
From this the conclusion follows.
\end{proof}

\begin{lemma}[Quasicircles near parallel lines]\label{lemma:quasicircle_parallel}
Let $K\geq 1$ and $L>0$. Let $U,V\subset \C$ be disjoint unbounded $K$-quasidisks such that $\partial V\cap \C\subset \{z\in \C: -LK\leq  \im(z)\leq 0 \}$ and $\partial U\cap \C \subset \{z\in \C: L\leq \im(z)\leq  L+LK\}$. Then for every $z,w\in \partial V\cap \C$ we have
$$ C(K)^{-1}\frac{|z-w|}{\dist_e(\partial U,\partial V)}\leq d_{V,U}(z,w)\leq C(K)\frac{|z-w|}{\dist_e(\partial U,\partial V)}.$$
\end{lemma}

See Figure \ref{figure:parallel} for an illustration.
\begin{proof}
Since $U$ is an unbounded quasidisk whose boundary is contained in a horizontal strip, the boundary $\partial U$ separates the half-plane $\im(z)>L+LK$ from the half-plane $\im(z)<L$; this follows from Lemma \ref{lemma:quasicircle_strip}. Similarly the boundary $\partial V$ separates the half-plane $\im(z)>0$ from the half-plane $\im(z)<-LK$. The quasidisks $U$ and $V$ are disjoint, so the half-plane $\im(z)<-LK$ is contained in $V$ and the half-plane $\im(z)>L+LK$ is contained in $U$. Thus, 
$$\C\setminus (U\cup V)\subset\{z\in \C: -LK\leq \im(z)\leq L+LK\}$$
and each vertical line intersects $\partial U$ and $\partial V$. As a consequence, if $z\in U^*\setminus V$, then
\begin{align}\label{lemma:quasicircle_parallel:1}
\dist_e(z,\partial U)\leq (2K+1)L \quad \text{and} \quad \dist_e(\partial U,\partial V)\leq (2K+1)L.
\end{align} 
Let $\gamma$ be an arbitrary curve in $U^*\setminus V$ that connects $z,w\in \partial V$. By Theorem \ref{theorem:hyperbolic_distance} and \eqref{lemma:quasicircle_parallel:1} we have
$$\int_{\gamma} \rho_{U^*}(\zeta)|d\zeta| \geq \frac{1}{2}\int_{\gamma} \frac{|d\zeta|}{\dist_e(\zeta,\partial U)}\geq \frac{\ell_e(\gamma)}{2(2K+1)L}  \geq \frac{|z-w|}{2(2K+1)L} .$$
Given that $\dist_e(\partial U,\partial V)\geq L$, we have
$$d_{V,U}(z,w)\geq \frac{|z-w|}{2(2K+1)L}\geq \frac{|z-w|}{2(2K+1)\dist_e(\partial U,\partial V)}.$$
On the other hand, by the quasiconvexity of $\C\setminus  V$, there exists a curve $\gamma$ in $\C\setminus V$ that connects $z,w\in \partial V$ and $\ell_e(\gamma)\leq C(K)|z-w|$. By replacing the parts of $\gamma$ that lie in $\mathbb H$ with segments on the real line, we may assume, in addition, that $\gamma$ is contained in  $\{\zeta\in \C\setminus V: \im(\zeta)\leq 0\}$. Each point of $\gamma$ has distance at least $L$ from $\partial U$. By Theorem \ref{theorem:hyperbolic_distance} we have
$$d_{V,U}(z,w) \leq \int_{\gamma} \rho_{U^*}(\zeta)|d\zeta|\leq \int_{\gamma}\frac{2|d\zeta|}{\dist_e(\zeta,\partial U)}\leq \frac{2C(K)|z-w|}{L}.$$
Since $\dist_e(\partial U,\partial V)\leq (2K+1)L$, we obtain the desired inequality.
\end{proof}

\begin{lemma}[Wormhole]\label{lemma:wormhole}
Let $K\geq 1$. Let $U,V\subset \C$ be unbounded $K$-quasidisks. Suppose that there exists a $K$-chord-arc curve $J\subset \C\setminus (\bar U\cup \bar V)$ that separates $\partial U$ from $\partial V$ and satisfies 
$$N_{1/K}(J)\subset \C\setminus (\bar U\cup \bar V)\subset N_K(J).$$
Then for every $z,w\in \partial V\cap \C$ we have
$$C(K)^{-1}|z-w|\leq d_{V,U}(z,w)\leq C(K)|z-w|.$$
\end{lemma}
See Figure \ref{figure:wormhole} for an illustration. Here an unbounded Jordan curve $J$ in $\C$ is a \textit{chord-arc curve} if there exists $K\geq 1$ such that for every $z,w\in J$ the length of the arc of $J$ between $z$ and $w$ is bounded above by $K|z-w|$. Also, for $r>0$ and $E\subset \C$ we denote by $N_r(E)$ the Euclidean open $r$-neighborhood of the set $E$. The lemma can be reduced to Lemma \ref{lemma:quasicircle_parallel} by using the fact that a chord-arc curve can be mapped to a line with a bi-Lipschitz map of $\C$  \cite{JerisonKenig:chordarc}*{Proposition 1.13}.

\begin{lemma}[Quasicircles near concentric circles]\label{lemma:quasicircle_concentric}
Let $K\geq 1$ and $R>r>0$. Let $V$ be a $K$-quasidisk such that $0\in V$ and 
$$\partial V\subset \{z\in \C: \max\{r-K(R-r),K^{-1}r\} \leq |z|\leq r\}$$
and let $U\subset\widehat \C$ be a $K$-quasidisk such that $\infty \in U$ and $\partial U \subset \{z\in \C: R\leq |z|\leq R+K(R-r)\}$. Then for every $z,w\in \partial V$ we have
$$C(K)^{-1}\frac{|z-w|}{\dist_e(\partial U,\partial V)} \leq d_{V,U}(z,w) \leq  C(K) \frac{|z-w|}{\dist_e(\partial U,\partial V)},$$
and for every $z,w\in \partial U$ we have
$$C(K)^{-1}\frac{\diam_e(\partial V)|z-w|}{\diam_e(\partial U)\dist_e(\partial U,\partial V)} \leq d_{U,V}(z,w) \leq  C(K) \frac{\diam_e(\partial V)|z-w|}{\diam_e(\partial U)\dist_e(\partial U,\partial V)}.$$
\end{lemma}

See Figure \ref{figure:concentric} for an illustration.

\begin{proof}
Since $\infty\in  U$, we have $\D(0,R+K(R-r))^*\subset U$ and $\D(0,R)$ lies outside $U$. Similarly, since $0\in V$, we have $\{z\in \C: |z|<r-K(R-r)\}\subset V$ and $\D(0,r)^*$ lies outside $V$. We conclude that
$$\C\setminus (U\cup V)\subset \{z\in \C: r-K(R-r)\leq |z| \leq R+K(R-r)\}$$
and each line through $0$ intersects $\partial U$ and $\partial V$. Let $\gamma$ be an arbitrary curve in $U^*\setminus V$ that connects $z,w\in \partial V$. Each point of $\gamma$ has distance at most $(2K+1)(R-r)$ from $\partial U$.  By Theorem \ref{theorem:hyperbolic_distance},
$$\int_{\gamma}\rho_{U^*}(\zeta)|d\zeta| \geq \frac{1}{2}\int_{\gamma} \frac{|d\zeta|}{\dist_e(\zeta,\partial U)}\geq \frac{\ell_e(\gamma)}{(2K+1)(R-r)} \geq  \frac{|z-w|}{(2K+1)(R-r)}.$$
Given that  $\dist_e(\partial U,\partial V)\geq R-r$, we have
$$d_{V,U}(z,w)\geq \frac{|z-w|}{(2K+1)(R-r)}\geq \frac{|z-w|}{(2K+1)\dist_e(\partial U,\partial V)}.$$
On the other hand, by the quasiconvexity of $\C\setminus V$, there exists a curve $\gamma$ in $\bar \D(0,r)\setminus V$ that connects $z,w\in \partial V$ and satisfies $\ell_e(\gamma)\leq C(K)|z-w|$. Each point of this curve has distance at least $R-r$ from $\partial U$. Thus, by Theorem \ref{theorem:hyperbolic_distance}
$$d_{V,U}(z,w)\leq \int_{\gamma}\rho_{U^*}(\zeta)|d\zeta|\leq \int_{\gamma} \frac{2|d\zeta|}{\dist_e(\zeta,\partial U)}\leq\frac{2C(K)|z-w|}{R-r}.$$
Since $\dist_e(\partial U,\partial V)\leq (2K+1)(R-r)$, we obtain the claimed inequalities for $d_{V,U}$.

Now we prove the inequalities for $d_{U,V}$. Let $r_1=\dist_e(0,\partial V)>0$ and let $K_1>0$ such that $r_1=r-K_1(R-r)$. Note that $r_1\geq r-K(R-r)$, so $K_1\leq K$.  By the comparison principle for hyperbolic metrics \cite{BeardonMinda:hyperbolic}*{Theorem 8.1}, we have 
$$\rho_{V^*}(z)\geq \rho_{\D(0,r_1)^*}(z)= \frac{2r_1}{|z|^2-r_1^2}, \quad z\in V^*.$$ 
Let $\gamma$ be an arbitrary curve in $V^*\setminus U$ that connects points $z,w\in \partial U$. We have
\begin{align*}
\int_{\gamma}\rho_{V^*}(\zeta)\, |d\zeta| &\geq \int_\gamma \frac{2r_1|d\zeta|}{|\zeta|^2-r_1^2}\geq \frac{2r_1|z-w|}{(R+K(R-r))^2-r_1^2}\\
&= \frac{2r_1|z-w|}{(R-r_1+K(R-r))(R+r_1+K(R-r))}\\
&\geq \frac{2r_1|z-w|}{(K+K_1+1)(R-r)(R+R+KR)}\\
&\geq \frac{2r_1|z-w|}{(2K+1)(K+2)R(R-r)}.
\end{align*}
Note that $r_1\geq \frac{1}{K}r\geq \frac{1}{2K}\diam_e(\partial V)$, $R\leq \frac{1}{2}\diam_e(\partial U)$, and $R-r\leq \dist_e(\partial U,\partial V)$. This proves that 
$$d_{U,V}(z,w) \gtrsim_K \frac{\diam_e(\partial V)|z-w|}{\diam_e(\partial U)\dist_e(\partial U,\partial V)}.$$
On the other hand, by the quasiconvexity of $\C\setminus U$, there exists a curve $\gamma$ in $(\C\setminus \D(0,R))\setminus U$ that connects $z,w\in \partial U$ and satisfies $\ell_e(\gamma)\leq C(K)|z-w|$. By the comparison principle for hyperbolic metrics, we have
\begin{align*}
d_{U,V}(z,w)&\leq \int_{\gamma}\rho_{V^*}(\zeta)\, |d\zeta| \leq \int_\gamma \rho_{\D(0,r)^*}(\zeta)\, |d\zeta|= \int_\gamma \frac{2r|d\zeta|}{|\zeta|^2-r^2}\\
&\leq \frac{2r\ell_e(\gamma)}{R^2-r^2}\leq  \frac{2rC(K)|z-w|}{R(R-r)}.
\end{align*}
To complete the proof, note that $r\leq K r_1 \leq \frac{K}{2}\diam_e(\partial V)$, $\diam_e(\partial U) \leq 2(R+K(R-r))\leq 2(K+1)R$, and $\dist_e(\partial U,\partial V)\leq (2K+1)(R-r)$.
\end{proof}

\begin{lemma}[Relatively separated quasidisks]\label{lemma:relatively_separated}
Let $K\geq 1$. Let $U,V\subset \widehat \C$ be $K$-quasidisks such that $\bar U\cap \bar V=\emptyset$ and $\infty\in  U$. If $\Delta_{e}(\partial U,\partial V)>K^{-1}$, then for every $z,w\in \partial V$ we have
$$C(K)^{-1}\frac{|z-w|}{\dist_e(\partial U,\partial V)}\leq  d_{V,U}(z,w)\leq C(K)\frac{|z-w|}{\dist_e(\partial U,\partial V)}.$$ 
\end{lemma}

\begin{proof}
Since $\infty \in U$, we have $\diam_e(\partial V)<\diam_e(\partial U)$. 
Note that $\bar V\subset U^*$, $\bar V$ is quasiconvex, and $\diam_e(\bar V)< K\dist_e(\bar V,\partial U^*)$. Thus, Lemma \ref{lemma:qh_estimate} can be applied and we have
$$k_{U^*}(z,w) \geq C(K)^{-1}\frac{|z-w|}{\dist_e(\partial U,\partial V)}$$
for $z,w\in \partial V$. Combining this with Theorem \ref{theorem:hyperbolic_distance}, for $z,w\in \partial V$ we have
$$d_{V,U}(z,w)\geq h_{U^*}(z,w) \geq \frac{1}{2} k_{U^*}(z,w)\geq \frac{1}{2C(K)}\frac{|z-w|}{\dist_e(\partial U,\partial V)}.$$

Next, we show an inequality in the reverse direction. The set $\C\setminus V$ is $C_1(K)$-quasiconvex. Let $a\in \partial V$ and consider the ball $\D(a, r(a))$, where 
$$r(a)=\frac{\dist_e(a,\partial U)}{2(2C_1(K)+1)}=C_2(K) \dist_e(a,\partial U).$$
Suppose that $z,w\in \D(a,r(a))\cap \partial V$. Then there exists a curve $\gamma$ of length at most $C_1(K)|z-w|$ that connects $z,w$ and lies in $\D(a,2^{-1}\dist_e(a,\partial U))\setminus V$. Thus, each point $\zeta$ on the curve $\gamma$ satisfies
$$\dist_e(\zeta,\partial U)\geq 2^{-1}\dist_e(a,\partial U)$$
and we have
$$d_{V,U}(z,w)\leq \int_\gamma \rho_{U^*}(\zeta)\,|d\zeta|\leq \int_\gamma \frac{2|d\zeta|}{\dist_e(\zeta,\partial U)}\leq  \frac{2C_1(K)|z-w|}{2^{-1}\dist_e(a,\partial U)}\leq  \frac{C_3(K)|z-w|}{\dist_e(\partial U,\partial V)}.$$

Note that for each $a\in \partial V$ we have
$$r(a)=C_2(K)\dist_e(a,\partial U) \geq C_2(K)\dist_e(\partial U,\partial V) \geq K^{-1}C_2(K)\diam_e(\partial V).$$
Hence, $\partial V$ can be covered by $N\leq N(K)$ balls $\D(a_i,r(a_i))$, $a_i\in \partial V$, $i\in \{1,\dots,N\}$. Suppose that $z,w\in \partial V$ and $z,w\notin \D(a,r(a))$ for any $a\in \partial V$. Then 
$$|z-w|\geq r(z)\geq K^{-1}C_2(K)\diam_e(\partial V).$$
Consider a chain of points $z=z_0,z_1,\dots,z_m=w$ in $\partial V$, where $m\leq N$, such that consecutive points $z_{i-1},z_i$ lie in the same ball $\D(a_j,r(a_j))$. By the previous case,
\begin{align*}
d_{V,U}(z,w)&\leq \sum_{i=1}^m d_{V,U}(z_{i-1},z_i)\leq \frac{N(K)C_3(K)\diam_e(\partial V)}{\dist_e(\partial U,\partial V)}\\
&\leq \frac{N(K)C_3(K)KC_2(K)^{-1}|z-w|}{\dist_e(\partial U,\partial V)}.
\end{align*}
This completes the proof.
\end{proof}

We observe that in the above special cases the relative hyperbolic metric $d_{V,U}$ is quasisymmetric to the Euclidean metric via the identity map, so it is also quasi-M\"obius to the chordal metric.

\subsection{Distortion under quasiconformal maps}\label{section:relative:distortion}
We discuss the distortion of the relative hyperbolic metric under quasiconformal maps of the sphere. We state the main result of the section.

\begin{theorem}\label{theorem:relative_hyperbolic_qc}
Let $U,V\subset \widehat \C$ be quasidisks such that $\bar U\cap \bar V$ contains at most one point. Let  $f\colon  \widehat \C\to \widehat \C$ be a quasiconformal map. Then the map
$$f\colon (\partial V\setminus \bar U, d_{V,U})\to (f(\partial V) \setminus f(\bar U),d_{f(V),f(U)})$$
is quasisymmetric, quantitatively.
\end{theorem}

We first establish some preliminary statements.

\begin{lemma}\label{lemma:reflection_distance_chi}
Let $K\geq 1$. Let $V\subset \widehat \C$ be a $K$-quasidisk with $\diam_{\chi}(\partial V)\geq K^{-1}$ and let $g\colon  \widehat \C\to  \widehat \C$ be a $K$-quasiconformal reflection along $\partial V$. 
\begin{enumerate}[label=\normalfont(\arabic*)]
\item\label{lemma:reflection_distance:1_chi} For every pair of sets $E,F\subset \bar V$ we have
\begin{align*}
\frac{\dist_\chi(E,F)}{\dist_\chi(E,g(F))} \leq C(K).
\end{align*}
\item\label{lemma:reflection_distance:2_chi} For each $z\in \partial V$ and $w\in \widehat \C$ we have
$$C(K)^{-1}\leq \frac{\chi(z,w)}{\chi(z,g(w))} \leq C(K).$$
\item\label{lemma:reflection_distance:3_chi} For each set $E\subset \widehat \C$ such that $E\cap \partial V\neq \emptyset$ we have
$$C(K)^{-1}\leq  \frac{\diam_\chi (E)}{\diam_\chi (g(E))}\leq C(K).$$
\end{enumerate}
\end{lemma}

\begin{proof}
In \ref{lemma:reflection_distance:1_chi}, without loss of generality, suppose that $E,F$ are closed sets.  By property \ref{q:qm_qc} the map $g$ is quasi-M\"obius and by \ref{q:qs_qm} it is $\eta$-quasisymmetric for some distortion function $\eta$ that depends only on $K$. 

We first consider the special case that $E$ consists of a single point $v\in \partial V$. Let $u\in F$  be an arbitrary point. If $\chi(v,u)\leq \diam_\chi (\partial V)$, we consider a point $u'\in \partial V$ such that $\chi(v,u)=\chi(v,u')$. If $\chi(v,u)>\diam_\chi(\partial V)\geq K^{-1}$, we consider a point $u'\in \partial V$ such that $\chi(v,u')=\diam_\chi(\partial V)\geq K^{-1} \geq  2^{-1}K^{-1} \chi(v,u)$. In both cases, 
$$\chi(v,u)\leq 2K \chi (v,u')\leq 2K\chi(v,u).$$
We have
\begin{align*}
\dist_\chi(v,F)\leq \chi(v,u)&\leq 2K\chi(v,u')=2K\chi(g(v),g(u'))\\
&\leq 2K\eta(1) \chi(g(v),g(u))=2K\eta(1)\chi(v,g(u)).
\end{align*}
Since $u\in F$ is arbitrary, we conclude that
\begin{align}\label{lemma:reflection_distance:ineq_chi}
\dist_\chi(v,F)\leq 2K\eta(1) \dist_\chi(v,g(F)).
\end{align}

Now we treat the general case. Let $z\in \bar V$ and $w$ be the point of $g(F)$ that is closest to $z$. Consider the great circle of the sphere $\widehat \C$ that passes through $z$ and $w$ and denote by $J$ the shortest of the two arcs that connect $z$ and $w$. Since $z\in \bar V$ and $w\in \C\setminus V$, the arc $J$ intersects $\partial V$ at a point $v$ (which could be equal to $z$ or $w$). Note that the point of $g(F)$ that is closest to $v$ is also $w$. (To see this, suppose that $z=0$ by composing with an isometry of the sphere. Then the arc $J$ is a radius of a circle $C$ centered at $z=0$ that connects $z=0$ and $w\in C$. Thus, $C=\partial B_\chi(0,\chi(z,w))$. The point of $\widehat \C\setminus B_\chi(0,\chi(z,w))$ that is closest to the point $v\in J$ in the chordal metric is $w$.) Finally, let $u$ be the point of $F$ that is closest to $v$. See Figure \ref{figure:lemma:reflection} for an illustration. By \eqref{lemma:reflection_distance:ineq_chi}, we have
\begin{align*}
\dist_\chi(z, F)&\leq \chi(z,u) \leq \chi(z,v)+\chi(v,u)= \chi(z,v)+\dist_\chi(v,F)\\
&\leq \chi(z,v)+2K\eta(1)\dist_\chi(v,g(F))= \chi(z,v)+2K\eta(1) \chi(v,w)\\
&\leq \chi (z,w) + 2K\eta(1)\chi(z,w)=(1+2K\eta(1))\dist_\chi(z,g( F)). 
\end{align*}
Therefore,
$$\dist_\chi(E,F)\leq (1+2K\eta(1)) \dist_\chi(z,g(F))$$
for each $z\in E$, which completes the proof of \ref{lemma:reflection_distance:1_chi}.

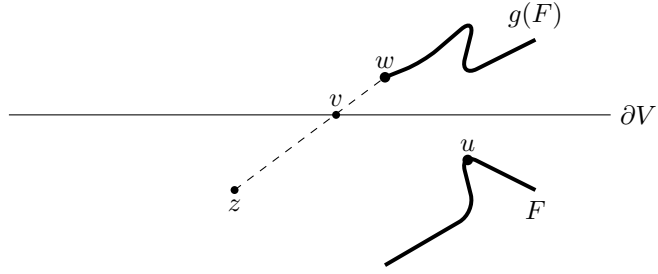
\begin{figure}
	\begin{tikzpicture}
		\draw (-4,0)--(4,0) node[right] {$\partial V$};	
		\fill (-1,-1) circle (1.5pt) node[below] {$z$};
		\draw[rounded corners=7pt, line width=1.5pt]  (1,-2)--(2.2,-1.3)--(2,-0.5)--(3,-1) node[below]{$F$};
		\draw[rounded corners=7pt, line width=1.5pt]  (1,0.5)--(1.5,0.7)--(2.2,1.3)--(2,0.5)--(3,1) node[above]{$g(F)$};
		\fill (1,0.5) circle (2pt) node[above] {$w$};
		\draw[dashed] (-1,-1)--(1,0.5);
		\fill (0.35,0) circle (1.5pt) node[above]{$v$};
		\fill (2.1,-0.6) circle (2pt) node[above]{$u$};
	\end{tikzpicture}
	\caption{The configuration in the proof of Lemma \ref{lemma:reflection_distance_chi} \ref{lemma:reflection_distance:1_chi}.}\label{figure:lemma:reflection}
\end{figure}

For \ref{lemma:reflection_distance:2_chi}, note that if $z\in \partial V$ and $w\in \bar V$, then part \ref{lemma:reflection_distance:1_chi} gives $\chi(z,w)\leq C(K)\chi(z,g(w))$. If instead $w\in \widehat\C\setminus \bar V$, we apply part \ref{lemma:reflection_distance:1_chi} to the quasidisk $\widehat\C\setminus \bar V$ and obtain the same inequality. Applying part \ref{lemma:reflection_distance:1_chi} to the quasiconformal reflection $g^{-1}$ gives $\chi(z,g(w))\leq C(K)\chi(z,w)$. 

For \ref{lemma:reflection_distance:3_chi}, we may assume that $E$ is closed.  Let $z\in E\cap \partial V$ and $w\in E$ such that $\chi(z,w)\geq \frac{1}{2}\diam_\chi (E)$. By part \ref{lemma:reflection_distance:2_chi}, we have
$$ \diam_\chi (E) \leq 2\chi(z,w) \leq 2C(K) \chi(z,g(w)) \leq 2C(K) \diam_\chi (g(E)).$$
To obtain an inequality in the reverse direction, we apply the same argument to $g^{-1}$.
\end{proof}

Next, we discuss the relation between the relative hyperbolic metric corresponding to the pair $(V,U)$ and the quasihyperbolic metric of a certain domain.

\begin{lemma}\label{lemma:hyperbolic_Z_chi}
Let $K\geq 1$. Let $U,V\subset \widehat \C$ be $K$-quasidisks such that $\bar U\cap \bar V$ contains at most one point and $\infty \in \bar U$. Let $Z$ be the domain bounded by $\partial U$ and some $K$-quasiconformal reflection of $\partial U$ along $\partial V$. If $\Delta_{\chi}(\partial U,\partial V)\leq K$, then 
$$C(K)^{-1}k_Z(z,w)\leq  d_{V,U}(z,w)\leq C(K)k_{Z}(z,w) \quad \text{for all $z,w\in \partial V\setminus \bar U$}.$$
\end{lemma}

See Figure \ref{figure:z} for an illustration of the domain $Z$.

\begin{proof}
We first perform some normalizations that do not affect the conclusion, so that $\diam_\chi(\partial V)$ is large, depending only on $K$. We apply a conformal map $\phi$ that maps $U^*$ onto $\D$ such that $0\in \partial V$. This map does not affect the inequalities in the conclusion since $d_{V,U}$ is invariant under conformal maps and $k_Z$ is quasi-invariant under conformal maps by Theorem \ref{theorem:gehring_osgood}. By Proposition \ref{proposition:extension_by_reflection} and \ref{q:qm_qc}, the conformal map $\phi$ extends to a quasi-M\"obius map of $\widehat \C$, so $\phi(V)$ is a quasidisk. By Lemma \ref{lemma:quasimobius:cross} we have $\Delta_\chi(\phi(\partial U),\phi(\partial V))\leq C(K)$. By Remark \ref{remark:quasimobius:cross}, the same inequality is true with Euclidean distance if we alter the constant. If $\phi( \partial V)\subset \D(0,r)$, then 
$$2r\geq \diam_e(\phi(\partial V))\geq C(K)^{-1} \dist_{e} (\phi(\partial U),\phi(\partial V)) \geq C(K)^{-1}(1-r)$$
which leads to a contradiction if $r=r(K)=\frac{1}{2C(K)+2}$. Thus, $\phi(\partial V)$ is not contained in $\D(0,r)$ and $\diam_e(\phi(\partial V))\geq r(K)$. Summarizing, by performing the above normalization, we may assume that $U^*=\D$ and $\diam_e(\partial V)\geq r(K)$. Since $\partial V\subset \bar \D$, the chordal diameter is also bounded from below depending on $K$, as desired.

Let $g\colon \widehat\C\to \widehat \C$ be a $K$-quasiconformal reflection along $\partial V$. Also, let $Z$ be the domain bounded by $\partial U=\partial \D$ and $g(\partial U)$; see Figure \ref{figure:z} for an illustration. We alter the definition of $g$ inside $V$ and set it equal to $(g|_{\widehat \C\setminus \bar V})^{-1}$. The quasiconformal removability of the quasicircle $\partial V$ implies that the modified map $g$ is still a $K$-quasiconformal reflection along $\partial V$. By construction, it satisfies $g\circ g=\id$. 

We apply Lemma \ref{lemma:reflection_distance_chi} \ref{lemma:reflection_distance:1_chi} to the quasidisk $\widehat \C\setminus \bar V$ and to the sets $E=\{z\}$, where $z\in U^*\setminus V$, and $F=\partial U$ and we obtain $\dist_{\chi}(z,\partial U)\leq C(K)\dist_{\chi}(z, g(\partial U))$. The same inequality with a different constant is true with Euclidean distances because the sets involved are contained in $\bar \D$. Therefore, 
$$\dist_e(z,\partial U)\leq C(K)\dist_e(z, \partial U\cup g(\partial U))=C(K)\dist_e(z,\partial Z).$$
Intuitively, the point $z$ is farther away from the reflection $g(\partial U)$ than from $\partial U$. If $\gamma$ is a rectifiable curve in $U^*\setminus V$, then by Theorem \ref{theorem:hyperbolic_distance} we have
$$\ell_{h_{U^*}}(\gamma) \geq  \frac{1}{2} \int_{\gamma} \frac{|dz|}{\dist_e(z,\partial U)} \geq \frac{1}{2C(K)} \int_{\gamma} \frac{|dz|}{\dist_e(z,\partial Z)}= \frac{1}{2C(K)}\ell_{k_Z}(\gamma).$$
This shows that $d_{V,U}\gtrsim_K k_Z$. The rest of the proof is devoted to the reverse inequality.

We first prove some preparatory results. Let $z\in \bar V\cap Z$ and consider the ball $B= \bar \D(z,2^{-1}\dist_e(z,\partial Z))\subset Z$. We claim that
\begin{align}\label{lemma:hyperbolic_Z:B_chi}
\dist_e(B,\partial Z)\simeq_K \dist_e (B, g(\partial U)) \quad \text{and} \quad \diam_e (g(B)) \simeq_K \dist_e(g(B),\partial U).
\end{align}
For the first inequality suppose that 
$$\dist_e(B,\partial Z)=\dist_e(B,\partial U)\leq \dist_e(B,g(\partial U)).$$
This implies that $\dist_e(z,\partial Z)=\dist_e(z,\partial U)$ and $\dist_e(B,\partial U) = \frac{1}{2}\dist_e(z,\partial U)$.
By Lemma \ref{lemma:reflection_distance_chi} \ref{lemma:reflection_distance:1_chi} (which can be applied with Euclidean distances since the sets involved are contained in $\bar \D$), we have 
$$\dist_e(B,g(\partial U))< \dist_e(z,g(\partial U)) \lesssim_K \dist_e(z,\partial U)\simeq_K \dist_e(B,\partial U);$$
recall here that $g\circ g=\id$. This completes the proof of the first relation in \eqref{lemma:hyperbolic_Z:B_chi}. For the second relation note that $g$ and $g^{-1}$ are quasisymmetric by \ref{q:qm_qc} and \ref{q:qs_qm}. By Lemma \ref{lemma:quasisymmetric_relative_distance} and the fact that $\diam_e (B) \simeq_K \dist(B,g(\partial U))$, we obtain the claimed inequality.

Let $\gamma$ be a curve in $\bar V\cap Z$ with endpoints $z,w$ and suppose that 
$$|\gamma|\subset \bar \D(z,2^{-1}\dist_e(z,\partial Z))\eqqcolon B.$$
By Lemma \ref{lemma:qh_estimate} and \eqref{lemma:hyperbolic_Z:B_chi} we have
\begin{align}\label{lemma:hyperbolic_Z:gamma_in_B_chi}
\ell_{k_{Z}}(\gamma) \simeq \frac{\ell_e(\gamma)}{\dist_e(B,\partial Z)}\simeq_K \frac{\ell_e(\gamma)}{\dist_e(B,g(\partial U))}.
\end{align}
Consider the reflection $g(B)$, which is a quasidisk, so it is quasiconvex. The reflected curve $g(|\gamma|)$ might not be rectifiable. However, we can find a rectifiable curve $\gamma'$ connecting $g(z),g(w)$ inside $g(B)$ and with $\ell_e(\gamma')\simeq_K |g(z)-g(w)|$. By  Theorem \ref{theorem:hyperbolic_distance}, \eqref{lemma:hyperbolic_Z:B_chi}, and Lemma \ref{lemma:qh_estimate}, we obtain
\begin{align}\label{lemma:hyperbolic_Z:gamma_reflected_chi}
\ell_{h_{U^*}}(\gamma')\simeq\ell_{k_{U^*}}(\gamma')\simeq_K \frac{\ell_e(\gamma')}{\dist_e(g(B),\partial U)}\simeq_K\frac{|g(z)-g(w)|}{\dist_e(g(B),\partial U)}.
\end{align}
We have completed the preparation for showing that $d_{V,U}\lesssim_K k_Z$.

Now let $\gamma$ be a curve in $\bar V\cap Z$ with endpoints $z,w\in \partial V$. Suppose that $|\gamma|\subset \bar \D(z,2^{-1}\dist_e(z,\partial Z))\eqqcolon B$.  By \eqref{lemma:hyperbolic_Z:B_chi} and Lemma \ref{lemma:reflection_distance_chi} \ref{lemma:reflection_distance:3_chi}, we have
\begin{align*}
\dist_e(g(B),\partial U)&\simeq _K \diam_e(g(B))\simeq_K \diam_e(B) \\
&\simeq_K\dist_e(B,\partial Z)\simeq_K\dist_e(B, g(\partial U)).
\end{align*}
Combining these inequalities with  \eqref{lemma:hyperbolic_Z:gamma_reflected_chi} and \eqref{lemma:hyperbolic_Z:gamma_in_B_chi}, we have 
$$\ell_{h_{U^*}}(\gamma')\simeq_K \frac{|z-w|}{\dist_e(g(B),\partial U)}\simeq_K \frac{|z-w|}{\dist_e(B, g(\partial U))} \lesssim_K \frac{\ell_e(\gamma)}{\dist_e(B,g(\partial U))}\simeq_K \ell_{k_Z}(\gamma).$$
We conclude that $\ell_{h_{U^*}}(\gamma')\lesssim_K \ell_{k_Z}(\gamma)$. 

Next, suppose that $|\gamma|$ is not contained in the ball $B_0=\bar \D(z,2^{-1}\dist_e(z,\partial Z))$.  Inductively, we can find a partition of the curve $\gamma$ into  subcurves $\gamma_i$ with endpoints $z_{i-1},z_i$, $i\in \{1,\dots,n\}$, such that $z_0=z$, $z_n=w$, $|\gamma_i|\subset {B_{i-1}}$, $B_i=\bar \D(z_i, 2^{-1}\dist_e(z_i,\partial Z))$ for $i\in \{1,\dots,n\}$ and $z_i\in \partial B_{i-1}\cap |\gamma|$ for $i\in \{1,\dots,n-1\}$. Observe that $n\geq 2$ by the assumption that $|\gamma|\not\subset B_0$. By \eqref{lemma:hyperbolic_Z:gamma_in_B_chi}, we have 
$$\ell_{k_Z}(\gamma_i) \simeq_K \frac{\ell_e(\gamma_i)}{\dist_e(B_{i-1},\partial Z)} \gtrsim_K  \frac{|z_i-z_{i-1}|}{\dist_e(B_{i-1},\partial Z)} \simeq_K 1$$
for $i\in \{1,\dots,n-1\}$. Thus,  $\ell_{k_{Z}}(\gamma)\gtrsim_K n$, given that $n\geq 2$. We ``reflect" each curve $\gamma_i$ according to the above procedure and obtain curves $\gamma_i'$, $i\in \{1,\dots,n\}$, whose concatenation gives a curve $\gamma'$. By \eqref{lemma:hyperbolic_Z:gamma_reflected_chi}, for $i\in \{1,\dots,n\}$ we have 
$$\ell_{h_{U^*}}(\gamma'_i) \simeq_K \frac{|g(z_{i}) -g(z_{i-1})|}{\dist_e(g(B_{i-1}),\partial U)}\lesssim_K \frac{\diam_e (g(B_{i-1}))}{\dist_e(g(B_{i-1}), \partial U)} \simeq_K 1,$$
where the last inequality follows from \eqref{lemma:hyperbolic_Z:B_chi}. Therefore, $\ell_{h_{U^*}}(\gamma')\lesssim n$. Altogether we obtain $\ell_{h_{U^*}}(\gamma')\lesssim_K \ell_{k_Z}(\gamma)$ in this case too. 

More generally, if $\gamma$ is an arbitrary rectifiable curve in $Z$ with endpoints  $z,w\in \partial V$ we can ``reflect" the (countably many) pieces of $\gamma$ that are in $V$ to obtain a curve $\gamma'$ in $U^*\setminus V$ that connects $z,w$ and satisfies
$$d_{V,U}(z,w)\leq \ell_{h_{U^*}}(\gamma')\lesssim_K \ell_{k_Z}(\gamma).$$
This implies that $d_{V,U}\lesssim_K k_Z$, as desired. 
\end{proof}

\begin{lemma}\label{lemma:kz_estimate}
Let $K\geq 1$. Let $U,V\subset \widehat \C$ be $K$-quasidisks such that $\bar U\cap \bar V=\emptyset$ and $\infty \in U$. Let $Z$ be the domain bounded by $\partial U$ and some $K$-quasiconformal reflection of $\partial U$ along $\partial V$. If $\Delta_e(\partial U,\partial V)>K^{-1}$, then for all $z,w\in \partial V$ we have
\begin{align}\label{lemma:kz_estimate:1}
\frac{1}{6}\frac{|z-w|}{\diam_e(\partial V)}\leq k_Z(z,w)\leq C(K) \frac{|z-w|}{\diam_e(\partial V)}
\end{align}
and
$$C(K)^{-1}\Delta_e(\partial U,\partial V)d_{V,U}(z,w)\leq k_Z(z,w)\leq C(K)\Delta_e(\partial U,\partial V) d_{V,U}(z,w).$$
\end{lemma}
\begin{proof}
By using a translation and a scaling we may assume that $0\in  V$ and that $\diam_e(\partial V)=1$, so $\bar V\subset \D$. Note that the quantity $\Delta_{e}(\partial U,\partial V)$ is unaffected by this normalization. Let $g$ be a $K$-quasiconformal reflection along $\partial V$ and $Z$ be the domain bounded by $\partial U$ and $g(\partial U)$.

Let $\gamma$ be a curve in $Z$ connecting two points $z,w\in \partial V$. We consider two cases. Suppose that $\gamma$ is contained in the disk $\bar \D(0,2)$. Then for $\zeta\in |\gamma|$ we have $\dist_e(\zeta,\partial Z)\leq \dist_e(\zeta,g(\partial U))\leq 3$. Thus, 
$$\ell_{k_Z}(\gamma)=\int_\gamma\frac{|d\zeta|}{\dist_e(\zeta,\partial Z)}\geq \frac{1}{3} \ell_e(\gamma)\geq \frac{1}{3} |z-w|.$$
Now, if $\gamma$ is not contained in the disk $\bar {\D}(0,2)$, it has a subcurve $\gamma'$ that is contained in $\bar \D(0,2)$ and connects the boundary components of the annulus $\mathbb A(0;1,2)$. Therefore, 
$$\ell_{k_Z}(\gamma) \geq \ell_{k_Z}(\gamma') \geq \frac{1}{3}\ell_e(\gamma') \geq \frac{1}{3}\geq \frac{1}{6}|z-w|.$$
Therefore, $k_Z(z,w)\geq \frac{1}{6}|z-w|$. 

If $\dist_e(\partial U,\partial V)\leq 2$, then $K^{-1}<\Delta_e(\partial U,\partial V)=\dist_e(\partial U,\partial V)\leq 2$. By Remark \ref{remark:quasimobius:cross}, we have $\Delta_\chi(\partial U,\partial V)\simeq _K1$. By Lemma \ref{lemma:hyperbolic_Z_chi}, we have $k_Z(z,w)\simeq_K d_{V,U}(z,w)$ and by Lemma \ref{lemma:relatively_separated},  we have $d_{V,U}(z,w)\simeq_K \frac{|z-w|}{\dist_e(\partial U,\partial V)}$ for $z,w \in \partial V$. Therefore, 
$$k_Z(z,w)\simeq_K |z-w|.$$

Finally, suppose that $\Delta_e(\partial U,\partial V)=\dist_e(\partial U,\partial V)>2$. In particular, $U$ is disjoint from $\bar \D(0,2)$. Since $\C\setminus  V$ is quasiconvex, for any two points $z,w\in \partial V$ there exists a curve $\gamma$ in $\C\setminus V$ connecting $z$ and $w$ with $\ell_e(\gamma)\lesssim_K |z-w|$. By altering the curve $\gamma$ and replacing parts of it by arcs of $\partial \D$ if necessary, we may assume that $|\gamma|\subset \bar \D\setminus V\subset Z$. 

Since $\partial V\subset \D$ and $\diam_e(\partial V)=1$, we have $\diam_\chi(\partial V)\simeq 1$. Moreover, since $\partial U\cap  \D(0,2)=\emptyset$ and $\partial V\subset \D$, we have $\dist_\chi(\partial V,\partial U)\simeq 1$. By Lemma \ref{lemma:reflection_distance_chi} \ref{lemma:reflection_distance:1_chi} we have $\dist_\chi(\partial V,g(\partial U))\simeq_K \dist_\chi(\partial V,\partial U)\simeq 1$. Thus, $\dist_e(\partial V,g(\partial U))\simeq_K 1$. This implies that for each point $\zeta\in |\gamma|$ we have $\dist_e(\zeta, g(\partial U)) \gtrsim_K 1$. Also, $\dist_e(\zeta, \partial U) \geq \dist_e(\partial \D, \partial U)\geq 1$. Therefore, $\dist_e(\zeta,\partial Z)\gtrsim_K1$ and
$$k_Z(z,w)\leq \int_\gamma \frac{|d\zeta|}{\dist_e(\zeta,\partial Z)} \lesssim_K \ell_e(\gamma)\lesssim_K |z-w|.$$
This completes the proof of \eqref{lemma:kz_estimate:1}. The last inequality in the statement of the lemma follows immediately from \eqref{lemma:kz_estimate:1} and Lemma \ref{lemma:relatively_separated}.
\end{proof}

\begin{corollary}\label{corollary:quasisymmetric_dk}
Let $K\geq 1$. Let $U,V\subset \widehat \C$ be $K$-quasidisks such that $\bar U\cap \bar V$ contains at most one point and $\infty \in \bar U$. Let $Z$ be the domain bounded by $\partial U$ and some $K$-quasiconformal reflection of $\partial U$ along $\partial V$. Then the map
$$\id\colon (\partial V\setminus \bar U,d_{V,U})\to (\partial V\setminus \bar U,k_Z)$$
is quasisymmetric, quantitatively depending only on $K$.
\end{corollary}

\begin{proof}
If $\Delta_{\chi}(\partial U,\partial V)\leq 1$, the claim follows from Lemma \ref{lemma:hyperbolic_Z_chi}. Suppose that $\Delta_{\chi}(\partial U,\partial V)> 1$. In particular $\bar U\cap \bar V=\emptyset$.  Let $g$ be a $K$-quasiconformal reflection along $\partial V$ and $Z$ be the domain bounded by $\partial U$ and $g(\partial U)$. By using a conformal isometry of the sphere, we may assume that $\infty \in U$; note that $d_{V,U}$ remains invariant under a conformal isometry and $k_Z$ is altered in a bi-Lipschitz way by Theorem \ref{theorem:gehring_osgood}. By Remark \ref{remark:quasimobius:cross} we have $\Delta_e(\partial U,\partial V)\gtrsim 1$. By Lemma \ref{lemma:kz_estimate}, the identity map $\id\colon (\partial V,d_{V,U})\to (\partial V,k_Z)$ is quasisymmetric.
\end{proof}

We are now in position to prove Theorem \ref{theorem:relative_hyperbolic_qc}.

\begin{proof}[Proof of Theorem \ref{theorem:relative_hyperbolic_qc}]
Let $U'=f(U)$, $V'=f(V)$, $d=d_{V,U}$, $d'=d_{V',U'}$. We perform some normalizations. By composing $f$ with suitable M\"obius transformations, we assume that  $\infty \in U$ and $f(\infty)=\infty$, so $\infty\in U'$. Let $Z$ be the domain bounded by $\partial U$ and some quasiconformal reflection of $\partial U$ along $\partial V$. If $g\colon \widehat \C \to \widehat \C$ is this quasiconformal reflection, then $h=f\circ g\circ f^{-1}$ is a quasiconformal reflection along $\partial V'$ and the domain $Z'=f(Z)$ is bounded by $\partial U'$ and $h(\partial U')$. By Corollary \ref{corollary:quasisymmetric_dk} the identity map $\id \colon (\partial V\setminus \bar U,d)\to (\partial V\setminus \bar {U},k_Z)$ is quasisymmetric. Theorem \ref{theorem:qc_qs_qh} implies that $f\colon (\partial V\setminus \bar U, k_Z) \to (\partial V'\setminus \bar{U'}, k_{Z'})$ is quasisymmetric. Finally, again by Corollary \ref{corollary:quasisymmetric_dk} the map $\id\colon (\partial V'\setminus \bar{U'},k_{Z'})\to (\partial V'\setminus \bar{U'},d')$ is quasisymmetric. Therefore, $f\colon (\partial V\setminus \bar U,d)\to (\partial V'\setminus \bar{U'},d')$ is quasisymmetric. All implications are quantitative.
\end{proof}

\subsection{Characterizing tangent quasicircles}\label{section:relative:tangent}
In this section we use only  planar topology, so the point $\infty$ does not lie in the closure of unbounded sets.

\begin{lemma}\label{lemma:conformal_strip}
Let $U,V\subset \C$ be unbounded Jordan regions such that $\bar U\cap \bar V=\emptyset$ and let $\Omega=\C\setminus (\bar U\cup \bar V)$. There exists a conformal map $\phi \colon \Omega\to S=\{z\in \C: 0<\im(z)<1\}$ such that $\phi$ has an extension to a homeomorphism  $\phi\colon \bar {\Omega}\to \bar S$ that satisfies $\phi(\partial V)=\R \times \{0\}$ and $\phi(\partial U)=\R\times\{1\}$. 
\end{lemma}

\begin{proof}
Since $U,V$ are unbounded, in the topology of the sphere their boundaries meet at $\infty$. By applying a M\"obius transformation, we may assume that $U$ contains a neighborhood of $\infty$, $V$ is bounded, $V\subset U^*$, and $\partial V\cap \partial U=\{1\}$.  

The domain $\Omega$ is simply connected and has locally connected boundary. By the Riemann mapping theorem, there exists a conformal map $\alpha\colon\D\to \Omega$. By Carath\'eodory's theorem \cite{Pommerenke:conformal}*{Theorem 2.1}, $\alpha$ has an extension to a continuous map $\alpha\colon \bar \D\to \bar{\Omega}$. Note that $\partial \Omega=\partial U\cup \partial V$ and $\partial \Omega \setminus\{1\}$ has two components.  By \cite{Pommerenke:conformal}*{Proposition 2.5},  we conclude that $\alpha^{-1}(1)$ consists of two points in $\partial \D$. On the other hand, for $z\in \partial \Omega\setminus \{1\}$  the set $\partial \Omega \setminus \{z\}$ is connected, so $\alpha^{-1}(z)$ is a single point. By precomposing $\alpha$ with a M\"obius transformation of the unit disk, we may assume that $\alpha(1)=\alpha(-1)=1$ and $\alpha(i)\in \partial V$. 

Similarly, we use a M\"obius transformation that maps the half-plane $\{z\in \C:\im(z)>1\}$ onto $U'=\D^*$ and the lower half-plane $\{z\in \C: \im(z)<0\}$ onto a disk $V'\subset \D$ that is tangent to $\D$ at $1$. There exists a conformal map $\beta\colon \D\to \Omega'=\D\setminus \bar {V'}$ (which can be written explicitly)  that extends to a continuous map of the closures and satisfies $\beta(1)=\beta(-1)=1$ and $\beta(i)\in \partial V'$. 

The map $\phi=\beta\circ \alpha^{-1}\colon \Omega\to \Omega'$ extends to a homeomorphism of the closures that fixes the point $1$ and maps $\partial V$ onto $\partial V'$ and $\partial U$ onto $\partial U'$. With appropriate compositions this map yields the desired map between the domains in the statement of the lemma. 
\end{proof}
 
We now restate and prove Theorem \ref{theorem:intro:characterization_tangent}.

\begin{theorem}\label{theorem:characterization_tangent}
Let $U,V\subset \C$ be unbounded quasidisks such that $\bar U\cap \bar V=\emptyset$. There exists a quasiconformal map $f\colon \C\to \C$ that maps $U$ and $V$ to half-planes if and only if the identity map 
$$\id \colon (\partial V,d_{V,U})\to (\partial V,|\cdot|)$$
is quasisymmetric, quantitatively.  
\end{theorem}

\begin{proof}
Let $f\colon \C\to \C$ be a quasiconformal map that maps $U,V$ to a pair of half-planes $U',V'\subset \C$ with disjoint closures, respectively. By property \ref{q:qc_qs}, $f$ and $f^{-1}$ are quasisymmetric maps. Let $d=d_{V,U}$ and $d'=d_{V',U'}$. By Theorem \ref{theorem:relative_hyperbolic_qc}, $f\colon (\partial V,d)\to (\partial V',d')$ is quasisymmetric.  Next, Lemma \ref{lemma:parallel} implies that $\id\colon (\partial V',d')\to (\partial V',|\cdot|)$ is quasisymmetric. Finally, the map $f^{-1}\colon (\partial V',|\cdot|)\to (\partial V,|\cdot|)$ is quasisymmetric. Therefore, the identity map $\id \colon (\partial V,d)\to (\partial V,|\cdot|)$ is quasisymmetric.  

Conversely, suppose that the identity map $\id \colon (\partial V,d_{V,U})\to (\partial V,|\cdot|)$ is quasisymmetric. We perform a normalization. By the definition of a quasidisk, there exists a quasiconformal map  $g\colon \C\to \C$ that maps $V$ onto the lower half-plane $g(V)=\{z\in \C: \im(z)<0\}$. Then $g(U)$ is an unbounded Jordan region contained in the upper half-plane. By Theorem \ref{theorem:relative_hyperbolic_qc}, $g\colon (\partial V,d_{V,U})\to (g(\partial V),d_{g(V),g(U)})$ is quasisymmetric and also $g\colon (\partial V,|\cdot|)\to (g(\partial V),|\cdot|)$ is quasisymmetric by property \ref{q:qc_qs}. Therefore, using our main assumption, we see that the identity map $\id \colon (g(\partial V),d_{g(V),g(U)})\to (g(\partial V),|\cdot|)$ is quasisymmetric.  By replacing $V$ with $g(V)$ and $U$ with $g(U)$ we assume that the identity map $\id \colon (\partial V,d_{V,U})\to (\partial V,|\cdot|)$ is quasisymmetric, $\partial V$ is the real line $\R\times \{0\}$ and $U$ is unbounded and contained in the upper half-plane.

By Lemma \ref{lemma:conformal_strip}, there exists a conformal map $f$ from $\C\setminus (\bar U\cup \bar V)$ onto the strip $S=\{z\in \C: 0<\im(z)<1\}$  that extends to a homeomorphism of the closures and satisfies $f(\partial U)=\R\times\{1\}$ and $f(\R\times\{0\})=\R\times\{0\}$. Let $Z$ be the domain bounded by $\partial U$ and by its reflection along the real line. By Schwarz reflection, we can extend $f$ to a conformal map from $Z$ onto $Z'=\{z\in \C:-1<\im (z)<1\}$. The map $f|_{\R\times\{0\}}$ is increasing. Indeed, if we write $f=u+iv$, the fact that $S\subset \UHP$ implies that $v_y(x,0)\geq 0$, so $u_x(x,0)\geq 0$ for all $x\in \R$. 

We identity $\R\times\{0\}$ with $\R$. Recall that the map $\id \colon (\R,|\cdot|)\to (\R,d_{V,U})$ is quasisymmetric by the normalizations. By Lemma \ref{lemma:hyperbolic_Z_chi}, the map $\id\colon (\R ,d_{V,U})\to(\R,k_Z)$ is bi-Lipschitz. By Theorem \ref{theorem:gehring_osgood}, $f\colon (\R,k_Z)\to (\R,k_{Z'})$ is bi-Lipschitz. Finally, in the strip $Z'$ we have the explicit formula $k_{Z'}(z,w)= |z-w|$  for $z,w\in \R$. By considering the composition of all these maps, we see that $f\colon \R\to \R$ is quasisymmetric and increasing. By the Beurling--Ahlfors extension theorem (Theorem \ref{theorem:beurling_ahlfors}) and an appropriate reflection, we can extend $f$ to a quasiconformal homeomorphism of the lower half-plane $V$. The fact that $\R$ is removable for quasiconformal homeomorphisms implies that $f\colon U^*\to \{z\in \C: \im(z)<1\}$ is a quasiconformal homeomorphism. Note that $f(z)\to \infty$ as $z\to\infty$. Since $U$ is a quasidisk, we can extend $f$ a quasiconformal map of $\C$ by Proposition \ref{proposition:extension_by_reflection}. 
\end{proof}

\subsection{Characterizing quasiannuli}\label{section:relative:quasiannuli}
In this section we use the topology of the sphere and the chordal metric. We restate and prove Theorem \ref{theorem:intro:characterization_quasiannuli}.

\begin{theorem}\label{theorem:characterization_quasiannuli}
Let $U,V\subset \widehat \C$ be quasidisks such that $\bar U\cap \bar V=\emptyset$. There exists a quasiconformal map $f\colon \widehat\C\to \widehat \C$ that maps $U$ and $V$ to disks if and only if the identity map
$$\id \colon (\partial V,d_{V,U})\to (\partial V,\chi)$$
is quasi-M\"obius, quantitatively.  
\end{theorem}

The proof follows almost identically the steps of the proof of Theorem \ref{theorem:characterization_tangent}. The main difference occurs at the normalizations that we have to perform.

\begin{proof}
Let $f\colon \widehat \C\to \widehat \C$ be a quasiconformal map that maps $U,V$ to a pair of disks $U',V'$ with disjoint closures. By postcomposing $f$ with a M\"obius transformation, we assume that $f\colon \widehat \C\to \widehat \C$ is a quasiconformal map that maps $U$ onto $\D(0,R)^*$ and $V$ onto $\D(0,r)$ for some $0<r<R$. Let $d=d_{V,U}$  and $d'=d_{V',U'}$. By Theorem \ref{theorem:relative_hyperbolic_qc}, $f\colon (\partial V,d)\to (\partial V',d')$ is quasisymmetric.  Next, Lemma \ref{lemma:parallel} implies that $\id\colon (\partial V',d')\to (\partial V',|\cdot|)$ is quasisymmetric. Also, $\id\colon (\partial V',|\cdot|)\to (\partial V',\chi)$ is quasi-M\"obius because it preserves cross ratios. Finally, the map $f^{-1}\colon (\partial V',\chi)\to  (\partial V,\chi)$ is quasi-M\"obius by \ref{q:qm_qc}. Therefore, the identity map $\id \colon (\partial V,d)\to (\partial V,\chi)$ is quasi-M\"obius.  

Conversely, suppose that the identity map $\id \colon (\partial V,d_{V,U})\to (\partial V,\chi)$ is quasi-M\"obius. We perform a normalization. There exists a quasiconformal map $g\colon \widehat \C\to \widehat \C$ that maps $V$ onto $\D$ such that $\infty \in g(U)$. By \ref{q:qm_qc} the map $g\colon (\partial V,\chi)\to (\partial \D,|\cdot|)$ is quasi-M\"obius.  The set $g(U)$ is a quasidisk contained in $\D^*$. By Theorem \ref{theorem:relative_hyperbolic_qc}, $g\colon (\partial V,d_{V,U})\to (\partial \D, d_{g(V),g(U)})$ is quasisymmetric. Therefore, by performing this normalization, we may assume that $V=\D$, $U$ is a quasidisk containing $\infty$, and $\id\colon (\partial V, d_{V,U})\to (\partial V,|\cdot|)$ is quasi-M\"obius. 

There exists a conformal map $f$ from $\C\setminus (\bar U\cup \bar \D)$ onto the annulus $\mathbb A(0;1,R)$ for some $R>1$ that extends to a homeomorphism of the closures and satisfies $f(\partial \D)=\partial \D$ and $f(\partial U)=\mathbb S(0,R)$; see Theorems 15.7.9 and 15.3.4 in \cite{Conway:complex2}. Let $Z$ be the domain bounded by $\partial U$ and by its reflection along $\partial \D$. By Schwarz reflection, we can extend $f$ to a conformal map from $Z$ onto $Z'=\mathbb A(0;1/R,R)$. Note that the map $f|_{\partial \D}$ is orientation-preserving. 

Recall that the map $\id\colon (\partial \D,|\cdot|)\to (\partial \D, d_{V,U})$ is quasi-M\"obius by our normalization. By Corollary \ref{corollary:quasisymmetric_dk}, the map $\id\colon (\partial \D ,d_{V,U})\to (\partial \D,k_Z)$ is quasisymmetric. By Theorem \ref{theorem:gehring_osgood}, $f\colon (\partial \D,k_Z)\to (\partial \D,k_{Z'})$ is bi-Lipschitz. Finally, in the annulus $Z'$ the map $\id\colon (\partial \D,k_{Z'}) \to (\partial \D,|\cdot|)$ is quasisymmetric, because
$$k_{Z'}(z,w) \simeq \frac{|z-w|}{R-1}$$
for $z,w\in \partial \D$, with uniform constants. Thus, $f\colon (\partial \D,|\cdot|)\to (\partial \D,|\cdot|)$ is quasi-M\"obius and orientation-preserving. By Corollary \ref{corollary:beurling_ahlfors}, there exists an extension of $f$ to a quasiconformal homeomorphism of $V$. The quasiconformal removability of $\partial \D$ implies that $f$ is a quasiconformal map from $U^*$ onto $\D(0,R)$. Since $U$ is a quasidisk, we can extend $f$ to a quasiconformal map of $\widehat \C$ by Proposition \ref{proposition:extension_by_reflection}.
\end{proof} 

It is desirable to have that the map $f$ in Theorem \ref{theorem:characterization_quasiannuli} is quasisymmetric and maps $\partial U,\partial V$ to concentric Euclidean circles. We give a sufficient condition.

\begin{corollary}\label{corollary:characterization_quasiannuli_qs}
Let $K\geq 1$ and $R>r>0$. Let $V$ be a $K$-quasidisk such that $0\in V$ and $$\partial V\subset \{z\in \C: \max\{r-K(R-r),K^{-1}r\} \leq |z|\leq r\}$$ and let $U\subset\widehat \C$ be a $K$-quasidisk such that $\infty \in U$ and $\partial U \subset \{z\in \C: R\leq |z|\leq R+K(R-r)\}$. Then there exists a quasiconformal map of $\C$ that maps $V$ and $U^*$ onto concentric Euclidean disks, quantitatively depending only on $K$.
\end{corollary}

\begin{proof}
By Lemma \ref{lemma:quasicircle_concentric}, the identity map $\id\colon (\partial V, d_{V,U})\to (\partial V,|\cdot|)$ is quasisymmetric and thus quasi-M\"obius. Since cross ratios in the Euclidean and chordal metrics are the same, we conclude that $\id\colon (\partial V, d_{V,U})\to (\partial V,\chi)$ is quasi-M\"obius. By Theorem \ref{theorem:characterization_quasiannuli} there exists a quasiconformal map $f$ of $\widehat \C$ that maps $V$ onto $\D$, and $U^*$ onto a disk $\D(0,R)$, $R>1$. Note that $\id\colon (\partial U,|\cdot|)\to (\partial U,d_{U,V})$ is quasisymmetric by Lemma \ref{lemma:quasicircle_concentric}. By Theorem \ref{theorem:relative_hyperbolic_qc} the map $f\colon (\partial U,d_{U,V})\to (\mathbb S(0,R), d_{\D(0,R)^*, \D})$ is quasisymmetric.  Finally, the identity map $\id\colon (\mathbb S(0,R), d_{\D(0,R)^*, \D})\to  (\mathbb  S(0,R),|\cdot|)$ is quasisymmetric by Lemma \ref{lemma:parallel} \ref{lemma:parallel:2}. Altogether, $f\colon (\partial U,|\cdot|)\to (\mathbb S(0,R),|\cdot|)$ is quasisymmetric. By Corollary \ref{corollary:beurling_ahlfors_embedding}, it extends to a quasiconformal map from $U\cap \C$ onto $\D(0,R)^*\cap \C$. If we replace the definition of $f$ in $U\cap \C$ with this extension, we obtain the desired map.
\end{proof}

\begin{proof}[Proof of Theorem \ref{theorem:intro:characterization_combined}]
The case that $\bar U\cap \bar V=\emptyset$ follows from Theorem \ref{theorem:characterization_quasiannuli}, so we suppose that $\bar U\cap \bar V\neq \emptyset$. Suppose that there exists a quasiconformal map $f$ of $\widehat \C$ that maps $U,V$ to disks. After composing with suitable M\"obius transformations, we may have that $\bar U\cap \bar V=\{\infty\}$, so $U,V$ are unbounded quasidisks in $\C$ and $f$ maps them to half-planes.  By Theorem \ref{theorem:characterization_tangent}, the identity map $\id \colon (\partial V\setminus \bar U,d_{V,U})\to (\partial V\setminus \bar U,|\cdot|)$ is quasisymmetric. Thus, it is also quasi-M\"obius. Since cross ratios in the Euclidean and spherical metrics are the same, the identity map $\id \colon (\partial V\setminus \bar U,d_{V,U})\to (\partial V\setminus \bar U, \chi)$ is quasi-M\"obius. 

Conversely, suppose that the identity map $\id \colon (\partial V\setminus \bar U,d_{V,U})\to (\partial V\setminus \bar U, \chi)$ is quasi-M\"obius. By applying a M\"obius transformation, we may assume that $\bar U\cap \bar V=\{\infty\}$. Thus, the identity map $\id \colon (\partial V\setminus \bar U,d_{V,U})\to (\partial V\setminus \bar U, |\cdot|)$ and its inverse are quasi-M\"obius. Note that if $w\in \partial V\setminus \bar U$ is a given point, then 
$$\lim_{\substack{z\to\infty\\z\in \partial V\setminus \bar U}}d_{V,U}(z,w)= \infty.$$  
By \ref{q:qm_qs}, the map $\id \colon (\partial V\setminus \bar U, |\cdot|)\to (\partial V\setminus \bar U,d_{V,U})$ is quasisymmetric.
\end{proof}

\section{Graphs of functions and cusps}\label{section:graph}

In this section we prove Theorem \ref{theorem:intro:graph} and Theorem \ref{theorem:intro:graph:characterization}. We use planar topology, so the closure of an unbounded set does not contain the point $\infty$. We first establish a preliminary statement.

\begin{proposition}\label{theorem:graph:generalization}
Let $K\geq 1$ and $L,a>0$. Let $f\colon \R\to (0,\infty)$ be an $L$-Lipschitz function and $U\subset \mathbb H$ be an unbounded $K$-quasidisk such that $\partial U\subset \{ (x,y)\in \R^2: a^{-1}f(x)\leq y\leq af(x)\}$ and each vertical line intersects $\partial U$. There exists a quasiconformal homeomorphism of $\C$ that maps preserves the real line and maps $\partial U$ onto a line if and only if an antiderivative of $1/f$ is quasisymmetric, quantitatively.
\end{proposition}

\begin{proof}
Let $\graph(f)=\{(x,f(x)):x\in \R\}$. For each $t\in \R$ the ball $\bar \D((t,0), f(t))$ contains the point $(t,f(t))$ and might contain other points of the graph of $f$ as well. Hence, there exists a point $y\in [t-f(t),t+f(t)]$ such that
\begin{align*}
\dist_e((t,0),\graph(f))= |(t,0)-(y,f(y))|.
\end{align*}
If $|t-y|\geq \frac{f(t)}{2L}$, then $\dist_e((t,0),\graph(f))\geq |t-y| \geq \frac{f(t)}{2L}$. If $|t-y|<\frac{f(t)}{2L}$, then $f(y)\geq f(t)-L|t-y|\geq \frac{1}{2} f(t)$. Therefore,
$$ \frac{1}{2}\min\{1,L^{-1}\} f(t)\leq \dist_e((t,0),\graph(f)) \leq f(t)$$
Since  $\partial U\subset \{ (x,y)\in \R^2: a^{-1}f(x)\leq y\leq af(x)\}$, we have
\begin{align}\label{theorem:graph:generalization:1}
\dist_e((t,0),\partial U)\geq \dist_e((t,0),\graph(a^{-1}f))\geq \frac{1}{2a}\min\{1,L^{-1}\} f(t)
\end{align}
for all $t\in \R$. 

Since $U$ is an unbounded Jordan region, the region $U^*$ is also an unbounded Jordan region in $\C$, so Theorem \ref{theorem:hyperbolic_distance} can be applied in order to estimate the hyperbolic metric on $U^*$. Let  $V$ be the lower half-plane. For $(x,0), (y,0)\in \partial V=\R\times \{0\}$ with $x<y$, by \eqref{theorem:graph:generalization:1} we have
$$d_{V,U}((x,0),(y,0))\leq \ell_{h_{U^*}}([x,y]) \leq \int_x^y\frac{2}{\dist_e((t,0),\partial U)}\, dt \leq C(a,L) \int_x^y \frac{1}{f(t)}\, dt.$$

Next, we show an inequality in the reverse direction. By the uniform continuity and positivity of $f$ in $[x,y]$, there exists $\delta>0$ such that if $|s-t|<\delta$ and $s,t\in [x,y]$, then
$$|f(s)-f(t)|<\frac{1}{2}\min \{f(r): r\in [x,y]\}\leq \frac{1}{2}f(t),$$
so $\frac{1}{2}f(t)\leq f(s)\leq \frac{3}{2}f(t)$. Therefore, there exists a partition $x=t_0<t_1<\dots<t_n=y$ of $[x,y]$ such that $\frac{1}{2}f(t_i)\leq f(t)\leq \frac{3}{2}f(t_i)$ for $t\in [t_{i-1},t_i]$, $i\in \{1,\dots,n\}$. 

Now, let $\gamma$ is an arbitrary curve in $U^*\setminus V\subset \{(t,s)\in \R^2: 0\leq s<af(t)\}$ with endpoints $(x,0)$ and $(y,0)$, $x<y$. There exists a collection of non-overlapping subcurves $\gamma_i$ of $\gamma$, $i\in \{1,\dots,n\}$ such that each $\gamma_i$ projects onto the interval $[t_{i-1},t_i]$ on the $x$-axis. Since $\partial U$ intersects each vertical line by assumption, observe that for each point $(t,s)$ on $\gamma_i$ we have 
$$\dist_e((t,s), \partial U) \leq af(t)\leq \frac{3a}{2}f(t_i).$$
Moreover, $\ell_e(\gamma_i)\geq t_i-t_{i-1}$. Therefore, by Theorem \ref{theorem:hyperbolic_distance} we have
\begin{align*}
\ell_{h_{U^*}}(\gamma)&\geq \frac{1}{2}\int_\gamma \frac{1}{\dist_e(\cdot,\partial U)}\, ds \geq \frac{1}{2}\sum_{i=1}^n \int_{\gamma_i} \frac{1}{\dist_e(\cdot,\partial U)}\, ds  \\&\geq \frac{1}{3a}\sum_{i=1}^n \int_{t_{i-1}}^{t_i}\frac{1}{f(t_i)}\, dt
\geq  \frac{1}{6a}\sum_{i=1}^n \int_{t_{i-1}}^{t_i}\frac{1}{f(t)}\, dt=\frac{1}{6a} \int_{x}^{y}\frac{1}{f(t)}\, dt.
\end{align*}
This proves that $d_{V,U}((x,0),(y,0))\geq \frac{1}{6a}\int_x^y \frac{1}{f(t)}\, dt$. 

Summarizing, if $F\colon \R\to \R$ is an antiderivative of $1/f$, we have proved that 
$$C(a,L)^{-1} |F(y)-F(x)|\leq d_{V,U}((x,0),(y,0)) \leq C(a,L) |F(y)-F(x)|$$
for all $x,y\in \R$. That is, the map $F\colon (\R, d_{V,U})\to (\R,|\cdot|)$ is bi-Lipschitz. By Theorem \ref{theorem:characterization_tangent} there exists a quasiconformal homeomorphism of $\C$ that maps $\partial V$ and $\partial U$ to two parallel lines if and only if the identity map $\id\colon (\R,d_{V,U})\to  (\R,|\cdot|)$ is quasisymmetric. Equivalently, the map $F\colon (\R,|\cdot|)\to (\R,|\cdot|)$ is quasisymmetric.
\end{proof}

\begin{proof}[Proof of Theorem \ref{theorem:intro:graph}]
By Proposition \ref{theorem:graph:generalization}, it suffices to show that if  $f\colon \R\to (0,\infty)$ is an $L$-Lipschitz function, then $U=\{(x,y)\in \R^2: y>f(x)\}$ is a quasidisk. For any two points $(x,f(x))$ and $(y,f(y))$ on the graph of $f$, if $A=\{(t,f(t)): t\in [x,y]\}$, then 
$$\diam_e(A)\leq \sqrt{1+L^2}|x-y|\leq \sqrt{1+L^2}|(x,f(x))-(y,f(y))|.$$
This shows that the boundary of $U=\{(x,y)\in \R^2: y>f(x)\}$ is a quasicircle (if we include the point at $\infty$). By Theorem \ref{prop:quasidisk_quasicircle}, $U$ is a quasidisk, as desired.
\end{proof}

\begin{proof}[Proof of Theorem \ref{theorem:intro:graph:characterization}]
First we show the sufficiency of conditions \ref{theorem:graph:characterization:1}, \ref{theorem:graph:characterization:3}, \ref{theorem:graph:characterization:4}. By Proposition \ref{theorem:graph:generalization}, it suffices to verify that $\partial U$ intersects each vertical line. By assumption \ref{theorem:graph:characterization:3}, we have $\partial U\subset \{(x,y)\in \R^2: 0< y\leq  af(x)\}$. By assumption \ref{theorem:graph:characterization:4} and Theorem \ref{theorem:intro:graph}, there exists a quasiconformal map of $\C$ that preserves the real line and maps the curve $y=af(x)$ onto the line $y=1$. The unbounded quasicircle $\partial U$ is mapped onto an unbounded quasicircle inside the strip $\{(x,y)\in \R^2: 0<y\leq 1\}$. By Lemma \ref{lemma:quasicircle_strip}, this quasicircle must separate the line $y=0$ from the half-plane $y>1$. Thus, $\partial U$ separates the line $y=0$ from the set $y>af(x)$. This implies that every vertical line intersects $\partial U$, as desired. 

We prove the necessity. Let $F$ be a $K$-quasiconformal map of $\C$ that  preserves the real line and maps $\partial U$ onto the line $y=1$. Note that $F|_{\R\times \{0\}}$ is increasing since $F$ is orientation-preserving. Condition \ref{theorem:graph:characterization:1} is immediate. Let $G=F^{-1}$, which is also $K$-quasiconformal and increasing on $\R\times \{0\}$. By \ref{q:qc_qs} the maps $F$ and $G$ are $\eta$-quasisymmetric for some $\eta$ depending only on $K$. Let $L>1$ such that $\eta(1/L)=1/3$. For $n\in \Z$ and $R_n=|G(nL,0)-G((n-1)L,0)|$ we have 
\begin{align}\label{graph:i}
|G(nL,1)-G(nL,0)|&\leq \eta(1/L) |G(nL,0)-G((n-1)L,0)|=R_n/3
\end{align}
and similarly $|G((n-1)L,1)-G((n-1)L,0)|\leq R_n/3$. 
Therefore, the $x$-coordinates of the points $G((n-1)L,1)$ and $G(nL,1)$ differ by at least $R_n/3$. We connect these two points with a line segment of slope bounded above by $1$. In this way we obtain a piecewise linear $1$-Lipschitz function $f\colon \bigcup_{n\in \Z}[x_{n-1},x_n]\to (0,\infty)$ whose graph passes through the points $(x_n,y_n)=G(nL,1)$, $n\in \Z$. See Figure \ref{figure:construction_f}.

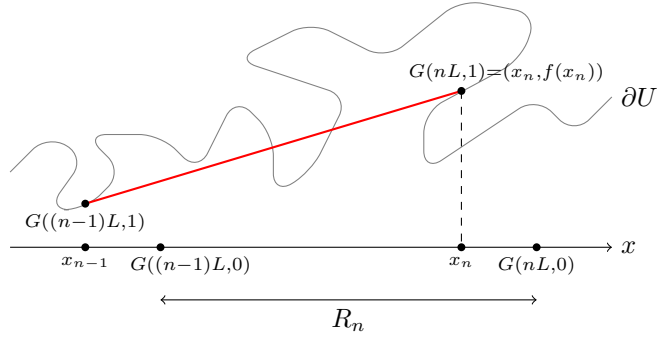
\begin{figure}
	\begin{tikzpicture}
		\draw[->] (-4,0)--(4,0) node[right] {$x$};
		
		\draw[rounded corners=8pt, color=black!50] (-4,1)--(-3.5,1.5)--(-3,1.2)--(-3.5,0.5)--(-3,0.5)--(-2.5,1)--(-3,1.5)--(-2,1.5)--(-1.5,1)--(-1,2)--(-0.5,1.5)--(-0.5,0.7)--(0,1)--(0.5,2)--(-1,2.5)--(0,3)--(1,2.5)--(2,3.3)--(3,3.1)--(2.8,2.5)--(1.5,1.8)--(1.5,1)--(3,2)--(3.5,1.5)--(4,2) node[right] {{\color{black}$\partial U$}};
		
		\draw[red, thick] (-3,0.58)--(2,2.08);
		\fill (3,0) circle (1.5pt) node[below] {$\scriptstyle G(nL,0)$};
		\fill (-2,0) circle (1.5pt) node[below, xshift=0.4cm] {$\scriptstyle G((n-1)L,0)$};
		\fill (2,2.08) circle (1.5pt) node[above, xshift=0.6cm] {$\scriptstyle G(nL,1)=(x_n,f(x_n))$};
		\fill (-3,0.58) circle (1.5pt) node[below,xshift=0cm] {$\scriptstyle G((n-1)L,1)$};
		\draw[dashed](2,2.08)--(2,0);
		\fill (2,0) circle (1.5pt) node[below] {$\scriptstyle x_n$};
		\fill (-3,0) circle (1.5pt) node[below] {$\scriptstyle x_{n-1}$};
		\draw[<->] (-2,-0.7)--(3,-0.7) node[below,pos=0.5] {$R_n$};
	\end{tikzpicture}
	\caption{The construction of the function $f$ in the proof of Theorem \ref{theorem:intro:graph:characterization}.}\label{figure:construction_f}
\end{figure}

Since $G|_{\R\times \{0\}}$ is an increasing homeomorphism onto $\R\times \{0\}$, we conclude that $\sum_{n=1}^\infty R_n=\infty$. As argued above, we have $x_n-x_{n-1}\geq R_n/3$ for each $n\in \Z$, which implies that $x_n\to \infty$ as $n\to\infty$. Similarly $x_n\to-\infty$ as $n\to-\infty$. Therefore, the domain of $f$ is $\R= \bigcup_{n\in \Z}[x_{n-1},x_n]$. Also, $\partial U$ contains the points $(x_n,y_n)$, $n\in \Z$, so a connectivity argument shows that each vertical line intersects $\partial U$. We will verify condition \ref{theorem:graph:characterization:3}, that is, 
$$\partial U\subset \{ (x,y)\in \R^2: a^{-1}f(x)\leq y\leq af(x)\}$$
for some $a=a(K)\geq 1$. Once this is done, then condition \ref{theorem:graph:characterization:4} follows from Proposition \ref{theorem:graph:generalization}.

We claim that, for each $\delta>0$, 
\begin{align}\label{graph:claim}
\text{if $u,v\in \R$ and $|u-v|\leq \delta f(u)$, then $f(u)\leq C(K,\delta)f(v)$.}
\end{align}
We fix $n\in \Z$. We have 
$$\frac{|(nL,0)-(nL,1)|}{|F(x_n,0)-(nL,1)|}\leq 1\,\,\, \text{and} \,\,\, \frac{|((n-1)L,0)-(nL,0)|}{|(nL,0)-(nL,1)|}=L$$
so
\begin{align*}
f(x_n)=y_n=|(x_n,0)-(x_n,y_n)| &\geq \eta(1)^{-1} |G(nL,0)-G(nL,1)|\\
&\geq \eta(1)^{-1} \eta(L)^{-1} |G((n-1)L,0)-G(nL,0)|\\
&= \eta(1)^{-1} \eta(L)^{-1}R_n.
\end{align*}
Also, by \eqref{graph:i}, we have
\begin{align*}
f(x_n)=y_n=|(x_n,0)-(x_n,y_n)|&\leq |G(nL,0)-G(nL,1)| \leq  R_n/3.
\end{align*}
Summarizing, we have $f(x_n)\simeq_K R_n$. With the same argument, one can show that $f(x_{n-1})\simeq_K R_n$. Since $f$ is linear on $[x_{n-1},x_n]$, it follows that 
\begin{align}\label{graph:ii}
C_1(K)^{-1}R_n\leq f(x)\leq C_1(K) R_n\,\,\, \text{for} \,\,\, x\in [x_{n-1},x_n].
\end{align}
Also, by \eqref{graph:i} we have
$$2R_n/3\leq |(x_n,0)- G((n-1)L,0)|\leq 4R_n/3.$$
As a consequence, 
$$|F(x_n,0)-((n-1)L,0)|\leq \eta(4/3)L$$
and
\begin{align*}
|F(x_n,0)-F(x_n,y_n)|&\leq \eta\left(\frac{f(x_n)}{|(x_n,0)-G((n-1)L,0)|}\right)|F(x_n,0)-((n-1)L,0)|\\
&\leq \eta(3C_1(K)/2)|F(x_n,0)-((n-1)L,0)|\\
&\leq \eta(3C_1(K)/2) \eta(4/3)L.
\end{align*}
Thus, 
\begin{align}\label{graph:iii}
1\leq |F(x_n,0)-F(x_n,y_n)| \leq C_2(K).
\end{align}

Suppose that $\delta>0$, $n,m\in \Z$, and $|x_n-x_m|\leq \delta f(x_n)=\delta y_n$. Since $F$ is $\eta$-quasisymmetric, by \eqref{graph:iii} we have 
\begin{align*}
|F(x_n,0)-F(x_m,0)|&\leq \eta(\delta) |F(x_n,0)-F(x_n,y_n)|\\
&\leq \eta(\delta)C_2(K)\\
&\leq \eta(\delta)C_2(K) |F(x_m,0)-F(x_m,y_m)|.
\end{align*}
Now, since $G$ is $\eta$-quasisymmetric, we obtain
$$|x_n-x_m| \leq \eta(\eta(\delta) C_2(K)) y_m.$$ 
Since $f$ is $1$-Lipschitz, this gives 
\begin{align}\label{graph:iv}
y_n\leq C_3(K,\delta) y_m. 
\end{align}

Now, let $u,v\in \R$ such that $|u-v|\leq \delta f(u)$. Suppose that $u\in [x_{n-1},x_n]$ and $v\in [x_{m-1},x_m]$. If $|m-n|\leq 1$, then then \eqref{graph:claim} follows immediately from \eqref{graph:ii}. Thus, we suppose that $|m-n|>1$. Without loss of generality, suppose that $m<n$. By \eqref{graph:ii} we have
$$|x_{n-1}-x_m|\leq |u-v|\leq \delta f(u)\leq \delta C_1(K) f(x_{n-1}).$$
Thus, by \eqref{graph:iv} and \eqref{graph:ii}, for $\delta'=\delta C_1(K)$ we have
$$f(u)\leq  C_1(K) f(x_{n-1})\leq C_1(K)C_3(K,\delta')f(x_m)\leq  C_1(K)^2C_3(K,\delta')f(v).$$
This completes the proof of the claim in \eqref{graph:claim}.

We now continue with the verification of \ref{theorem:graph:characterization:3}. Let $x\in [(n-1)L,nL]$ and note that  
$$k_\UHP((x,1),(nL,1))\leq L.$$
By Theorem \ref{theorem:gehring_osgood}, we have
$$k_\UHP(G(x,1), G(nL,1))\leq r(K).$$
The hyperbolic and quasihyperbolic metrics on $\UHP$ coincide. Hence,
$$\partial U\subset \{(x,y)\in \UHP: \dist_{h_\UHP}((x,y),\graph(f))\leq r(K)\}.$$
We will show the inclusion
\begin{align}\label{graph:final_claim}
\begin{aligned}
\{(x,y)\in \UHP: &\dist_{h_\UHP}((x,y),\graph(f))\leq r(K)\}\\&\subset \{(x,y)\in \R^2: a^{-1}f(x)\leq y\leq af(x)\}
\end{aligned}
\end{align}
for some $a=a(K)\geq 1$ and this will complete the proof of \ref{theorem:graph:characterization:3}.

Let $(u_0,v_0)\in \UHP$ be a point such that $\dist_{h_\UHP}((u_0,v_0),\graph(f))\leq r(K)$. Let $(u_1,f(u_1))\in \graph(f)$ such that $h_\UHP((u_0,v_0),(u_1,f(u_1)))\leq r(K)$. We use the formula
$$h_{\UHP}((u_0,v_0),(u_1,v_1))=2\arcsinh\frac{|(u_1,v_1)-(u_0,v_0)|}{2\sqrt{v_1v_0}}.$$ 
Thus, 
\begin{align}\label{theorem:graph:characterization:eq}
\frac{|(u_1,f(u_1))-(u_0,v_0)|}{\sqrt{f(u_1)v_0}} \leq 2\sinh(r(K)/2)=C_4(K).
\end{align}
We conclude that 
$$|f(u_1)-v_0| \leq C_4(K)\sqrt{f(u_1)v_0}.$$
Equivalently, 
$$\left|\left(\frac{f(u_1)}{v_0}\right)^{1/2}-\left(\frac{v_0}{f(u_1)}\right)^{1/2}\right|\leq C_4(K).$$
Setting $t= \left(\frac{f(u_1)}{v_0}\right)^{1/2}$, we have $|t-t^{-1}|\leq C_4(K)$, which implies that $C_5(K)^{-1}\leq t\leq C_5(K)$. Equivalently, $f(u_1)\simeq_K v_0$. Next, by  \eqref{theorem:graph:characterization:eq} we have
\begin{align*}
|u_1-u_0| \leq C_4(K) \sqrt{f(u_1)v_0} =C_4(K) t^{-1}f(u_1)\leq C_4(K)C_5(K) f(u_1).
\end{align*}
By the claim in \eqref{graph:claim}, we have $f(u_1)\leq C_6(K) f(u_0)$. Also, since $f$ is $1$-Lipschitz, we have $f(u_0)\leq (1+C_4(K)C_5(K))f(u_1)$. Therefore, $f(u_0)\simeq_K f(u_1)\simeq_K v_0$. That is, there exits $a=a(K)\geq 1$ such that
$$ a^{-1}f(u_0)\leq v_0\leq af(u_0).$$
This proves the inclusion in \eqref{graph:final_claim}.
\end{proof}

\begin{example}\label{example:function}
We apply Theorem \ref{theorem:intro:graph} in two fundamental examples that highlight the importance of the rate of decay of $f$ in Theorem \ref{theorem:intro:graph}.

Let $p>-1$ and $f(x)=(|x|+1)^{-p}$ for $x\in \R$. We have $|f'(x)|\leq |p|$ for $x\neq 0$ so $f$ is $|p|$-Lipschitz. We compute an antiderivative of $1/f$. For $x\in \R$ we have
$$F(x)=\int_0^x \frac{dt}{f(t)}=\int_{0}^x (|t|+1)^p \, dt =\begin{cases}(p+1)^{-1}((|x|+1)^{p+1}-1) & x\geq 0\\
-(p+1)^{-1}((|x|+1)^{p+1}-1) & x<0.
\end{cases}.$$
It is an exercise to show that $F$ is quasisymmetric (see \cite{Heinonen:metric}*{Exercise 10.3} and also \cite{LehtoVirtanen:quasiconformal}*{Lemma II.7.1}). By Theorem \ref{theorem:intro:graph}, the curves $y=(|x|+1)^{-p}$ and $y=0$ can be mapped to a pair of parallel lines with a quasiconformal homeomorphism of $\C$.  This example was also investigated independently by Yusheng Luo with different techniques.

Next, let $f(x)=e^{-|x|}$, $x\in \R$, which is $1$-Lipschitz. An antiderivative of $1/f$ is
$$F(x)= \int_0^x e^{|t|}\, dt= \begin{cases}
e^{|x|}-1 & x\geq 0 \\
-(e^{|x|}-1) & x<0
\end{cases}.$$
Note that for $x>0$ we have
$$\frac{F(2x)-F(x)}{F(x)-F(0)}= \frac{e^{2x}-e^x}{e^x-1}=e^x$$
which is not bounded, so $F$ is not quasisymmetric. As a conclusion, by Theorem \ref{theorem:intro:graph} the curves $y=e^{-|x|}$ and $y=0$ cannot be mapped to a pair of parallel lines with a quasiconformal homeomorphism of $\C$.
\end{example}

For $\alpha>1$ consider the polynomial cusp $C_\alpha=\{(x,y)\in \R^2: |y|=x^\alpha, \, 0\leq x\leq 1\}$. We now restate and prove Corollary \ref{corollary:intro:cusps}.
\begin{corollary}
For every $\alpha,\beta>1$ there exists a quasiconformal map $f$ of $\C$ such that $f(C_\alpha)=C_\beta$. 
\end{corollary}

\begin{proof}
We show that there exists a quasiconformal homeomorphism $f$ of $\C$ that maps the real line onto itself and maps $\{(x,y)\in \R^2: y=x^\alpha,\, 0\leq x\leq 1\}$, onto the semicircle $A=\{(x,y)\in \R^2:x^2+(y-1/2)^2=1/4,\, x\geq 0\}$, which is tangent to the real line at $(0,0)$. By reflection, we can define a quasiconformal homeomorphism of $\C$ that is equal to $f$ in the upper half-plane and maps $\{(x,y)\in \R^2: y=-x^\alpha,\, 0\leq x\leq 1\}$, onto the reflection of $A$. Since any polynomial cusp $C_\alpha$ can be mapped to two copies of $A$, this is sufficient to give the conclusion of the theorem. We now focus on proving our claim.

Consider the map
$$\phi(z)=\frac{1}{\bar z}= \frac{x}{x^2+y^2}+i \frac{y}{x^2+y^2}.$$
Note that for $0<x\leq1$ we have
$$\phi(x+ix^\alpha)= \frac{x}{x^2+x^{2\alpha}}+i \frac{x^\alpha}{x^2+x^{2\alpha}}= \frac{1}{x(1+x^{2\alpha-2})}+ i \frac{x^{\alpha-1}}{x(1+x^{2\alpha-2})}.$$
Let 
$$t(x)= \re(\phi(x+ix^\alpha))=\frac{1}{x(1+x^{2\alpha-2})},\quad 0<x\leq 1. $$ 
Note that $t$ is strictly decreasing, and takes values in $[1/2,\infty)$. Let $s\colon [1/2,\infty)\to (0,1]$ be the inverse of $t$. We have
$$ \im(\phi(x+ix^\alpha)) = t(x)x^{\alpha-1}=t(x)s(t(x))^{\alpha-1}.$$
Thus, $\{\phi(x+ix^\alpha): 0<x\leq 1\}$ is the graph of the function $h(t)=ts(t)^{\alpha-1}$, $t\geq 1/2$. We extend $h$ continuously by defining $h(t)=|t|s(|t|)^{\alpha-1}$ for $|t|\geq 1/2$ and letting $h(t)=1/2$ for $t\in (-1/2,1/2)$. 

We show that $h$ is Lipschitz continuous. We have $\frac{1}{2x}\leq t(x)\leq \frac{1}{x}$ for $0<x\leq 1$, so $\frac{1}{2t}\leq s(t)\leq \frac{1}{t}$ for $t\geq 1/2$.  Note that 
$$|t'(x)|=(x+x^{2\alpha-1})^{-2}(1+(2\alpha-1)x^{2\alpha-2}) \geq t(x)^2$$
for $x\in (0,1)$, so $|s'(t)|\leq t^{-2}\leq 4s(t)^2$ for $t\in (1/2,\infty)$. Thus,
\begin{align*}
|h'(t)|&=|s(t)^{\alpha-1}+t(\alpha-1)s(t)^{\alpha-2} s'(t)|\leq 1+4(\alpha-1)ts(t)s(t)^{\alpha-1} \leq 1+4(\alpha-1)
\end{align*}
for $t\in (1/2,\infty)$. The same bound applies for $t\in (-\infty,-1/2)$ and we have $h'(t)=0$ for $t\in (-1/2,1/2)$. Thus, $h$ is Lipschitz. 

Since $\frac{1}{2t}\leq s(t)\leq \frac{1}{t}$ for $t\geq 1/2$, we have
$$ \frac{1}{2^{\alpha-1}}|t|^{2-\alpha}\leq h(t) \leq |t|^{2-\alpha}$$
for $|t|\geq 1/2$. Also, for $|t|\geq 1/2$ we have  $|t|\leq 1+|t|\leq 3|t|$. 
Hence, for $C=2^{\alpha-1}\max\{3^{\alpha-2},3^{2-\alpha}\}$ we have
$$C^{-1}(1+|t|)^{2-\alpha}\leq h(t) \leq C(1+|t|)^{2-\alpha}.$$
The inequality is also true for $|t|\leq 1/2$. We set $g(t)=(1+|t|)^{2-\alpha}$, $t\in \R$, and we have $C^{-1}g(t)\leq h(t)\leq Cg(t)$ for all $t\in \R$. 

By Example \ref{example:function}, an antiderivative $G$ of $1/g$ is quasisymmetric. Also, if $H$ is an antiderivative of $1/h$, then for $t_1,t_2\in \R$ we have
$$ C^{-1} |G(t_1)-G(t_2)|\leq |H(t_1)-H(t_2)|\leq C |G(t_1)-G(t_2)|,$$
so $H$ is also quasisymmetric. By Theorem \ref{theorem:intro:graph}, there exists a quasiconformal map $F$ of $\C$ that maps the line $y=0$ to itself and the graph of $h$ onto the line $y=1$. We may assume that $F(1/2,1/2)=(1,0)$. The composition $\phi^{-1}\circ F\circ \phi$ is a quasiconformal map of $\C$ that maps the real line onto itself and maps $\{(x,y):y=x^\alpha,\, 0\leq x\leq 1\}$, onto the semicircle $A$.  
\end{proof}

\section{Quasiconformal extension results}\label{section:extension_general}

We state a variant of the Beurling--Ahlfors extension theorem  for locally quasisymmetric homeomorphisms of $\R$. 

\begin{theorem}\label{theorem:qc_extension}
Let $\eta$ be a distortion function. Let $h\colon \R\to \R$ be an increasing homeomorphism that is $\eta$-quasi\-symmetric in $[t,t+1]$ for each $t\in \R$ and  satisfies $h(n)=n$ for each $n\in \Z$. Then there exists an extension of $h$ to a homeomorphism $H\colon \R\times [0,1]\to \R\times [0,1]$ such that $H|_{\R\times \{1\}}$ is the identity map and $H|_{\R\times (0,1)}$ is $K(\eta)$-quasiconformal.  In addition, if $h$ commutes with $z\mapsto z+N$  for some $N\in \Z$, then the extension $H$ may be taken to have the same property.
\end{theorem}
The construction is given in \cite{AstalaIwaniecMartin:quasiconformal}*{Theorem 5.8.1}, under slightly different assumptions. We only provide a sketch of the proof. 

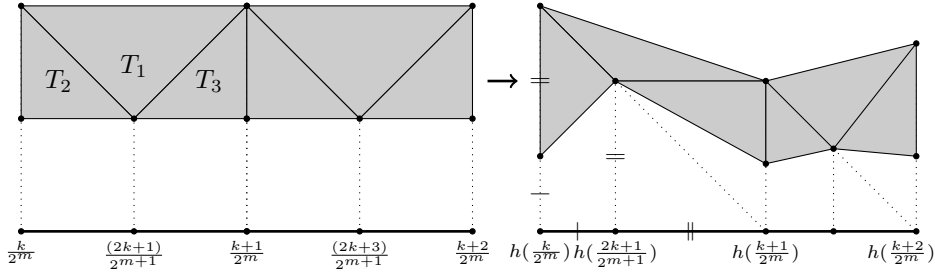
\begin{figure}
	\centering
	\begin{tikzpicture}
	
	\draw[->, line width=1pt] (6.2,2)--(6.6,2);	
	
		\begin{scope}
			\draw[line width=1pt] (0,0)--(3,0);
			\draw[dotted] (0,0)--(0,3);
			\draw[dotted] (3,0)--(3,3);
			\draw[dotted] (1.5,0)--(1.5,1.5);			
			
			\draw[fill=black!20!white] (0,3)--(3,3)--(1.5,1.5)--cycle;
			\draw[fill=black!20!white] (0,3)--(0,1.5)--(1.5,1.5)--cycle;
			\draw[fill=black!20!white] (1.5,1.5)--(3,1.5) --(3,3)--cycle;
			\node at (1.5,2.2) {$T_1$};
			\node at(0.5, 2) {$T_2$};		
			\node at(2.5, 2) {$T_3$};	
			
			\draw[fill=black] (0,3) circle (1pt);
			\draw[fill=black] (3,3) circle (1pt);
			\draw[fill=black] (0,1.5) circle (1pt);
			\draw[fill=black] (1.5,1.5) circle (1pt);
			\draw[fill=black] (3,1.5) circle (1pt);
			
			\draw[fill=black] (0,0) circle (1pt) node[anchor=north] {$\scriptstyle \frac{k}{2^{m}}$};
			\draw[fill=black] (1.5,0) circle (1pt) node[anchor=north] {$\scriptstyle \frac{(2k+1)}{2^{m+1}}$};
			\draw[fill=black] (3,0) circle (1pt) node[anchor=north] {$\scriptstyle \frac{k+1}{2^{m}}$};
			
		\end{scope}
		
		\begin{scope}[shift={(3,0)}]
			\draw[line width=1pt] (0,0)--(3,0);
			\draw[dotted] (0,0)--(0,3);
			\draw[dotted] (3,0)--(3,3);
			\draw[dotted] (1.5,0)--(1.5,1.5);			
			
			\draw[fill=black!20!white] (0,3)--(3,3)--(1.5,1.5)--cycle;
			\draw[fill=black!20!white] (0,3)--(0,1.5)--(1.5,1.5)--cycle;
			\draw[fill=black!20!white] (1.5,1.5)--(3,1.5) --(3,3)--cycle;

			\draw[fill=black] (0,3) circle (1pt);
			\draw[fill=black] (3,3) circle (1pt);
			\draw[fill=black] (0,1.5) circle (1pt);
			\draw[fill=black] (1.5,1.5) circle (1pt);
			\draw[fill=black] (3,1.5) circle (1pt);
			
			\draw[fill=black] (0,0) circle (1pt);
			\draw[fill=black] (1.5,0) circle (1pt) node[anchor=north] {$\scriptstyle \frac{(2k+3)}{2^{m+1}}$};
			\draw[fill=black] (3,0) circle (1pt) node[anchor=north] {$\scriptstyle \frac{k+2}{2^{m}}$};
			
		\end{scope}

		\begin{scope}[shift={(6.9,0)}]
			\draw[line width=1pt] (0,0)--(3,0);
			\draw[dotted] (0,0)--(0,3);
			\draw[dotted] (3,0)--(3,2);
			\draw[dotted] (1,0)--(1,2);			
			
			\draw[fill=black!20!white] (0,3)--(3,2)--(1,2)--cycle;
			\draw[dotted] (0,3)--(3,0);
			
			\draw[fill=black!20!white] (0,3)--(0,1)--(1,2)--cycle;
			\draw[fill=black!20!white] (1,2)--(3,0.9) --(3,2)--cycle;

			\draw[fill=black] (0,0) circle (1pt) node[anchor=north] {$\scriptstyle  h(\frac{k}{2^{m}})$};
			\draw[fill=black] (1,0) circle (1pt) node[anchor=north] {$\scriptstyle  h(\frac{2k+1}{2^{m+1}})$};
			\draw[fill=black] (3,0) circle (1pt) node[anchor=north] {$\scriptstyle  h(\frac{k+1}{2^{m}})$};

			\draw [draw,decoration={markings, mark=at position .5  with {\node[transform shape] {$\scriptstyle|$};}}]
  (0,0) edge[decorate] (1,0)
  (0,0) edge[decorate] (0,1);
  			\draw [draw,decoration={markings, mark=at position .5  with {\node[transform shape] {$\scriptstyle \parallel$};}}]
  (1,0) edge[decorate] (3,0)
  (1,0) edge[decorate] (1,2)
  (0,1) edge[decorate] (0,3);

  		\draw[line width=1pt] (3,0)--(5,0);
  		\draw[dotted] (5,0)--(5,2);
  		\draw[fill=black!20!white] (3,2)--(5,2.5)-- (5,1)--(3.9,1.1)--(3,0.9)--cycle;
  		\draw[dotted] (3.9,1.1)--(5,0);
  		\draw[dotted] (3.9,1.1)--(3.9,0);
  		\draw (3.9,1.1)--(3,2);
  		\draw (3.9,1.1)--(5,2.5);
  		\draw[fill=black] (5,0) circle (1pt);
  		\node[anchor=north] at (4.8,0)  {$\scriptstyle h( \frac{k+2}{2^m})$};
  		\draw[fill=black] (3.9,0) circle (1pt);
  		\draw[fill=black] (3.9,1.1) circle (1pt);
  		\draw[fill=black] (5,2.5) circle (1pt);
  		\draw[fill=black] (5,1) circle (1pt);
  		
  		\draw[fill=black] (0,3) circle (1pt);
		\draw[fill=black] (3,2) circle (1pt);
		\draw[fill=black] (0,1) circle (1pt);
		\draw[fill=black] (1,2) circle (1pt);
		\draw[fill=black] (3,0.9) circle (1pt);
		\end{scope}
	\end{tikzpicture}
	\caption{The quasiconformal extension constructed in Theorem \ref{theorem:qc_extension}. Marks show congruent line segments.}\label{fig:linear_extension}
\end{figure}

\begin{proof}
Consider the map
$$H(k2^{-m}, 2^{-m})= (h(k2^{-m}), h((k+1)2^{-m})-h(k2^{-m}))$$
for $k\in \Z$ and $m\in \N\cup \{0\}$. At each level $m\in \N\cup \{0\}$ there are three different types of triangles with vertex sets
\begin{align*}
V_1&=\{(k2^{-m},2^{-m}), ((k+1)2^{-m}, 2^{-m}), ((2k+1)2^{-m-1}, 2^{-m-1})\},\\
V_2&= \{(k2^{-m}, 2^{-m}),(k2^{-m}, 2^{-m-1}),((2k+1)2^{-m-1}, 2^{-m-1})\},\\
V_3&= \{((k+1)^{-m},2^{-m}),((2k+1)2^{-m-1}, 2^{-m-1}),((k+1)2^{-m}, 2^{-m-1})\}.
\end{align*}
Denote by $T_i$ the triangle that is the convex hull of $V_i$, $i=1,2,3$. The map $H$ is extended linearly to each triangle. By construction and using the fact that $h$ is an increasing homeomorphism, the image of each triangle is non-degenerate, and distinct triangles are mapped to distinct triangles. See Figure \ref{fig:linear_extension} for an illustration. Hence, $H$ extends to a homeomorphism from $\R\times [0,1]$ onto itself. Note that $H$ is the identity map on $\R\times \{1\}$ because $h(n)=n$ for $n\in \Z$. Also, if $h$ commutes with $z\mapsto z+N$ for some $N\in \Z$, then $H$ has the same property. In order to justify that $H$ is quasiconformal in $\R\times (0,1)$ it suffices to show that the triangles $H(T_i)$, $i=1,2,3$, have angles uniformly bounded from below away from $0$. We direct the reader to \cite{AstalaIwaniecMartin:quasiconformal}*{Theorem 5.8.1} for further details.
\end{proof}

\subsection{Extension in a strip}
The next result provides sufficient conditions so that a homeomorphism of the lines $y=0$ and $y=1$ extends quasiconformally to the strip $0<y<1$. 

\begin{proposition}\label{theorem:extension_strip}
Let $\eta$ be a distortion function and $a\geq 1$. Let $h\colon \R\times \{0,1\}\to \R\times \{0,1\}$ be a homeomorphism such that $h(\R\times \{0\})=\R\times \{0\}$ and $h|_{\R\times \{0\}}$ is increasing. Suppose that for each $t\in \R$  
\begin{enumerate}[label=\normalfont(\arabic*)]
\item the map $h$ is quasisymmetric in $[t,t+1]\times \{0\}$ and in $[t,t+1]\times \{1\}$ and
\item the points $(t,0),(t,1),(t+1,0), (t+1,1)$ are mapped to points with mutual distances in the interval $[a^{-1},a]$.
\end{enumerate}
Then there exists an extension of $h$ to a homeomorphism $H\colon \R\times [0,1]\to \R\times [0,1]$ that is $K(\eta,a)$-quasiconformal in $\R\times (0,1)$. In addition, if $h$ commutes with $z\mapsto z+(N,0)$ for some $N\in \Z$, then the extension $H$ may be taken to have the same property.
\end{proposition}

\begin{proof}
Since $|h(t,0)-h(t,1)|\leq a$ for each $t\in \R$, we conclude that $h|_{\R\times \{1\}}$ is increasing. Thus, $h$ is increasing on both lines. For $t\in \R$ let $Q_t=[t,t+1]\times [0,1]$ and denote by $T_t$ the closed trapezoid whose bases are $h([t,t+1]\times \{0\})$ and $h([t,t+1]\times \{1\})$; see Figure \ref{fig:composition}. The vertices of $T_t$ have mutual distances in $[a^{-1},a]$, by assumption. 

For each $n\in \Z$, the trapezoid $T_n$ can be mapped to the square $Q_n=[n,n+1]\times [0,1]$ via an $L(a)$-bi-Lipschitz map $\psi$ that is linear on each edge of $T_n$; moreover, we may have that $\psi$ maps the bottom and top sides of $T_n$ to the bottom and top sides of $Q_n$, respectively, in an increasing way. To see this, we split $T_n$ into two triangles $A_n,B_n$ by a diagonal from the lower left to the upper right vertex. Note that the height of $A_n$ and $B_n$ is equal to $1$, so their area is comparable to $1$, depending on $a$. This implies that the angles of $A_n$ and $B_n$ are uniformly bounded from below away from $0$. Thus, $A_n$ and $B_n$ can be mapped with a linear bi-Lipschitz map to two right triangles arising by dividing a unit square with a diagonal; the bi-Lipschitz constant depends only on $a$.  

By pasting these bi-Lipschitz maps, we construct an $L(a)$-bi-Lipschitz map $\psi\colon \R^2\times [0,1]\to \R^2\times [0,1]$ such that $\psi(\R\times \{0\})=\R\times \{0\}$, $\psi(\R\times \{1\})=\R\times \{1\}$, and $\psi(T_n)=[n,n+1]\times [0,1]$ for each $n\in \Z$. 

\begin{figure}
	\centering
	\begin{tikzpicture}
		\begin{scope}		
		\draw (0,0)--(3,0);
		\draw (0,1)--(3,1);
		\draw[line width=1.5pt] (1,0)--(2,0);
		\draw[line width=1.5pt] (1,1)--(2,1);
		\draw[dotted] (1,0)--(1,1);
		\draw[dotted] (2,0)--(2,1);
		\node at (1.5, 0.5) {$Q_n$};
		\fill (1,0) circle (1.5pt) node[below,yshift=-2pt] {$n$};
		\fill (2,0) circle (1.5pt) node[below] {$n+1$};
		\end{scope}
		
		\begin{scope}[shift={(4.5,0)}]		
		\draw (0,0)--(3,0);
		\draw (0,1)--(3,1);
		\draw[line width=1.5pt] (1,0)--(2,0);
		\draw[line width=1.5pt] (1.7,1)--(2.5,1);
		\draw[dotted] (1,0)--(1.7,1);
		\draw[dotted] (2,0)--(2.5,1);
		\node at (1.8, 0.5) {$T_n$};
		\fill (1,0) circle (1.5pt);
		\fill (2,0) circle (1.5pt);
		\end{scope}
		
		\begin{scope}[shift={(9,0)}]
		\draw (0,0)--(3,0);
		\draw (0,1)--(3,1);
		\draw[line width=1.5pt] (1,0)--(2,0);
		\draw[line width=1.5pt] (1,1)--(2,1);
		\draw[dotted] (1,0)--(1,1);
		\draw[dotted] (2,0)--(2,1);
		\fill (1,0) circle (1.5pt) node[below,yshift=-2pt] {$n$};
		\fill (2,0) circle (1.5pt) node[below] {$n+1$};
		\end{scope}
		
		\draw[->] (3.5,0.5)--(4,0.5) node[pos=0.5,anchor=south] {$h$};
		\draw[->] (8,0.5)--(8.5,0.5) node[pos=0.5,anchor=south] {$\psi$};
	\end{tikzpicture}
		\caption{The map $\omega= \psi\circ h$.}\label{fig:composition}
\end{figure}
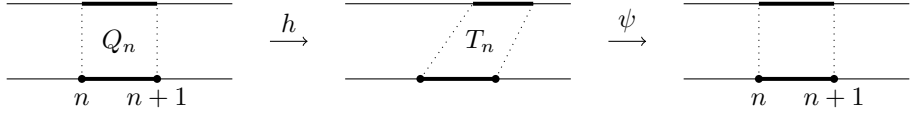

Consider the map $\omega=\psi\circ h$ on $\R\times \{0,1\}$.  The map $\omega$ is quasisymmetric on $[t,t+1]\times \{0\}$ for each $t\in \R$. Moreover,  $\omega$ is increasing on $\R\times \{0\}$ and $\omega(n,0)=(n,0)$, $n\in \Z$; see Figure \ref{fig:composition}. By Theorem \ref{theorem:qc_extension}, the map $\omega|_{\R\times \{0\}}$  has an extension to a homeomorphism $(u,v)$ of $\R\times [0,1]$ that is quasiconformal in $\R\times (0,1)$, so that $(u,v)|_{\R\times \{1\}}$ is the identity map. The map $\omega_0(x,y)=(u(x,2y), 2^{-1}v(x,2y))$ is a homeomorphism from $\R\times[0,1/2]$ onto itself that extends $\omega|_{\R\times \{0\}}$, is quasiconformal in $\R\times (0,1/2)$, and is the identity map on $\R\times \{1/2\}$.  Similarly, $\omega|_{\R\times \{1\}}$ has an extension to a  homeomorphism $\omega_1$ of $\R\times [1/2,1]$ that is quasiconformal in $\R\times (1/2,1)$, and this extension is also the identity map on $\R\times \{1/2\}$. 

Therefore, by the quasiconformal removability of lines, $\omega|_{\R\times \{0,1\}}$ has an extension to a homeomorphism of $\R\times [0,1]$ that is quasiconformal in $\R\times (0,1)$, quantitatively. Thus, $h=\psi^{-1}\circ \omega$ has an extension to a homeomorphism of $\R\times [0,1]$ that is quasiconformal in $\R\times (0,1)$, quantitatively.  

Finally, we show that if the map $h$ commutes with $z\mapsto z+N$ for some $N\in \Z$, the extension may be taken to have the same property. Indeed, in this case the trapezoids $T_n$ satisfy $T_{n+N}=T_n+(N,0)$, $n\in \Z$, so the piecewise linear map $\psi$ commutes with $z\mapsto z+(N,0)$. Also, the extensions given by Theorem \ref{theorem:qc_extension} have this property as well.
\end{proof}

\begin{corollary}\label{corollary:extension_strip}
Let $\eta$ be a distortion function. Let $h\colon \R\times \{0,1\}\to \R\times \{0,1\}$ be an $\eta$-quasi\-symmetric homeomorphism such that $h(\R\times \{0\})=\R\times \{0\}$ and $h|_{\R\times \{0\}}$ is increasing. Then there exists an extension of $h$ to a $K(\eta)$-quasiconformal homeomorphism of $\C$.
\end{corollary}

\begin{proof}
For each $t\in \R$ we have 
$$ \frac{|h(t,1)-h(t,0)|}{|h(t,0)-h(t+1,0)|} \leq  \eta(1).$$
If $h(t,0)=(y,0)$, then 
$$|h(t,0)-h(t+1,0)| = \frac{|h(t,0)-h(t+1,0)|}{|(y,0)-(y,1)|} \leq \eta\left( \frac{1}{|(t,0)-h^{-1}(y,1)|}\right) \leq \eta(1).$$
Therefore, $1\leq |h(t,1)-h(t,0)|\leq \eta(1)^2$. This implies that $h|_{\R\times \{1\}}$ is increasing as well. Moreover, 
$$|h(t,0)-h(t+1,0)|\geq \eta(1)^{-1} |h(t,1)-h(t,0)|  \geq \eta(1)^{-1}.$$
By symmetry, the above bounds remain true if we switch the roles of $0$ and $1$ in the $y$-coordinates. Thus, there exists $a=a(\eta)>0$ such that for each $t\in \R$ the vertices of the square $Q_t=[t,t+1]\times [0,1]$ are mapped under $h$ to points with mutual distances in the interval $[a^{-1},a]$. Proposition \ref{theorem:extension_strip} implies that $h$ has a $K(\eta)$-quasiconformal extension $H$ in the strip $\R\times (0,1)$. Using the Beurling--Ahlfors extension theorem (Theorem \ref{theorem:beurling_ahlfors}) and the quasiconformal removability of lines, we may extend $H$ quasiconformally to the entire plane. 
\end{proof}

\subsection{Extension in an annulus}
We give the analogues of the extension results of the preceding section for annuli. First we need an elementary lemma about lifts of circle homeomorphisms.

\begin{lemma}\label{lemma:lift_close}
Let $f,g\colon \mathbb S(0,1)\to \mathbb S(0,1)$ be orientation-preserving homeomorphisms. Then there exist lifts $F,G$ of $f,g$, respectively, under the map $z\mapsto e^{2\pi i z}$ such that 
$$\|F-G\|_\infty\leq \frac{1}{2} \|f-g\|_\infty.$$
\end{lemma}
\begin{proof}
Let $F,G\colon \R\to \R$ be arbitrary lifts of $f,g$, respectively. The maps $F, G$ are increasing homeomorphisms satisfying $F(\theta+1)=F(\theta)+1$ and $G(\theta+1)=G(\theta)+1$ for $\theta\in \R$. We will correct the map $G$ by an integer. Let $D(\theta)=F(\theta)-G(\theta)$, $\theta\in \R$. Observe that $|e^{2\pi i \theta}-1| \geq 4 \dist_e(\theta,\Z)$ for $\theta \in \R$. Thus, for each $\theta\in \R$,
\begin{align*}
\|f-g\|_\infty &\geq |f(e^{2\pi i\theta})-g(e^{2\pi i \theta})|=|e^{2\pi i F(\theta)}-e^{2\pi i G(\theta)}|= |e^{2\pi i D(\theta)}-1| \\
&\geq 4\dist_e(D(\theta),\Z).
\end{align*}
If $\|f-g\|_\infty<2$, then $\dist_e(D(\theta),\Z)<1/2$ for each $\theta\in \R$. By continuity, there exists $k\in \Z$ such that $\dist_e(D(\theta),\Z)= |D(\theta)-k|$ for each $\theta\in \R$. In particular,
$$\|F-(G+k)\|_\infty =\|D-k\|_\infty\leq  \frac{1}{4} \|f-g\|_\infty.$$
Next, suppose that $\|f-g\|_\infty\geq 2$, so in fact $\|f-g\|_\infty=2$. For $0\leq \theta_1\leq \theta_2\leq 1$ we have
\begin{align*}
|D(\theta_2)-D(\theta_1)|&=|(F(\theta_2)-F(\theta_1))- (G(\theta_2)-G(\theta_1 ))| \\
&\leq \max \{ F(\theta_2)-F(\theta_1), G(\theta_2)-G(\theta_1 )\}\leq 1.
\end{align*}
Thus, given that $D$ is $1$-periodic, the oscillation of $D$ is at most $1$. This implies that there exists $k\in \Z$ such that 
$$\|F-(G+k)\|_\infty =\|D-k\|_\infty \leq 1= \frac{1}{2}\|f-g\|_\infty.$$ 
This completes the proof. 
\end{proof}

\begin{proposition}\label{theorem:extension_annulus}
Let $\eta$ be a distortion function and let $N\in \N$ with $N\geq 2$ and $a\geq 1$. Let $h\colon  \mathbb S(0,1)\cup \mathbb S(0,e^{2\pi/N}) \to \mathbb S(0,1)\cup \mathbb S(0,e^{2\pi/N}) $ be a homeomorphism that is orientation-pre\-serving and preserves each of the two circles. Suppose that for each $t\in \R$
\begin{enumerate}[label=\normalfont(\arabic*)]
\item the map $h$ is $\eta$-quasisymmetric in $B_r(t)=\{re^{i\theta}: t \frac{2\pi}{N}\leq \theta\leq t\frac{2\pi}{N}+ \frac{2\pi}{N}\}$ for $r=1$ and $r=e^{2\pi/N}$ and
\item the four endpoints of the arcs $B_1(t), B_{e^{2\pi/N}}(t)$ are mapped to points with mutual distances in the interval $[a^{-1} N^{-1},a N^{-1}]$. 
\end{enumerate}
Then there exists an extension of $h$ to a homeomorphism $H\colon \bar{\mathbb A}(0;1,e^{2\pi/N})\to \bar{\mathbb A}(0;1,e^{2\pi/N})$ that is $K(\eta,a)$-quasiconformal in ${\mathbb A}(0;1,e^{2\pi/N})$.
\end{proposition}

Note that the quasiconformal dilatation of the extension does not depend on $N$ and in particular on the distance of the circles $\mathbb S(0,1),\mathbb S(0,e^{2\pi/N})$. 

\begin{proof}
Let $L=e^{2\pi/N} \in (1,e^\pi]$ and observe that $L-1\simeq N^{-1}$ with uniform constants. Let $f(z)=h(z)$ and $g(z)=L^{-1}h(Lz)$ for $z\in \mathbb S(0,1)$. We have
\begin{align*}
|f(z)-g(z)|&= |h(z)- L^{-1} h(Lz)|\leq |h(z)- L^{-1}h(z)|+|L^{-1}h(z)- L^{-1} h(Lz)|\\
&\leq \frac{L-1}{L}+\frac{aN^{-1}}{L}\leq C_1(a)N^{-1}. 
\end{align*}
By Lemma \ref{lemma:lift_close}, there exist lifts $F,G$ of $f,g$, respectively, under the map $z\mapsto e^{-2\pi i z}$ such that 
\begin{align}\label{theorem:extension:annulus:bound}
|F(t)-G(t)|\leq \frac{1}{2}C_1(a)N^{-1}\,\,\, \text{for each $t\in \R$.}
\end{align}

We define $\widetilde h(t,0)=(NF(t/N),0)\equiv NF(t/N)$, which is a lift of $h|_{\mathbb S(0,1)}$ under $\psi(z)=e^{-2\pi i z/N}$. We also define $\widetilde h(t,1)=(NG(t/N),1)\equiv NG(t/N)+i$, which is a lift of $h|_{\mathbb S(0,L)}$ under $\psi$. We will verify the assumptions of Proposition \ref{theorem:extension_strip} for the map $\widetilde h$; assuming this, we can extend $\widetilde h$ to a $K(\eta,a)$ quasiconformal homeomorphism $\widetilde H$ of the strip $\R\times (0,1)$ that commutes with $z\mapsto z+(N,0)$ and thus descends to a $K(\eta,a)$-quasiconformal homeomorphism $H$ of the annulus $\mathbb A(0;1,L)$ that extends $h$.

By the definition of $\widetilde h$ and \eqref{theorem:extension:annulus:bound}, for each $t\in \R$ we have 
$$|\widetilde h(t,0)- \widetilde h(t,1)|\leq 1+ N|F(t/N)-G(t/N)|\leq 1+ \frac{1}{2}C_1(a).$$
For $s=0,1$ we will show that $\widetilde h([t,t+1]\times \{s\})$ has length comparable to $1$, depending on $a$, and that $\widetilde h|_{[t,t+1]\times \{s\}}$ is quasisymmetric, depending only on $\eta$ and $a$. These claims verify the assumptions of Proposition \ref{theorem:extension_strip}. We now prove the claims. 

Let $t\in \R$. Note that $\psi([t,t+1]\times \{0\})=B_1(-t-1)$. Since $N\geq 2$, we have
\begin{align}\label{theorem:extension_annulus:ineq:1}
C_1^{-1}N^{-1}|z-w|\leq |\psi(z)-\psi(w)|\leq C_1 N^{-1}|z-w|,\,\,\, z,w \in [t,t+1]\times \{0\},
\end{align}
where $C_1\geq 1$ is a uniform constant. 

By assumption, the distance of the endpoints of the arc $h(B_1(-t-1))$ lies in the interval $[a^{-1}N^{-1},aN^{-1}]$. If this arc has length greater than $\pi$, then the complementary arc has length at most $aN^{-1}\pi/2$. Yet, it is equal to the union of the arcs $h(B_1(-t-k))$, $k=2,\dots,N$, each of which has length at least $a^{-1}N^{-1}$. We conclude that
$$a^{-1}N^{-1}\cdot (N-1) \leq aN^{-1} \frac{\pi}{2}$$
so $2\leq N\leq 1+a^2\pi/2$. Thus, the length of $h(B_1(-t-1))$ is comparable to $1$, depending on $a$. This implies that the length of $\widetilde h([t,t+1]\times \{0\})$ is also comparable to $1$ and $\psi$ is bi-Lipschitz on $\widetilde h([t,t+1]\times \{0\})$, depending only on $a$. 

Now, if the arc $h(B_1(-t-1))$ has length less than $\pi$, then its length is comparable to $N^{-1}$, depending on $a$, and we have
\begin{align}\label{theorem:extension_annulus:ineq:2}
 C_2 N^{-1} |z-w|\leq |\psi(z)-\psi(w)|\leq C_2 N^{-1}|z-w|,\,\,\, z,w\in \widetilde h([t,t+1]\times \{0\}),
\end{align}
where $C_2\geq 1$ is a uniform constant. Therefore, \eqref{theorem:extension_annulus:ineq:2} holds in both cases if the constant $C_2$ is allowed to depend on $a$. Also, note that the length of $\widetilde h([t,t+1]\times \{0\})$ is comparable to $1$, depending on $a$, in both cases.  

By \eqref{theorem:extension_annulus:ineq:1},  \eqref{theorem:extension_annulus:ineq:2}, and the relation $\psi\circ \widetilde h=h\circ \psi$ we conclude that $\widetilde h$ is $\widetilde \eta$-quasi\-symmetric on $[t,t+1]\times \{0\}$ for some $\widetilde \eta$ that depends only on $\eta$ and $a$.  With the same argument one shows that $\widetilde h$ is $\widetilde \eta$-quasisymmetric on $[t,t+1]\times \{1\}$ and the length of $\widetilde h ([t,t+1]\times \{1\})$ is comparable to $1$, depending on $a$.
\end{proof}

\begin{proposition}\label{corollary:extension_annulus}
Let $\eta$ be a distortion function and let $L,L'>1$ such that $L\in (1,M)$ for some $M>1$. Let $h\colon  \mathbb S(0,1)\cup \mathbb S(0,L) \to \mathbb S(0,1)\cup \mathbb S(0,L') $ be an $\eta$-quasisymmetric homeomorphism that is orien\-tation-preserving and preserves each of the two circles. Then there exists an extension of $h$ a $K(\eta,M)$-quasi\-conformal quasiconformal homeomorphism of $\C$.
\end{proposition}

\begin{proof}
Note that $h$ extends quasiconformally to $\D$ and to $\C\setminus \bar \D(0,L)$ by Corollary \ref{corollary:beurling_ahlfors}. Thus, it suffices to extend it in the annulus $\mathbb A(0;1,L)$. Our goal is to modify the map $h$ in a bi-Lipschitz way so that Proposition \ref{theorem:extension_annulus} can be applied. 

To begin with, we show that $L-1\simeq_{\eta,M} L'-1$. Let $n\in \N$ such that 
$$\frac{M-1}{n+1}\leq  L-1 <  \frac{M-1}{n}.$$
Consider a partition  $\{z_0,\dots,z_n\}$ of $\mathbb S(0,1)$ such that the arc $A_i$ from $z_{i-1}$ to $z_i$ has length $2\pi/(n+1)$ for $i\in \{0,\dots,n\}$, where $z_{-1}\coloneqq z_n$. Thus, for $i\in \{0,\dots,n\}$, we have
$$|z_{i-1}-z_i|\leq \frac{2\pi}{n+1}.$$
Fix $0\in \{1,\dots,n\}$. Denote by $B_i$ be the complementary arc of $A_i$ and note that $\diam_e(B_i)=2=\diam_e(\mathbb S(0,1))$. Thus, $\diam_e(h(B_i)) \geq \eta(1)^{-1}$ by \eqref{definition:qs}. This implies that $\ell(h(A_i))\leq 2\pi -\eta(1)^{-1}$, so
\begin{align*}
\ell(h(A_i))\leq C_1(\eta)|h(z_{i-1})-h(z_i)|.
\end{align*}
Also,
\begin{align*}
\frac{|h(z_{i-1})-h(z_i)|}{L'-1}&=\frac{|h(z_{i-1})-h(z_i)|}{|h(z_{i})- L'h(z_i)|}\\
&\leq \eta\left( \frac{|z_{i-1}-z_i|}{|z_i- h^{-1}(Lh(z_i))|}\right)\leq \eta \left(\frac{ 2\pi(n+1)^{-1}}{L-1}\right)\leq \eta\left(\frac{2\pi}{M-1}\right). 
\end{align*}
Therefore, $\ell(h(A_i))\leq C_2(\eta,M) (L'-1)$. We sum over $i\in \{0,\dots,n\}$ and we obtain
$$2\pi \leq C_2(\eta,M) \cdot (n+1) (L'-1)\leq C_2(\eta,M) \cdot 2n  (L'-1) \leq C_2(\eta,M) 2(M-1) \frac{L'-1}{L-1}.$$
If we show that $L'\lesssim_{\eta,M}1$, then we can apply the same argument to $h^{-1}$ in order to obtain an inequality in the reverse direction. Observe that by applying \eqref{definition:qs} to $A=\mathbb S(0,1)$ and $B=\mathbb S(0,1)\cup \mathbb S(0,L)$, we have
$$\frac{2}{2L'} \geq \frac{1}{2\eta\left(\frac{2L}{2}\right)}> \frac{1}{2\eta(M)}.$$ 
Therefore, $L'< 2\eta(M)\eqqcolon M'$, as desired. This completes the proof of the claim that $L-1\simeq_{\eta,M} L'-1$.

Next, we modify the map $h$ in a bi-Lipschitz way. Let $N\in \N$, $N\geq 2$, such that $\frac{M-1}{N}\leq  L-1<\frac{M-1}{N-1}$, so $L-1\simeq_M\frac{1}{N}$. Note that $\frac{1}{N}\simeq e^{2\pi/N}-1$, so  
$$L-1\simeq_{M} e^{2\pi/N}-1\quad \text{and}\quad L'-1\simeq_{\eta,M}e^{2\pi/N}-1.$$
Consider the maps
\begin{align*}
\sigma(re^{i\theta}) &= \bigg(1+ \frac{e^{2\pi/N}-1}{L-1}(r-1)\bigg) e^{i\theta}\,\,\, \text{and} \,\,\, \tau(re^{i\theta}) &= \bigg(1+ \frac{e^{2\pi/N}-1}{L'-1}(r-1) \bigg) e^{i\theta}.
\end{align*}
The map $\sigma$ (resp.\ $\tau$) interpolates between the identity map on $\mathbb S(0,1)$ and a scaling on $\mathbb S(0,L)$ (resp.\ $\mathbb S(0,L')$). The map $\sigma$ is  bi-Lipschitz on $\bar{ \mathbb A}(0;1,L)$, depending only on $M$, and maps it onto $\bar{\mathbb A}(0;1,e^{2\pi/N})$. The map $\tau$ is bi-Lipschitz on $\bar{ \mathbb A}(0;1,L')$, depending only on $\eta,M$, and maps it onto $\bar{\mathbb A}(0;1,e^{2\pi/N})$. If we can extend the map $\tau \circ h\circ \sigma^{-1}$ to a $K(\eta,M)$-quasiconformal map of the annulus $\mathbb A(0;1,e^{2\pi/N})$, then we can also extend the original map $h$ to a $K'(\eta,M)$-quasiconformal map from $\mathbb A(0;1,L)$ onto $\mathbb A(0;1,L')$. Therefore, by replacing the map $h$ with $\tau \circ h\circ \sigma^{-1}$, we may assume that $L=L'=e^{2\pi/N}\in (1,e^{\pi}]$ and $h$ is $\eta$-quasisymmetric.

We will verify the assumptions of Proposition \ref{theorem:extension_annulus}. Consider the curvilinear square $S(t)=\{z\in \bar {\mathbb A}(0;1,L): t\log L\leq   \arg(z)\leq (t+1)\log L\}$. It suffices to show that its vertices are mapped to points with mutual distances in the interval $[a^{-1}\log L, a\log L]$ for some $a=a(\eta)\geq 1$.

Let $z,w\in \mathbb S(0,1)$ be two vertices of $S(t)$. The other two vertices are $Lz$ and $Lw$. We will bound the distances $|h(z)-h(w)|$, $|h(z)-h(Lz)|$, $|h(z)-h(Lw)|$, $|h(Lz)-h(Lw)|$. Then the bounds for $|h(w)-h(Lw)|$ and $|h(w)-h(Lz)|$ follow by symmetry. 

Note that the arc of $S(t)$ between $z$ and $w$ has length $\log L\leq \pi$. Thus, $ \frac{2}{\pi}\log L \leq |z-w|\leq \log L$. Moreover, since $L\in (1,e^{\pi}]$, we have $c_0^{-1}(L-1)\leq \log L \leq L-1$ for $c_0=\pi^{-1}(e^\pi-1)$. Therefore, we have
$$\frac{L-1}{|h(z)-h(w)|}\leq \frac{|h(Lz)-h(z)|}{|h(z)-h(w)|}\leq \eta\left( \frac{|Lz-z|}{|z-w|}\right) \leq \eta\left(\frac{L-1}{2\pi^{-1}\log L}\right)\leq \eta\left(\frac{c_0\pi}{2}\right) $$ 
and
\begin{align*}
\frac{|h(z)-h(w)|}{L-1}&=\frac{|h(z)-h(w)|}{|h(z)- Lh(z)|}\leq \eta\left( \frac{|z-w|}{|z- h^{-1}(Lh(z))|}\right)\leq \eta \left(\frac{\log L}{L-1}\right)\leq \eta(1). 
\end{align*}
Thus $L-1\leq |h(z)-h(Lz)|\leq \eta(\frac{c_0\pi}{2})\eta(1)(L-1)$ and $\eta(\frac{c_0\pi}{2})^{-1}(L-1)\leq  |h(z)-h(w)|\leq \eta(1)(L-1)$.  These bounds give immediately bounds for $|h(z)-h(Lw)|$.  Similarly, 
\begin{align*}
\frac{|h(Lz)-h(Lw)|}{L-1}& = \frac{|h(Lz)-h(Lw)|}{|h(z)-Lh(z)|} \leq \eta \left(\frac{L|z-w|}{L-1}\right)\leq \eta(L)\leq \eta(e^\pi)
\end{align*}
and
\begin{align*}
\frac{L-1}{|h(Lz)-h(Lw)|}\leq \frac{|h(Lz)-h(z)|}{|h(Lz)-h(Lw)|} \leq \eta\left(\frac{L-1}{2\pi^{-1}L\log L}\right)\leq \eta\left(\frac{c_0\pi}{2L}\right)\leq \eta\left(\frac{c_0\pi}{2}\right).
\end{align*}
So, $\eta(\frac{c_0\pi}{2})^{-1}(L-1)\leq |h(Lz)-h(Lw)|\leq \eta(e^\pi)(L-1)$. This completes the proof.
\end{proof}

\begin{proposition}\label{corollary:extension_annulus:large}
Let $\eta$ be a distortion function and let $L,L'>1$ such that
\begin{align}\label{corollary:extension_annulus:large:assumption}
c_0^{-1}\leq  \frac{\log L'}{\log L}\leq c_0
\end{align}
for some $c_0\geq 1$. Let $h\colon  \mathbb S(0,1)\cup \mathbb S(0,L) \to \mathbb S(0,1)\cup \mathbb S(0,L') $ be an $\eta$-quasisymmetric homeomorphism that is orien\-tation-preserving and preserves each of the two circles. Then there exists an extension of $h$ to a $K(\eta,c_0)$-quasi\-conformal homeomorphism of $\C$.
\end{proposition}

\begin{proof}
By Proposition \ref{corollary:extension_annulus}, we may assume that $L\geq 2$. Note that $h$ extends quasiconformally to $\D$ and to $\C\setminus \bar \D(0,L)$ by Corollary \ref{corollary:beurling_ahlfors}. Thus, it suffices to extend it in the annulus $\mathbb A(0;1,L)$. 

Let $z\in \mathbb S(0,1)$ and $w\in \mathbb S(0,L)$ such that $|h(z)-h(w)|=L'-1$. By \eqref{definition:qs}, we have
$$\frac{L'-1}{2L'}=\frac{|h(z)-h(w)|}{2L'}\geq \frac{1}{2\eta\left(\frac{2L}{|z-w|}\right)} \geq \frac{1}{2\eta\left(\frac{2L}{L-1}\right )} \geq \frac{1}{2\eta(4)}.$$ 
Therefore, $L'\geq L_0'\coloneqq (1-\eta(4)^{-1})^{-1}>1$. Let  $R=\frac{1}{2}(1+\min\{2, L_0'\}^{1/2})$, so 
\begin{align}\label{corollary:extension_annulus:large:1}
1<R^2<(2R-1)^2=\min\{2,L_0'\}\leq \min\{L,L'\}.
\end{align}
Note that $1<R<L/R<L$ and $1<R<L'/R<L'$. We split the annulus $\mathbb A(0;1,L)$ into the annuli $A_1=\mathbb A(0;1,R)$, $A_2=\mathbb A(0; R, L/R)$, and $A_3=\mathbb A(0;L/R,L)$. Similarly, we partition $\mathbb A(0;1,L')$ into the annuli $A_1'=A_1$, $A_2'=\mathbb A(0;R,L'/R)$, $A_3'=\mathbb A(0;L'/R,L')$. The map that is equal to $h$ on $\mathbb S(0,1)$ and equal to the identity map on $\mathbb S(0,R)$ is trivially quasisymmetric, depending only on $\eta$. Therefore, by Proposition \ref{corollary:extension_annulus}, there exists an extension of $h|_{\mathbb S(0,1)}$ to a $K(\eta)$-quasiconformal homeomorphism $H$ of the annulus $A_1$ that is the identity map in the outer circle. Similarly, there exists an extension of $h|_{\mathbb S(0,L)}$ to a $K(\eta)$-quasiconformal homeomorphism $H\colon A_3\to A_3'$ that is a scaling $z \mapsto L'L^{-1}z$ in the inner circle. 

It remains extend $H$ to a quasiconformal map from $A_2$ onto $A_2'$. Equivalently, if we compose with appropriate scalings, it suffices to extend quasiconformally the map $g\colon \mathbb S(0,1)\cup \mathbb S(0,L/R^2)\to \mathbb S(0,1)\cup \mathbb S(0,L'/R^2)$, which is equal to the identity in the inner circle and it is equal to a scaling in the outer circle. Let $\beta>0$ such that  $(R^{-2} L)^\beta =R^{-2} L'$, so $L'=R^{-2(\beta-1)} L^\beta$. We show that $c_1^{-1}\leq \beta\leq c_1$ for some constant $c_1=c_1(\eta,c_0)\geq 1$. Assuming this, note that the map $re^{i\theta} \mapsto r^\beta e^{i\theta}$ is $\max\{\beta,\beta^{-1}\}$-quasiconformal, so it is $c_1$-quasiconformal and has the desired boundary behavior in the annulus $\mathbb A(0;1,L/R^2)$. 

We now prove the claim. By the assumption \eqref{corollary:extension_annulus:large:assumption}, we have
$$c_0^{-1}\leq \frac{\beta\log L-2(\beta-1)\log R}{\log L}\leq c_0.$$
Equivalently, 
\begin{align}\label{corollary:extension_annulus:large:2}
c_0^{-1} \leq \beta\left( 1 -\lambda \right)+\lambda \leq c_0,
\end{align}
where $\lambda=2\frac{\log R}{\log L}$. By \eqref{corollary:extension_annulus:large:1}, we have  
$$0< \lambda < \lambda_0\coloneqq \frac{\log R}{\log (2R-1)}<1.$$
By \eqref{corollary:extension_annulus:large:2}, we conclude that $\beta\leq c_0(1-\lambda)^{-1}\leq c_0(1-\lambda_0)^{-1}$. If $\lambda<\frac{1}{2}c_0^{-1}$, then \eqref{corollary:extension_annulus:large:2} gives $\beta\geq c_0^{-1}-\lambda\geq \frac{1}{2}c_0^{-1}$.  Finally, if $\lambda\geq \frac{1}{2}c_0^{-1}$, then $L\leq R^{4c_0}$. Thus, by \eqref{corollary:extension_annulus:large:1},
$$R^{(4c_0-2)\beta} \geq (R^{-2}L)^\beta = R^{-2}L'\geq  \left(\frac{2R-1}{R}\right)^2.$$
Thus,
$$\beta \geq (4c_0-2)^{-1} (\log R)^{-1}2\log\left(\frac{2R-1}{R}\right)$$
and the latter is a positive number depending only on $\eta$ and $c_0$. 
This completes the proof. 
\end{proof}

\section{Extension of embeddings of pairs of circles}\label{section:extension_disks}

In this section we prove Theorems \ref{theorem:intro:extension} and \ref{theorem:extension:disjoint}. First we consider the case that $U$ and $V$ are tangent disks as in Theorem \ref{theorem:intro:extension}. We establish a preliminary result.

\begin{lemma}\label{lemma:strip_bounds}
Let $f\colon \R\times \{0,1\}\to \C$ be an $\eta$-quasisymmetric embedding for some distortion function $\eta$. Suppose that $f|_{\R\times \{0\}}$ is the identity map and $f(\R\times \{1\})\subset \UHP$. Then $f(\R\times \{1\})\subset \R\times [\eta(1)^{-1},\eta(1)]$.
\end{lemma}

\begin{proof}
Let $x\in \R$ and $f(x,1)=(z,w)\in \UHP$. Then 
$$w\leq |f(x,1)-(x,0)|=|f(x,1)-f(x,0)|=  \frac{|f(x,1)-f(x,0)|}{|f(x,0)-f(x+1,0)|} \leq \eta(1).$$
For an inequality in the reverse direction we have
\begin{align*}
w=|f(x,1)-f(z,0)|&=  \frac{|f(x,1)-f(z,0)|}{|f(z,0)-f(z+1,0)|} \\
&\geq  \eta\left(\frac{|(z,0)-(z+1,0)|}{|(x,1)-(z,0)|} \right)^{-1} \geq \eta(1)^{-1}.
\end{align*} 
This completes the proof.
\end{proof}

\begin{proof}[Proof of Theorem \ref{theorem:intro:extension}]
All claims in the proof are quantitative. Let $U,V$ be disjoint tangent disks in the sphere and $f\colon \partial U\cup \partial V\to \widehat \C$ be a quasi-M\"obius embedding. By using suitable M\"obius transformations we may assume that $V$ is the lower half-plane $\{z\in \C:\im(z)<0\}$, $U$ is the half-plane $\{z\in \C: \im(z)>1\}$ and $f$ fixes $\infty$. Henceforth we use planar topology. Therefore, $f\colon \R\times \{0,1\}\to \C$ is a quasi-M\"obius embedding such that $f(z)\to \infty$ as $z\to\infty$. By \ref{q:qm_qs}, $f$ is quasisymmetric. By Corollary \ref{corollary:beurling_ahlfors_embedding}, the map $f|_{\R\times \{0\}}$ has an extension to a quasiconformal, and thus quasisymmetric, homeomorphism $F$ of $\C$. Thus, the map $g=F^{-1}\circ f\colon \R\times \{0,1\}\to \C$ is quasisymmetric and is the identity map on $\R\times \{0\}$. By postcomposing with a reflection, we may assume that $g$ maps $\R\times \{1\}$ into the upper half-plane. 

By Lemma \ref{lemma:strip_bounds}, $g(\R\times \{1\})$ is contained in a strip $\R\times [a^{-1},a]$. By Lemma \ref{lemma:quasicircle_parallel} we conclude that the condition of Theorem \ref{theorem:characterization_tangent} is satisfied. Hence, there exists a quasiconformal homeomorphism $G$ of $\C$ that maps the lower half-plane $V$ onto itself and $g(U)$ onto the half-plane $\{z\in \C: \im(z)>1\}$. We consider the map $h=G\circ g$. The map $h$ is a quasisymmetric homeomorphism from $\R\times \{0,1\}$ onto itself that preserves each of the two lines and is increasing on $\R\times \{0\}$.  By Corollary \ref{corollary:extension_strip}, $h$ has a quasiconformal extension in $\C$. With appropriate compositions this extension gives the desired extension of the original map $f$. 
\end{proof}

Next we treat the case of disks with disjoint closures as in Theorem \ref{theorem:extension:disjoint}. We include a preliminary result in the spirit of Lemma \ref{lemma:strip_bounds}.

\begin{lemma}\label{lemma:annuli_identity}
Let $\eta$ be a distortion function and $L>1$. Let $f\colon \mathbb S(0,1)\cup \mathbb S(0,L)  \to \C$ be an $\eta$-quasisymmetric embedding such that $f|_{\mathbb S(0,1)}$ is the identity map and $f(\mathbb S(0,L))\subset \D^*$ . Then there exists $L'>1$ and $K=K(\eta)>0$ such that for each $z\in \mathbb S(0,L)$ we have
$$ L'\leq |f(z)|\leq  L'+K(L'-1).$$ 
\end{lemma}

\begin{proof}
Let $z_1,z_2\in \mathbb S(0,L)$ such that $|f(z_1)|=\sup\{|f(z)|:z\in \mathbb S(0,L)\}\eqqcolon R$ and $|f(z_2)|= \inf\{|f(z)|:z\in \mathbb S(0,L)\}\eqqcolon r$. Note that $r>1$ because $f(\mathbb S(0,L))\subset \D^*$. Suppose that $L\geq 2$. Then 
$$\frac{R-1}{r+1}\leq \frac{|f(z_1)-f(1)|}{|f(z_2)-f(1)|}\leq \eta\left(\frac{L+1}{L-1}\right)\leq \eta(3).$$
Thus, 
\begin{align}\label{lemma:annuli_identity:1}
R\leq 1+\eta(3)(r+1)\leq (1+2\eta(3))r.
\end{align}
For $w=f(z_2)/|f(z_2)|$ we have $|z_2-w|\geq L-1\geq 1$. Let $u\in \mathbb S(0,1)$ such that $|u-w|=1$. Then  
$$r-1=|f(z_2)-w|=\frac{|f(z_2)-f(w)|}{|f(w)-f(u)|} \geq {\eta\left(\frac{1}{|z_2-w|}\right)}^{-1} \geq {\eta(1)}^{-1}.$$ 
We conclude that $r\geq 1+{\eta(1)}^{-1}$, which is equivalent to 
$$(1+2\eta(3))r\leq r+(\eta(1)+1)2\eta(3)(r-1).$$ 
Combining this with \eqref{lemma:annuli_identity:1}, we conclude that 
$$R\leq (1+2\eta(3))r\leq  r+(\eta(1)+1)2\eta(3)(r-1).$$
The desired conclusion follows for $L'=r$ and $K=(\eta(1)+1)2\eta(3)$.

Next, suppose that $1<L<2$. Let $w\in  \mathbb S(0,1)$ such that $|L^{-1}z_1-w|=L-1=|L^{-1}z_1-z_1|$; such a point exists by the restriction on $L$. Then
$$\frac{R-1}{L-1}\leq   \frac{|f(z_1)-f(L^{-1}z_1)|}{|f(L^{-1}z_1)-f(w)|}\leq \eta\left( \frac{|z_1-L^{-1}z_1|}{|L^{-1}z_1-w|}\right)= \eta(1).$$
This, $R\leq 1+\eta(1)(L-1)$. Next, consider the point $w=f(z_2)/|f(z_2)|$ and let $u\in \mathbb S(0,1)$ such that $|w-u|=L-1$. We have
$$\frac{r-1}{L-1}=\frac{\left|f(z_2)- f(w)\right|}{|f(w)-f(u)|} \geq  {\eta\left(\frac{|w-u|}{|z_2-w|} \right)}^{-1}\geq {\eta(1)}^{-1}.$$
Thus, $r\geq 1+\eta(1)^{-1}(L-1)$, which is equivalent to $$1+\eta(1)(L-1)\leq r+ (\eta(1)^2-1)(r-1).$$
Thus, 
$$  R\leq 1+\eta(1)(L-1)\leq r+(\eta(1)^2-1)(r-1).$$
The claim follows for $L'=r$ and $K=\eta(1)^2-1$. 
\end{proof}

\begin{proof}[Proof of Theorem \ref{theorem:extension:disjoint}]
By applying suitable M\"o\-bius trans\-formations, we assume that $V=\D$, $U^*=\D(0,L)$ for some $L>1$, $f(\partial V),f(\partial U)$ are Jordan curves in $\C$, $f(\partial U)$ lies in the exterior of $f(\partial V)$, and the maps $f|_{\partial V}$, $f|_{\partial U}$ are orientation-preserving. The curve $f(\partial U)$ is a quasicircle; see \eqref{quasicircle:crossratio} and the preceding discussion.

We perform a further normalization. Consider three points $z_1,z_2,z_3\in \partial U$ with mutual distances at least $L$. We consider a quasiconformal homeomorphism $\phi$ of $\widehat \C$ that preserves the Jordan region bounded by the quasicircle $f(\partial U)$ and the points $\phi(f(z_1))$, $\phi(f(z_2))$, $\phi(f(z_3))$ have mutual distances at least $\diam_e(f(\partial U))/2$; the existence of $\phi$ can be justified by Theorem \ref{prop:quasidisk_quasicircle} and the three-point transitivity of M\"obius transformations of $\D$. The composition $\phi\circ  f$ is quasi-M\"obius and satisfies the normalization in \ref{q:qs_qm}, so it is quasisymmetric. Thus, by replacing $f$ with $\phi\circ f$ we may assume, in addition, that $f$ is quasisymmetric. 

By Corollary \ref{corollary:beurling_ahlfors_embedding}, the map $f|_{\partial V}$ has an extension to a quasiconformal, and thus quasisymmetric,  homeomorphism $F$ of $\C$. The map $g=F^{-1}\circ f\colon \partial  U\cup \partial V\to \C$ is quasisymmetric and is equal to the identity map on $\partial V$. Moreover, $g$ maps $\partial U$ to a Jordan curve in the exterior of $V=\D$. 

By Lemma \ref{lemma:annuli_identity}, there exists $L_1>1$ such that $g(\partial U)$ is contained in the annulus $\{z\in \C: L_1\leq |z|\leq  L_1+K(L_1-1)\}$, where $K$ depends only on the data.  By Corollary \ref{corollary:characterization_quasiannuli_qs}, there exists a quasiconformal homeomorphism $G$ of $\C$ that maps $\D$ onto $\D$ and the interior of the Jordan curve $g(\partial U)$ onto a disk $\D(0,L')$. Note that $G$ is also quasisymmetric.  Consider the map $h=G\circ g\colon  \mathbb S(0,1)\cup \mathbb S(0,L)\to \mathbb S(0,1)\cup\mathbb S(0,L')$, which is quasisymmetric.  

Suppose that condition \ref{theorem:extension:disjoint:1} is true. Since all of the above normalizations are implemented with quasiconformal maps of $\widehat \C$, condition \ref{theorem:extension:disjoint:1} implies that 
$$c_0^{-1}\leq \frac{\log L'}{\log L}\leq c_0$$
for some $c_0$ that depends only on the data. By Proposition \ref{corollary:extension_annulus:large}, the map $h$ extends quasiconformally to $\C$. If instead \ref{theorem:extension:disjoint:2} is true, then we can obtain an extension using Proposition \ref{corollary:extension_annulus}. With appropriate compositions this extension gives the desired extension of the original map $f$. 
\end{proof}

\bibliography{../../biblio} 
\end{document}